\documentclass[reqno]{amsart}
\usepackage{amsmath,amsfonts,amssymb,amsthm}
\usepackage[usenames,dvipsnames]{xcolor}
\usepackage{enumitem}
\usepackage{comment}
\usepackage{graphicx}
\usepackage{overpic}
\usepackage{cancel}
\usepackage[outdir=./]{epstopdf}
\usepackage[colorlinks=true]{hyperref} 
\usepackage{url}
\usepackage[utf8]{inputenc}
\hypersetup{linkcolor=Violet,citecolor=PineGreen}
\newtheorem{teor}{Theorem}[section]
\newtheorem{lema}[teor]{Lemma}
\newtheorem{prop}[teor]{Proposition}
\newtheorem{coro}[teor]{Corollary}
\theoremstyle{definition}
\newtheorem{defi}[teor]{Definition}

\newtheorem{hipo}[teor]{Hypothesis}

\newtheorem{nota}[teor]{Remark}
\newtheorem{rmk}[teor]{Remark}
\newtheorem{notas}[teor]{Remarks}
\numberwithin{equation}{section}

\newcommand{\R}{\mathbb R}

\newcommand{\Z}{\mathbb{Z}}
\newcommand{\Q}{\mathbb{Q}}

\newcommand{\N}{\mathbb{N}}

\newcommand{\mB}{\mathcal{B}}

\newcommand{\mF}{\mathcal{F}}

\newcommand{\mR}{\mathcal{R}}

\newcommand{\mU}{\mathcal{U}}

\newcommand{\ep}{\varepsilon}

\newcommand{\mI}{\mathcal{I}}

\newcommand{\W}{\Omega}

\newcommand{\lb}{\lambda}

\newcommand{\wma}{\wit{\mathfrak{a}}}
\newcommand{\wmr}{\wit{\mathfrak{r}}}

\newcommand{\ma}{\mathfrak{a}}
\newcommand{\mb}{\mathfrak{b}}
\newcommand{\mr}{\mathfrak{r}}

\newcommand{\mah}{\mathfrak{a}_h}

\newcommand{\mrh}{\mathfrak{r}_h}
\newcommand{\mach}{\mathfrak{a}_{c,h}}

\newcommand{\mrch}{\mathfrak{r}_{c,h}}

\newcommand{\G}{\Gamma}

\newcommand{\wit}{\widetilde}

\newcommand{\n}[1]{\left\|#1\right\|}
\newcommand{\Frac}[2]{\displaystyle\frac{#1}{#2}}

\newcommand{\lsm}{\left[\begin{smallmatrix}}
\newcommand{\rsm}{\end{smallmatrix}\right]}

\newcommand\note[1]{\marginpar{\flushleft\sffamily\tiny\textcolor{red}{#1}}}
\begin{document}
\title[Rigorous estimates for critical transitions]
{Critical transitions for scalar nonautonomous systems with concave
nonlinearities: some rigorous estimates}
\author[I.P.~Longo]{Iacopo P.~Longo}
\author[C.~N\'{u}\~{n}ez]{Carmen N\'{u}\~{n}ez}
\author[R.~Obaya]{Rafael Obaya}
\address[I.P.~Longo]{Technische Universit\"{a}t M\"{u}nchen,
Forschungseinheit Dynamics, Zentrum Mathematik, M8,
Boltzmannstra{\ss}e 3, 85748 Garching bei M\"{u}nchen, Germany.}
\address[C. N\'{u}\~{n}ez and R. Obaya]{Universidad de Valladolid,
Departamento de Matem\'{a}tica Aplicada,
EII, Pso. Prado de la Magdalena 3-5,
47011 Valladolid, Spain.}
\email[Iacopo Longo]{longoi@ma.tum.de}
\email[Carmen N\'{u}\~{n}ez]{carmen.nunez@uva.es}
\email[Rafael Obaya]{rafael.obaya@uva.es}
\thanks{
All the authors were partly supported by Ministerio de Ciencia, Innovaci\'{o}n y Universidades (Spain)
under project PID2021-125446NB-I00 and by Universidad de Valladolid under project PIP-TCESC-2020.
I.P.~Longo was also partly supported by the European Union's
Horizon 2020 research and innovation programme under the Marie Skłodowska-Curie
grant agreement No 754462, by the European Union’s Horizon 2020 -- Societal Challenges
grant agreement No 820970, and by TUM International Graduate School of Science
and Engineering (IGSSE)}
\date{}
\begin{abstract}
The global dynamics of a nonautonomous Carath\'{e}odory scalar ordinary differential equation $x'=f(t,x)$,
given by a function $f$ which is concave in $x$, is determined by the existence or absence of
an attractor-repeller pair of hyperbolic solutions. This property, here extended to a very
general setting, is the key point to classify the dynamics of an equation which is a transition
between two nonautonomous asypmtotic limiting
equations, both with an attractor-repeller pair. The main focus of the paper is to get rigorous
criteria guaranteeing tracking (i.e., connection between the attractors of the past and the
future) or tipping (absence of connection) for the particular case of equations $x'=f(t,x-\G(t))$,
where $\G$ is asymptotically constant. Some computer simulations show the accuracy of the
obtained estimates, which provide a powerful way to determine the occurrence of critical
transitions without relying on a numerical approximation of the (always existing)
locally pullback attractor.
\end{abstract}
\maketitle

\section{Introduction}
Tipping points, also called critical transitions, are sudden and large changes
of a system's dynamics as a consequence of small changes in its input.
Under this headline, we find phenomena like disruption of
climate~\cite{boers2021,lenton2008, lenton2019}, of ecological
environments~\cite{okwi,osullivan2022rate}, and of financial
markets~\cite{may2008,yukalov2009}, among others. From a mathematical standpoint,
the understanding of certain critical transitions is still limited. Namely,
when  a time-dependent variation of parameters intervenes~\cite{kaszas2019}.
This is the case of rate-induced tipping~\cite{aspw,awvc,kulo,lno,lnor,walc},
size-induced tipping~\cite{lno,duno3}, and phase-induced tipping \cite{alas,altw,duno3,lno}.
\par
The setting of these problems is usually the same: an asymptotically constant time-dependent
and known variation of a parameter takes place between a so-called past limit-problem
(obtained as $t\to-\infty)$ and a future limit-problem (obtained as $t\to\infty)$;
the past and the future limit-problems are also assumed to be known and at least a
local attractor exists for each one of them. Depending on the nature of the time-dependent
variation of parameters (e.g., its rate, phase and/or size) the ``connection" between
the attractors of the past and of the future can vary considerably or even break up.
Every such qualitative change of dynamics is associated with a tipping point.
Their analysis in asymptotically autonomous systems has been carried out with
several mathematical techniques including compactification arguments
\cite{wixt,wieczoreck2021rate}, statistical early-warning signals \cite{risi1,risi2},
and the tipping probability of physical measures \cite{ashwin2021physical}.
Numerous are also the instances of real models which have been studied,
for example in climate \cite{bastiaansen2022climate},
ecology \cite{osullivan2022rate,walc}, and population dynamics~\cite{vahf}.

Recently, an interpretation of these phenomena through nonautonomous bifurcation
theory has arisen \cite{lno,lnor}. In this new formulation, all the terms involved
(the past and the future systems and the transition equations between them)
are nonautonomous.
There are two promising advantages in this new theory: on the one hand, this approach provides
a unified framework for both asymptotically autonomous and nonautonomous systems; and,
on the other hand, the assumptions required on the time-dependent variation of parameters,
as well as the time-dependent forcing, are particularly mild compared to the rest of the
literature. In the wake of such theory, this paper deals with the bifurcations occurring
in coercive concave scalar nonautonomous ordinary differential equations. Although this class of
problems might seem abstract, (globally or locally) concave or convex models are widely
used in applied sciences \cite{bass,botg,fraed,rens}.
As an example and a motivation, we start our work in
Section \ref{2.sec} by studying the occurrence of rate-induced and size-induced tipping
in simple locally concave models from climate science and neural network theory.

The bulk of our results is theoretical, with potential relevance in the kind of problems
mentioned above. We investigate families of scalar Carath\'{e}odory ordinary differential
equations admitting bounded $m$-bounds and $l$-bounds, which allows the compactification
of the equations in their hulls using a weak topology in the temporal variation of
the vector fields. This type of coefficients and topologies are the natural language
to formulate and solve control theory problems; for the same reasons, they are also
appropriate for formulating and investigating problems related to critical transitions.
In this paper, we show that, if one of these equations is concave, then
the only possible bifurcation is the
nonautonomous saddle-node bifurcation of a pair of hyperbolic solutions, and that
all the tipping points occurring for these equations are in fact bifurcations of such type.
This extends the previous conclusions obtained in \cite{lno,lnor} for quadratic concave
differential equations to this more general framework.

From here, we focus on transition equations of the
form $x'=f(t,x-\G(t))$, where the asymptotic limits $\gamma_\pm$ of $\G$
at $\pm\infty$ are well-defined and finite, and give rise to the past and
future equations. To understand the geometry of the solutions of the
transition equation is fundamental to achieve our main goal of obtaining
calculable rigorous criteria which permit to identify some scenarios
of tracking (with connection of attractors) and tipping (where the connection is lost).
These criteria give a more trustworthy result as opposed to a finite
numerical approximation of the locally pullback attractive solutions, which,
by its own limited nature, can be wrong or misled.
It is important to highlight that the inequalities that our criteria require
are uncoupled: they depend on $\G$ and on $x'=f(t,x)$, but
not on the transition equation itself.

Our criteria are obtained in two steps. In the first one, we deduce a set of inequalities
providing tracking or tipping in the case of a piecewise constant transition function
$\G$. This kind of transitions have physical relevance and have been studied in some previous
references in the literature \cite{altw,hamw,lodi}.
For instance, they appear as limits of
Carath\'{e}odory equations with continuous transitions, when some of the parameters decrease
to $0$ or increase to infinity, providing a natural framework to study models with big or small
size of physical parameters.
Moreover, every continuous transition function $\G$
can be well approximated by a suitable family of piecewise continuous functions.
This property combined with the fact that some of the previously obtained
inequalities hold uniformly for these piecewise continuous functions approximating
$\G$ allows us to obtain rigorous criteria for tracking or tipping also for
continuous transition functions, which is the second step.

Let us delve with a bit more details into the main results. After the already mentioned
Section \ref{2.sec}, we present general results for coercive concave
Carath\'{e}odory ordinary differential equations $x'=f(t,x)$ in Section \ref{3.sec}. Theorem \ref{3.teoruno}
characterizes the sets of solutions bounded in the past and/or in the future,
respectively bounded from above by a certain solution $a$ (defined
at least on a negative half-line)
and from below by a certain solution $r$ (defined
at least on a positive half-line).
Theorem \ref{3.teorhyp} shows that $a$ and $r$ are globally defined and uniformly
separated if and only if they are hyperbolic. This situation is referred to as the {\em existence
of an attractor-repeller pair $(a,r)$}. The crucial property of persistence of hyperbolic solutions
under small perturbations, a well-known property for more regular
systems, is proved in Proposition \ref{3.proppersiste} for the specific class of
Carath\'{e}odory differential equations used in our work.
Theorem \ref{3.teorlb*} describes a
nonautonomous global saddle-node bifurcation pattern, which is
fundamental for the rest of the paper, and completes this section.

From this point, we consider a transition equation $x’=f(t,x-\G(t))$,
where $f$ is a concave and coercive Carath\'{e}odory function with smooth variation
in the second component, and $\G$ is an essentially bounded function
with finite asymptotic limits $\gamma_{\pm}=\lim_{t\to\pm\infty}\G(t)$.
Our main hypothesis, also in force for the rest of the paper,
is the existence of an attractor repeller pair $(a,r)$ for $x'=f(t,x)$,
which implies this same property for the past and future equations $x'=f(t,x-\gamma_\pm)$.
In Section \ref{4.sec},
we prove that the special solutions $a_\G$ and $r_\G$ provided by
Theorem \ref{3.teoruno} have additional properties: $a_\G$ is locally pullback attractive
and connects with the attractor of the past as time decreases, and $r_\G$ is locally pullback
repulsive and connects with the repeller of the future as time increases.
In addition, we classify in three cases the internal dynamics of the transition equation
in terms of the domains and the relative position of both solutions.
This classification
combined with the previous bifurcation result shows that the collision of the
attractor and the repeller is the unique mechanism leading to a critical transition
for this type of equations.

For the rest of the paper, $\G$ is also assumed to be continuous.
A crucial role in our analysis will be played by some special maps:
the concavity of $f$ ensures that, for any $\nu\in(0,1)$,
the convex combination $b^\nu=\nu a+(1-\nu)r$ of the attractor $a$ and the repeller $r$ is a strict
lower solution of $x'=f(t,x)$.
For each $h>0$, we define $\G^h$ as the piecewise constant function which coincides
with $\G(t)$ at $t=jh$ for all $j\in\Z$.
In Section \ref{5.sec}, we analyze the transition equation $x’=f(t,x-\G^h(t))$. Let
us denote $a_h=a_{\G^h}$ and $r_h=r_{\G^h}$.
The rigorous tracking and tipping criteria which we obtain in
this section are based on the construction of a {\em tracking route} for the graph of
$a_h$: the set $\mathcal R:=\bigcup_{m=1,\ldots n}\{(t,x)\,|\;t\in[(j_0+m-1)h,(j_0+m)h],
\,x\ge b^{\nu_m}(t)+\G((j_0+m)t)\}$.
Theorem \ref{5.teortrack} provides suitable
indices $j_0,n\in\Z$, parameters $\nu_1,\ldots,\nu_m\in[0,1]$ and
a list of inequalities guaranteeing that the graph of $a_h$ remains in
$\mathcal R$ for $t\in[j_0h,(j_0+n)h]$,
and that $a_h((j_0+n)h)>r_h((j_0+n)h)$: this inequality ensures tracking.
As pointed out before, the choice of the indices and parameters, and the
inequalities to be satisfied, depend on $\G$ and on the dynamics of
$x'=f(t,x)$, but not on that of $x'=f(t,x-\G(t))$.
In the case of a nondecreasing $\G$, Theorem \ref{5.teortipping}
provides new lists of indices, parameters and inequalities (depending only
on $\G$ and on $x'=f(t,x)$) which ensure that either $a_h$ blows up before
$t=(j_0+n)h$ or its graph on $[j_0h,(j_0+n)h]$ is strictly below $\mathcal R$,
with $a_h((j_0+n)h)<r_h((j_0+n)h)$: this inequality ensures tipping.
Some easy applications of these results are also included in this section.

Section \ref{6.sec} deals with $x'=f(t,x-\G(t))$. Although the theory
admits a general version, for the sake of simplicity we reduce the analysis
to the quadratic function $f(t,x)=-x^2+p(t)$.
Equations of this type are analyzed, for instance, in
\cite{aspw,awvc,risi1}. We point out that this
transition equation can be rewritten as $x'=-x(x-2\G(t))+q(t)$
for $q=-\G^2+p$: logistic nonautonomous models for population dynamics with migration
term are included in our analysis.
The unique (saddle-node) bifurcation value for
$x'=-(x-\G(t))^2+p(t)+\lb$ is the key point to provide
rigorous conditions of tracking and total tracking on the hull of $p$; and,
if $p$ is almost periodic, we describe simple situations with infinite tipping
points caused by a phase shift on $p$.
These are the main points of Theorem \ref{6.teorconthull}.
Then, coming back to the main objective of the paper,
we check that, for adequate choices of indices and parameters, the
results obtained in the previous section remain uniformly valid for the piecewise
transitions $\G^h$ with $h>0$ small, and use this fact combined with
the uniform convergence of $\G^h$ to $\G$ as $h\downarrow 0$ to
provide rigorous criteria of tracking or tipping for the continuous transition.
These criteria are given in Theorems \ref{6.teortrack1}, \ref{6.teortrack2}
and \ref{6.teortip1}. Once more, the inequalities involved in the criteria depend
only on $\G$ and on $x'=-x^2+p(t)$. Section \ref{6.sec} is completed with
the particular case of a simple piecewise linear function $\G$.
We show that a simple application of the previous criteria
allows us to distinguish the dynamical case of the
transition equation in many cases,
even for all the values of the rate and of the phase of the transition.

Some numerical examples intended to show the scope of our conclusions
are included in Sections \ref{5.sec} and \ref{6.sec}. In these examples, a big percentage
of the cases of tracking or tipping satisfy some of the criteria of our theorems.
To some extent,
it seems correct to say that our sufficient conditions are a potent
tool to determine the type of global dynamics in a high proportion of cases.

The paper ends with an Appendix on the compactification of the mild
Cara\-th\'{e}odory differential equation by means of the construction of the hull
of $f$, and on the description of the main dynamical properties of
the local skewproduct flow defined by the solutions of the equations
given by the elements of the hull.
These properties play an essential role in the development of the theory of this paper.

\section{Examples of tipping in nonautonomous concave or convex models}\label{2.sec}
Concave (or respectively convex) dynamics is the main focus of this paper. Although an assumption of concavity might seem strong, models which are at least locally concave (or convex) are greatly important and wide-spread in applications. As an example and a motivation for the subsequent theoretical analysis, we hereby include a short presentation of some locally concave or locally convex models which are susceptible to the tipping phenomena studied in the rest of the paper. The first example deals with three-parametric qualitative models of climate undergoing a rate-induced tipping, phase-induced and size-induced tipping.
The second one shows a tipping point in a synchronized Hopfield neural network
due to a change in the size of the one-parametric transition.
\subsection{Tipping points for Fraedrich's zero-dimensional climate models}
In this section, we present three nonautonomous zero-dimensional models of the
global climate obtained from the respective autonomous versions presented by Klaus
Fraedrich in 1978~\cite{fraed}. The underlying physical ground is given by an energy
conservation law after integrating over the total mass (per unit area) of the system.
Specifically, it is assumed that the average surface temperature $T(t)$ of a spherical
planet subjected to radiative heating changes depending on the balance between incoming
solar radiation $R\downarrow$ and outgoing emission $R\uparrow$:
\begin{equation}\label{eq:climate_cons_law}
 \tau\,T'=R\downarrow-R\uparrow.
\end{equation}
The temperature is measured in Kelvin degrees and therefore the model is considered for $T>0$
(i.e., above the absolute zero), and the time is measured in seconds.
The constant $\tau$
is the thermal inertia of a well-mixed ocean of depth 30 meters,
covering 70.8\% of the planet (see \eqref{2.cons} for its value and those of the remaining
constants). We model the incoming solar radiation by
\[
 R\!\downarrow\,= \frac{1}{4}\,\mu(t)\,I_0\,(1-\alpha)\,,
\]
where $I_0$
is the average total solar irradiance at Earth and $\alpha$
is the albedo of the planet, i.e.,~the fraction of solar radiation which is reflected
from the surface of Earth outside the atmosphere (for example due to ice, deserts or clouds).
The positive function $\mu(t)$ allows for variations of the total solar irradiance.
In contrast to Fraedrich, who takes $\mu=1$ for his first models,
we consider a periodic coefficient $\mu(t)$ which simulates the following astronomical
phenomena: the 11-year cycle of solar activity (also known as the Schwabe cycle)
with amplitude $0.05\%$ of $I_0$, and the annual
revolution of Earth along its elliptic orbit with amplitude $0.02\%$ of $I_0$
(see \cite{Hathaway2015}).
In other words, we consider
\begin{equation}\label{2.defmu}
 \mu(t)=1+0.0005\,\sin\!\left(\frac{2\pi t}{11\,\kappa}\right)+0.0002\,\sin\left(\frac{2\pi t}{\kappa}\right),
\end{equation}
where $\kappa:=60{\cdot}60{\cdot}24{\cdot}365.25$ is the average number of seconds per year.
Several other time-dependent components on time scales that vary from minutes till
tens or hundreds of millennia could be included. We avoid it for at least two reasons:
the most relevant time-scale for our subsequent tipping analysis is the one of decades;
and our aim is to analyze the different causes for the occurrence of
a tipping point in the asymptotically time-dependent case, not to provide a definitive
quantitative estimate of the current or future climate on Earth.
\par
The outgoing radiation $R\uparrow$ is modelled via the Stefan-Boltzmann law, i.e.,~proportional to a blackbody surface emission
\[
R\!\uparrow\,=\varepsilon_{sa}\,\sigma\,T^4\,,
\]
where $\varepsilon_{sa}=\varepsilon_{s}-\varepsilon_{a}$
stands for the effective emissivity  obtained as the difference between
the surface emissivity and the atmospheric emittance, and $\sigma$ is the Stefan–Boltzmann constant.
With the previous definitions, \eqref{eq:climate_cons_law} reads as the following zero-dimensional model
for the planet's average temperature,
\begin{equation}\label{eq:Fraed0}
T'=\frac{1}{\tau}\,\left(-\varepsilon_{sa}\,\sigma\,T^4+\frac{1}{4}\,\mu(t)\,I_0\,(1-\alpha)\right).
\end{equation}
A first refinement of this model is given by adding an albedo feedback law. Indeed, the albedo of a
planet can be related to the average temperature via the amount of ice on the planet's surface.
Fraedrich \cite{fraed} employs and motivates the use of a quadratic law $\alpha=a -b\,T^2$,
which yields
\begin{equation}\label{eq:Fraed1}
 T'=\frac{1}{\tau}\,\left(-\varepsilon_{sa}\,\sigma\,T^4+\frac{1}{4}\,\mu(t)\,I_0\,b\,T^2+ \frac{1}{4}\,\mu(t)\,I_0\,(1-a)\right)=:f_1(t,T)\,.
\end{equation}
The equilibria, their hyperbolicity and the bifurcation points of the autonomous model \eqref{eq:Fraed1}
(with $\mu(t)\equiv 1$) have been studied in detail by Fraedrich~\cite{fraed}, taking
\begin{equation}\label{2.cons}
\begin{split}
 &\tau:=10^8\,kg\,K^{-1}s^{-2}\,,\quad
 \varepsilon_{sa}:=0.62\,,\quad
 \sigma:=5.6704{\cdot}10^{-8}\,W m^{-2} K^{-4}\,,\\
 &I_0:=1366\, Wm^{-2}\,,\quad
 b:=10^{-5}K^{-2}\,,\quad
 a:=1.2\,.
\end{split}
\end{equation}
For these values, there are always (at least) two hyperbolic equilibria,
one stable $T_s$ and the other unstable $T_u$, which form an attractor-repeller pair for $T>0$.

To simplify the notation in the next paragraphs, we call $I(t):=\mu(t)\,I_0$.

In what follows, we adapt an idea used by Ashwin et al.~\cite{awvc} (in the case $\mu(t)\equiv 1)$
to our nonautonomous setting: we
substitute the parameters $a$ and $b$ by asymptotically autonomous functions of time,
accounting for the anthropogenic emissions (combined greenhouse gasses and industrial aerosols)
in the atmosphere; and we let $a$ vary with respect to a new parameter $d$.
More precisely, once a function $b(t)>0$ is chosen, we take
$a(t)=a_d(t):=1+d\,m(t)$ for $m(t):=I_- b^2(t)/(16\,\ep_{sa}\,\sigma)$,
with $I_-\le I(t)$ for all $t\in\R$, and where $d$ varies in a range which we will specify later.
Then, for each fixed time $s$, we rewrite \eqref{eq:Fraed1} as the differential equations
\begin{equation}\label{eq:Fraed1ts}
 T'=\frac{1}{\tau}\,\left(-\varepsilon_{sa}\,\sigma\,T^4+\frac{1}{4}\,I(t)\,b(s)\,T^2- \frac{1}{4}\,I(t)\,d\,m(s)\right)=:g_s^d(t,T)\,.
\end{equation}
The family $\{T'=g_s^d(t,T)\,|\;s\in\R\}$ plays a fundamental role to analyze
$T'=g_t^d(t,T)$ (i.e., \eqref{eq:Fraed1ts} with $s=t$): this last equation
can be understood as a transition between the past
and future equations, which are obtained by changing $b(s)$ and $m(s)$ by the values of their limits
at $-\infty$ and $+\infty$, and which are nonautonomous due to $I(t)$.
Note that analyzing the occurrence of critical transitions for $T'=g_{c(t+l)}^d(t,T)$
as either $c$, $d$, or $l$ vary, we are analyzing the
occurrence of either rate-induced, size-induced, or phase-induced critical points
due to the transition $(b,a_d)$.

A simple quantitative study of \eqref{eq:Fraed1ts} 
gives immediate qualitative information on its possible dynamics.
Below, we will find a constant $T_1$ such that $g_s^d(t,T)<0$ for all $t,s\in\R$ whenever $T<T_1$. Hence, any constant temperature $T<T_1$ is an upper solution
for \eqref{eq:Fraed1ts} for any $s\in\R$, which means that any solution starting
below such threshold is destined to decrease as time increases.
Furthermore, we will find a second constant $T_2<T_1$ such that, for all $t,s\in\R$,
$g_s^d(t,T)$ is concave for $T>T_2$. In other words, in the region of the phase space where the dynamics is not already clear, the problem is concave and therefore can be completely investigated using the theory developed in the rest of the paper.

It is easy to check that, whenever $0<d<1$, the map $T\mapsto g_s^d(t,T)$ (a fourth-degree polynomial)
has two different positive roots for each $(s,t)$, namely
\[
 T_s^\pm(t)=\left(\frac{b(s)\sqrt{I(t)}}{8\,\ep_{sa}\,\sigma}\,
 \left(\sqrt{I(t)}\pm\sqrt{I(t)-d\,I_-}\right)\right)^{1/2}\,.
\]
The temperatures $T_1$ and $T_2$ before described can be taken as
\[
 T_1:=\left(\frac{b_+I_+}{8\,\ep_{sa}\,\sigma}\right)^{1/2}
 \qquad\text{and}\qquad T_2:=\left(\frac{b_+I_+}{24\,\ep_{sa}\,\sigma}\right)^{1/2}\,,
\]
where $I_+\ge I(t)$ for all $t\in\R$ and $b_-\le b(s)\le b_+$ for all $s\in\R$:
$T_1$ satisfies $T_1\ge \sup_{t,s\in\R}T_s^+(t)$, and $\sup_{t,s\in\R}(\partial^2/\partial T^2)\,g_s^d(t,T)<0$ for
$T>T_2$.

Note that
$\sup_{t\in\R}T_s^-(t)<\inf_{t\in\R}T_s^+(t)$ for each $s\in\R$ if
$I_+-\sqrt{I_+\,I_-\,(1-d)}<I_-+I_-\sqrt{1-d}$ for $I_+\ge \sup_{t\in\R}I(t)$;
i.e., if $0<d<d_1:=1-(\sqrt{I_+/I_-} -1)^2$. In this case, a strict lower solution exists
for each equation \eqref{eq:Fraed1ts}$_s$ (including $s=\pm\infty$, i.e., the past and future equations), and it lies in
the area of concavity whenever $T_2\le\sup_{s,t\in\R}T_s^-(t)$, which we assume from now on (and which occurs if
$d>5/9$).
 This existence ensures the occurrence of an attractor-repeller pair of solutions for
\eqref{eq:Fraed1ts}$_s$: a pair of hyperbolic solutions, attractive the upper one and repulsive the
lower one, which determine the global dynamics in $\R\times(T_2,\infty)$.
Similarly, $\sup_{t,s\in\R}T_s^-(t)<\inf_{t,s\in\R}T_s^+(t)$ if
$b_+\big(I_+-\sqrt{I_+\,I_-\,(1-d)}\big)<b_-\big(I_-+I_-\sqrt{1-d}\big)$ where $b_-\le b(s)\le b_+$ for all $s\in\R$; i.e.,
if $0<d<d_2:=1-\big(b_+I_+/(b_-I_-)-1\big)^2\big(b_+\sqrt{I_+}/(b_-\sqrt{I_-})+1\big)^{-2}$ (which satisfies
$d_2\le d_1$). This property ensures that
also the transition equation $T'=g_t^d(t,T)$ has a strict lower solution in the area of concavity, and hence an
atractor-repeller pair. In addition, using similar techniques to those of \cite{lno,lnor}, it can be proved that this attractor-repeller pair approaches that of the past equation as time decreases and that of the future
equation as time increases, which is the situation usually called {\em tracking}. And this is true for all the equations
$T'=g_{c(t+l)}^d(t,T)$ for $c>0$ and $l\in\R$. In other words: there are no critical transitions
of rate-induced, size-induced or phase-induced type for $d\in(5/9,d_2)$.
\par
It can also be proved that, if $d\in(5/9,d_1)$, the existence of an attractor-repeller pair for $T'=g_s^d(t,T)$ for each
$s\in\R$ ensures the absence of critical transitions for a rate $c>0$ small enough.
We will now construct an example showing that, for a fixed $d\in(d_2,d_1)$, rate-induced tipping may occur
as $c$ increases, and that the value of the critical point may change with the initial phase.
And will also show that size-induced tipping may occur for a fixed small rate for values of $d$ slightly greater than $d_1$. We take the same $b(t)$ as in \cite{awvc}, namely
\[
 b(t):=1.690{\cdot}10^{-5}(1-\lambda(t/\kappa))+1.835{\cdot}10^{-5}\lambda(t/\kappa)
\]
where $\lambda(t)=ke^t/(1+ke^t)$ satisfies $\lambda(0)=10^{-6}$. Recall that
$I(t):=\mu(t)\,I_0$ for $\mu$ given by \eqref{2.defmu}, and that
$m(t):=I_- b^2(t)/(16\,\ep_{sa}\,\sigma)$. We consider the three-parametric family of equations
\begin{equation}\label{eq:Fraedcdl}
 T'=\frac{\kappa}{\tau}\,\left(-\varepsilon_{sa}\,\sigma\,T^4+\frac{1}{4}\,I(\kappa\,u)\,b(c\kappa(u+l))\,T^2-
 \frac{1}{4}\,I(\kappa\,u)\,d\,m(c\kappa(u+l))\right)
\end{equation}
where $u$ is time in years, $T(u)$ is the temperature in Kelvin and, for the transition: $c$ is the rate,
$d$ determines its size, and $l$ is its initial phase. Taking $\mu_-:=0.99993$, $\mu_+:=1.00007$,
$b_-=1.690{\cdot}10^{-5}$ and $b_+:=1.835{\cdot}10^{-5}$, we get
$d_1=0.9999995097$, $d_2=0.9982486778$, $T_1=298.6397736$ and $T_2=172.4197537$.
The numerical evidence, depicted in Figure \ref{fig:f1-tipping}, suggests that any tipping point for this model happens
through the collision of an attractor and a repeller
which do no longer exist after the critical transition.
This is in fact the only mechanism leading to a tipping point in ``the concave region" of this model,
as we shall see in Section \ref{3.sec}.
\begin{figure}
    \centering
    \includegraphics[width=0.98\textwidth]{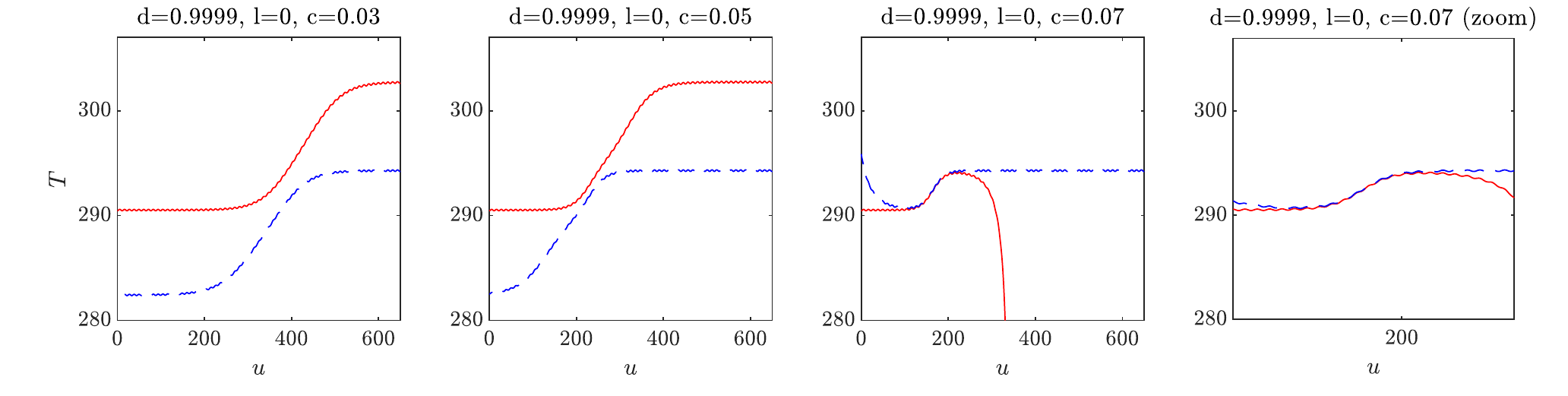}
    \includegraphics[width=0.98\textwidth]{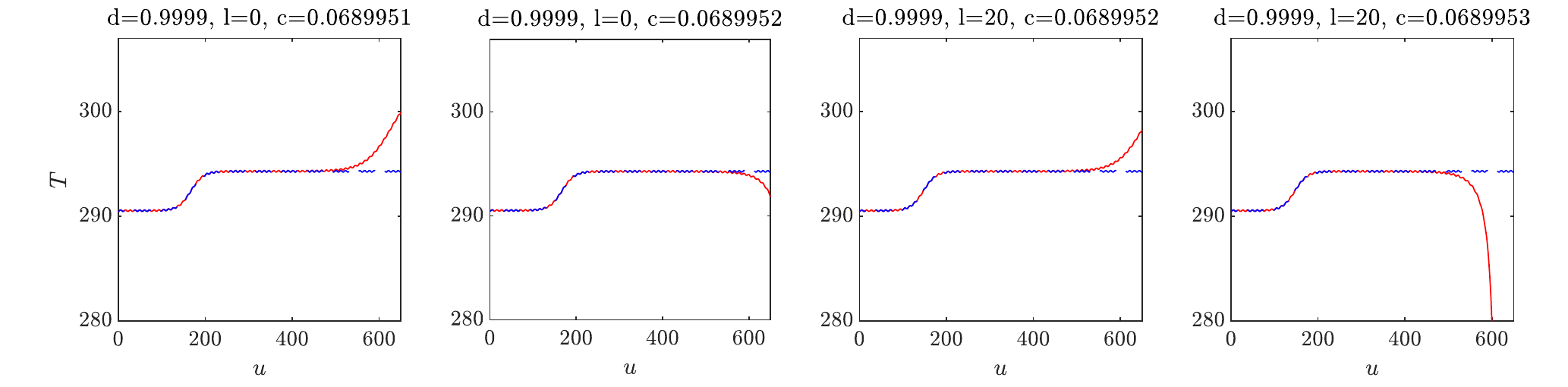}
    \includegraphics[width=0.98\textwidth]{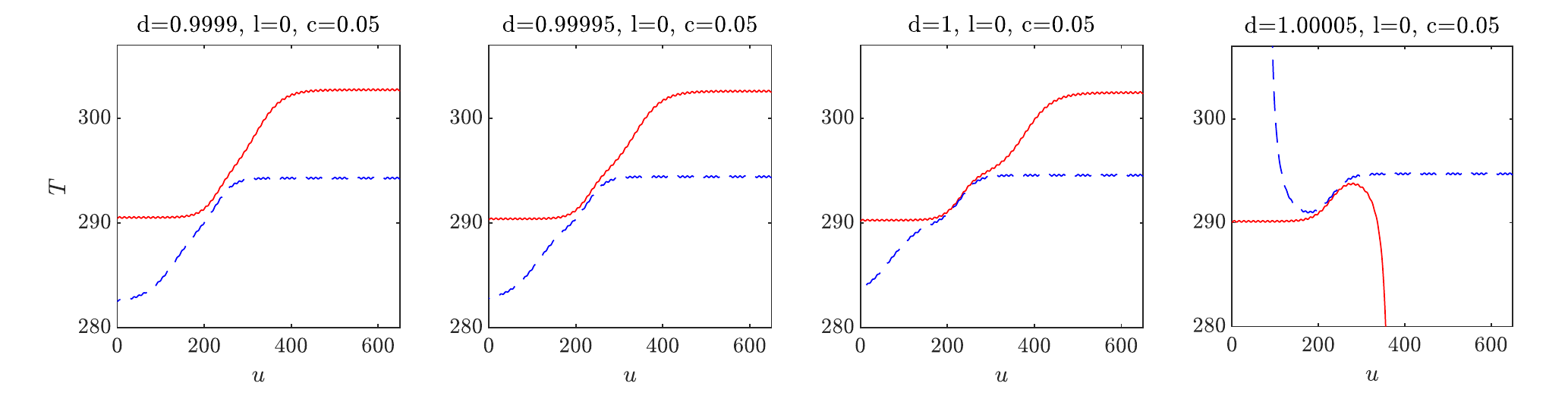}
    \caption{Different critical transitions occur for \eqref{eq:Fraedcdl} depending on the values of $d$, $l$ and $c$. In the horizontal axes, time in years; in the vertical ones, temperature in Kelvin. The solid red and dashed blue lines are respectively the locally pullback attractive and locally pullback repulsive solutions. Tracking (or tipping) occurs whenever the red curve is above (or below) the blue one at a common value of the time. The first row shows that, for a fixed value of $d\in(d_1,d_2)$ and for $l=0$, a tipping rate $c$ occurs between $0.05$ and $0.07$; and a zoom shows the (always present) oscillatory character of the solutions. The first two panels in the second row provide a more accurate approximation for the critical rate, while the third and forth ones show that a different initial phase in the transition ($l=20$) means a different value for the critical rate. The third row shows the occurrence of a size-induced critical transition as $d$ increases for fixed values of $c$ and $l$.}
    \label{fig:f1-tipping}
\end{figure}

Fraedrich~\cite{fraed} presents two additional models with a greenhouse feedback law (respectively derived from \eqref{eq:Fraed0} and \eqref{eq:Fraed1}).  The underlaying motivation is that the temperature and the amount of vapour water in the air depend on each other. An empirical relation between the temperature and the long-wave radiation $L\!\uparrow$ of the atmosphere is the so-called Swinbank's formula according to which $L\!\uparrow\,\sim T^6$.~Therefore,
it is assumed that $\varepsilon_a$ follows a quadratic law, $\varepsilon_a=\varepsilon_c+k\,T^2$, where $\varepsilon_c=0.0235\, \textrm{ln}(CO_2)$ (with $CO_2$ in ppm) is the $CO_2$ emittance, and $k=3{\cdot}10^{-6}K^{-2}$~\cite{fraed}. The basic climate model \eqref{eq:Fraed0} with the greenhouse feedback law becomes
\begin{equation}\label{eq:Fraed2}
 T'=\frac{1}{\tau}\left(k\,\sigma \,T^6 -\varepsilon_{sc}\,\sigma\,T^4+\frac{1}{4}\,\mu(t)\,I_0\,(1-\alpha)\right)=:f_2(t,T)\,,
\end{equation}
with $\alpha=0.284$, $\varepsilon_{sc}=\varepsilon_{s}-\varepsilon_{c}=0.87$, and the rest of the constants as
in \eqref{2.cons}. Finally, \eqref{eq:Fraed0}
with both the ice-albedo feedback and the greenhouse feedback laws becomes
\begin{equation}\label{eq:Fraed3}
T'=\frac{1}{\tau}\left(k\,\sigma\,T^6 -\varepsilon_{sc}\,\sigma\,T^4+\frac{1}{4}\,\mu(t)\,I_0\,b\,T^2+ \frac{1}{4}\,\mu(t)\,I_0\,(1-a)\right)=:f_3(t,T)\,.
\end{equation}
It is easy to check that $f_2(t,T)$ and $f_3(t,T)$ are convex in $T$ for all $t$ if $T$ is larger than a certain constant.
Numerical simulations show that \eqref{eq:Fraed2} and \eqref{eq:Fraed3} also have an attractor-repeller pair in both their autonomous
(with $\mu(t)\equiv 1$) and nonautonomous (with $\mu(t)$ given by \eqref{2.defmu}) versions.
Importantly, this implies that a similar mechanism for tipping as the one in \eqref{eq:Fraed1}, due to the collision of the attractor and the repeller,
is possible also in  \eqref{eq:Fraed2} and \eqref{eq:Fraed3}.
However, the relative position of the attractor and repeller (and hence the whole dynamics) is inverted in the convex case with respect to that in the concave one. Hence, tipping points appearing in \eqref{eq:Fraed2} and \eqref{eq:Fraed3} are opposite in nature with respect to those in \eqref{eq:Fraed1}. Specifically, the concave problems passing through a bifurcation point will end up in a deep-frozen state, while the convex ones in an arid-desert state. It is worth noting that such a tipping point for the later models would take place at an earth temperature which is already beyond the survival of life on Earth, as we know it.
\subsection{Critical transitions for Hopfield synchronized neural networks}
A biological neuron is an electrically charged cell which is able to receive and transmit a signal via a complex electrochemical mechanism.
In either case, each neuron has an equilibrium voltage where it tends to spontaneously remain in the absence of inputs. 
An inhibitory or excitatory signal pulse triggers a resistor-capacitor type exponential response followed by
relaxation. Specifically, when the injected current causes the voltage to exceed a certain threshold $\Theta$, 
a special electrochemical process generates a rapid
increase in the potential leading to a single non-weakening and propagating pulse. 
This phenomenon is called neuron firing. The pulse halts at the synapses where the transmission of the information to the next neurons is taken care of by special neurotransmitter molecules. Depending on the physiological effect that these neurotransmitter induces in the next neuron, the synapse is termed excitatory or inhibitory, whether they facilitate or not the occurrence of a pulse by increasing or lowering the potential of the next neuron.

An artificial neural network mimics the same process either using electrical components or as a software simulation on a digital computer. For this reason, the construction of a differential model of a  neural network of $N$ neurons of either type can be carried out analogously. We  follow the presentation by Wu~\cite{wu}, to which we point the interested reader for further details. We consider a network of $N>0$ neurons and denote by $x_i$, for $i=1,\dots, N$, the potential of the $i$-th neuron in terms of deviation from the equilibrium potential. The variable $x_i$ is also called the action potential or \emph{short-term memory} or {\em STM}.  
The change of neuron potential can be caused by internal and external processes,
\[
x_i'=\left(x_i'\right)_{\text{internal}}+
\left(\left(x_i'\right)_{\text{excitatory}}+\left(x_i'\right)_{\text{inhibitory}}+\left(x_i'\right)_{\text{stimuli}}\right)_{\text{external}}\,.
\]
For the internal dynamics, it is assumed that, in absence of external inputs, the neurons' potential decays
exponentially to the equilibrium; i.e.,~for $i=1,\dots,N$,
\[
 \left(x_i'\right)_{\text{internal}}=-a_i(t)\,x_i   \quad\text{for some } a_i\colon\R\to(0,\infty)\,.
\]
For the sake of simplicity, we assume that the external processes can be only excitatory or external
stimuli (i.e., $(x_i')_{\text{inhibitory}}\equiv 0$), following the laws
\[
 \left(x_i'\right)_{\text{excitatory}}=\!\sum_{\substack{k=1\\k\neq i}}^N Z_{ik}(t)\,\wit f_k(x_k-\Theta_k)\,,
 \qquad \left(x_i'\right)_{\text{stim.}}=\wit I_i(t)\,,
\]
where $\wit f_k:\R\to [0,\infty)$ is the \emph{signal function} of the $k$-th neuron, $\Theta_k$ is the firing threshold,
$Z_{ik}(t)=Z_{ki}(t)>0$
(also called the \emph{long-term memory} or {\em LTM}) represents the synaptic coupling coefficient,
and $\wit I_i(t)$ is the input from a contiguous patch of neurons.
The signal function $\wit f_k$ is typically a step function, a piecewise linear function, or a sigmoid limiting at $0$ as $y\to-\infty$ and at $1$ as $y\to\infty$.
In order to read the model as a classical Hopfield-type neural network (see \cite[Chapter 4]{wu}),
two steps are needed.
First, we assume that the signal function
$\wit f_k$ is a common increasing sigmoid $\wit f\colon\R\to[0,1]$ for all the neurons, and define $f(x)=2\,\wit f(x)-1$, which is strictly increasing and
limits at $\pm 1$ as $y\to\pm\infty$. That is, we obtain
\[
 x_i'=-a_i(t)\,x_i+\frac{1}{2}\,\sum_{\substack{k=1\\k\neq i}}^N Z_{ik}(t)\,f(x_k-\Theta_k)+\wit I_i(t)+\frac{1}{2}\,\sum_{\substack{k=1\\k\neq i}}^N Z_{ik}(t)\,.
\]
Second, we make the changes of variables $y_i=x_i-\Theta_i$, so that $y_i$ measures the deviation of the voltage of the $i-th$ neuron
from its activation threshold, getting
\[
 y_i'=-a_i(t)\,y_i+\frac{1}{2}\,\sum_{\substack{k=1\\k\neq i}}^N Z_{ik}(t)\,f(y_k)+\wit I_i(t)+\frac{1}{2}\,\sum_{\substack{k=1\\k\neq i}}^N Z_{ik}(t)-\Theta_i\,a_i(t)\,.
\]
Now, in order to get a synchronized neural network, we assume that
$Z_{ik}(t)$ synchronize at a certain bounded continuous function
$\wit z\colon\R\to(0,\infty)$ which can be fed into the previous equation as a reference trajectory.
And we also assume that $a_i=a$, $\Theta_i=\Theta$, and $\wit I_i=\wit I$
for $i=1,\ldots,N$. So, whenever the system is initialized with synchronized (i.e., equal) initial conditions $y_i(0)=y_0$ for $i=1,\dots, N$,
the resulting solutions remain synchronized for all $t>0$. In this case, the previous system reduces to a unique scalar differential equation
satisfied by all the neurons,
\begin{equation}\label{eq:hopf}
 y'=-a(t)\,y_i+z(t)\,f(y)+I(t)\,,
\end{equation}
where $z(t):=(N-1)\,\wit z(t)/2$ and $I(t):=\wit I(t)+(N-1)\,\wit z(t)/2-\Theta\,a(t)$.
The function $a$ will be important for our tipping analysis, as better explained in what follows.
\par
In our presentation, we will assume that $z(t)$ is positive and bounded, that $f(0)=0$, and that
the resulting input $I(t)$ is mildly inhibitory: upper-bounded by a small negative constant $k<0$.
Since, for $y>0$, $-a(t)\,y+z(t)\,f(y)+I(t)<f(y)\n{z}_\infty+k$,
there is a value $\bar y>0$ such that, for all $y\in(0,\bar y)$, the right-hand side of \eqref{eq:hopf} is negative.
In other words, every $y\in(0,\bar y)$ is an upper solution  for \eqref{eq:hopf}, which means that any
solution starting below $\bar y$ is destined to decrease as time increases while it is positive. Hence,
if a bounded and positive solution exists, then it is located above $\bar y$.
\begin{figure}
    \centering
    \includegraphics[width=\textwidth]{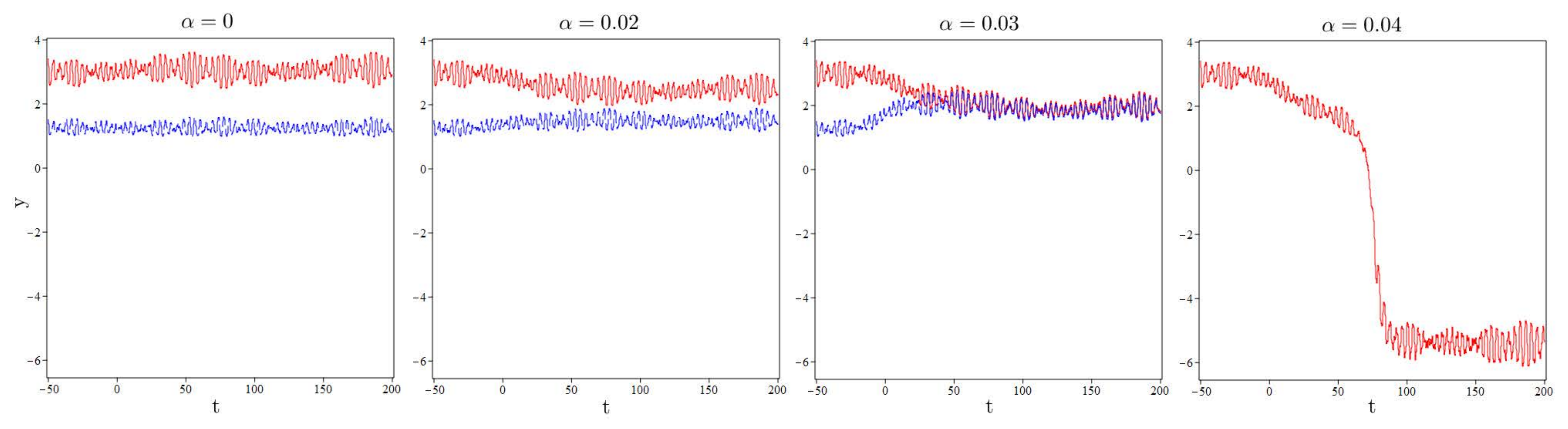}
    \caption{A time-dependent tipping scenario is depicted for the Hopfield-type differential equation \eqref{eq:hopf}
        with $I(t)$ and $z(t)$ given by \eqref{2.defz}, and with $a(t)=a_\alpha(t)$ given by \eqref{eq:incr_internal_decay_rate} and
        hence depending on the parameter $\alpha$ which characterizes the internal decay rate of a neuron:
        this rate increases with $\alpha$. After tipping
        (when the dashed blue line is no longer below the solid red one: see Figure \ref{fig:f1-tipping}),
        the set of synchronized neurons collapses to a (stable)
        ``inactive'' state. Indeed, after some time, the potential of each neuron remains uniformly bounded below
        the activation threshold 0.}
    \label{fig:neurons_tipping}
\end{figure}
For our tipping analysis, we assume that the internal decay rate $a(t)$ has a baseline behaviour $a_0(t)$ and
an increasing trend with time (for example due to ageing): $a(t)$ becomes
\begin{equation}\label{eq:incr_internal_decay_rate}
\begin{split}
 a_\alpha(t)&:=a_0(t)+\alpha\,(1+e^{-t/5})^{-1}\\
 &:=0.2\,\left(1+\sin(\pi t)\,\sin\big(((1+\sqrt{5})\,t)/2\big)\right)
 +\alpha\,(1+e^{-t/5})^{-1}\,.
\end{split}
\end{equation}
In addition, we take
\begin{equation}\label{2.defz}
\begin{split}
 f(y)&:=(1-e^{-y})/(1+e^{-y})\,,\\
 I(t)&:=0.2\,\big(\sin(2\pi t/\sqrt5)\,\cos(t/7)-1.5\big)\,,\\
 z(t)&:= 1+0.75\,\sin(t/7)\sin\big((1+\sqrt{5})\,t/2\big)\,.
\end{split}
\end{equation}
Note that the function $-a(t)\,y_i+z(t)\,f(y)+I(t)$ is concave in $y$
for $y>0$ and convex for $y<0$, since so is $f$ and $z(t)\ge0.25$.

As shown in Figure \ref{fig:neurons_tipping}, there exists a value $\alpha>0$ beyond which the set of synchronized neurons collapses to a (stable)
``inactive'' state: some time after $t=0$ the state of each neuron remains uniformly bounded below the activation threshold
(which is 0 after the change of variables).
Observe finally that the critical transition occurs as $\alpha$ increases.
Since $\alpha$ determines the size of the transition, we have a size-induced tipping point.
\section{General results for concave and coercive Carath\'{e}odory ODEs}\label{3.sec}
Throughout the paper, $L^\infty=L^\infty(\R,\R)$ is the Banach space of
essentially bounded functions $q\colon\R\to\R$ endowed
with the norm $\n{q}_{L^\infty}$ given by the inferior of the set of
real numbers $k\ge 0$ such that of
$l(\{\,t\in\R\:|\;|q(t)|>k\,\})=0$, where $l(\mI)$ is the Lebesgue measure of
$\mI\subseteq\mR$.
We write ``\,$l$-a.a."~instead of ``for Lebesgue almost all $t\in\R$", ``Lebesgue almost always"
and ``Lebesgue almost all".
\par
We consider nonautonomous scalar equations of the type
\begin{equation}\label{3.ecu}
 x'=f(t,x)\,,
\end{equation}
where $f:\R^2\to\R$ is assumed to satisfy (all or part of) the next conditions:
\begin{list}{}{\leftmargin 17pt}
\item[\hypertarget{f1}{{\bf f1}}] $f$ is Borel measurable;
\item[\hypertarget{f2}{{\bf f2}}] for all $j\in\N$ there exists
$m^f_j\in\R$ such that, $l$-a.a., $\sup_{x\in[-j,j]}|f(t,x)|\le m^f_j$.
\item[\hypertarget{f3}{{\bf f3}}] for all $j\in\N$ there exists
$l^f_j\in\R$ such that, $l$-a.a., $\sup_{x_1,\,x_2\in[-j,j]}|f(t,x_2)-f(t,x_1)|/|x_2-x_1|\le l^f_j$;
\item[\hypertarget{f4}{{\bf f4}}] $l$-a.a., the map $x\mapsto f(t,x)$ is $C^1$;
\item[\hypertarget{f5}{{\bf f5}}] the map $x\mapsto
f_x(t,x)$ satisfies condition \hyperlink{f3}{\bf f3},
where $f_x=\partial f/\partial x$;
\item[\hypertarget{f6}{{\bf f6}}] (coercivity) there exists a subset
$\mR\subseteq\R$ with full Lebesgue measure and $\delta>0$ such that
$\limsup_{x\to\pm\infty} f(t,x)/x^2<-\delta$ uniformly on $\mR$;
\item[\hypertarget{f7}{{\bf f7}}] (strict concavity) for all $j\in\N$
there exists a constant $\delta_j>0$ such that, $l$-a.a.,
$\sup_{x_1,\,x_2\in[-j,j],\,x_1\ne x_2}
(f_x(t,x_2)-f_x(t,x_1))/(x_2-x_1)<-\delta_j$.
\end{list}
All these conditions are in force unless otherwise stated.
The results establishing the existence, uniqueness and properties of
maximal solutions of \eqref{3.ecu} for $f$ satisfying
\hyperlink{f1}{\bf f1}, \hyperlink{f2}{\bf f2}
and \hyperlink{f3}{\bf f3} (less restrictive conditions, in fact) can be
found in~\cite[Chapter 2]{cole}.
In the usual terminology for Carath\'{e}odory and Lipschitz-Carath\'{e}odory
equations, conditions \hyperlink{f2}{\bf f2} and \hyperlink{f3}{\bf f3}
mean the {\em existence of $L^\infty$ $m$-bounds and $l$-bounds}, respectively.
It is easy to check that, if $f$ satisfies \hyperlink{f1}{\bf f1} and
\hyperlink{f3}{\bf f3}, then it is {\em strong Carath\'{e}odory}: the map $x\mapsto f(t,x)$
is continuous on $\R$ for $l$-a.a.~$t\in\R$.
Conditions \hyperlink{f4}{\bf f4} and \hyperlink{f5}{\bf f5} require more
regularity on the state variable $x$, and will allow us to discuss hyperbolicity of
solutions of \eqref{3.ecu}.
Note that, if \hyperlink{f4}{\bf f4} holds, condition \hyperlink{f3}{\bf f3}
ensures \hyperlink{f2}{\bf f2} for $f_x$.
Conditions \hyperlink{f6}{\bf f6} and \hyperlink{f7}{\bf f7} provide us with
a framework where, for instance, the existence
of an attractor-repeller pair of hyperbolic solutions can be discussed.
\par
Let $t\mapsto x(t,s,x_0)$ denote
the maximal solution of \eqref{3.ecu}
satisfying $x(s,s,x_0)=x_0$, defined on $\mI_{s,x_0}=(\alpha_{s,x_0},\beta_{s,x_0})$
with $-\infty\le\alpha_{s,x_0}<s<\beta_{s,x_0}\le\infty$. Recall that,
in this setting, a solution is an absolutely continuous function
on each compact interval of $\mI_{s,x_0}$
which satisfies \eqref{3.ecu} at Lebesgue a.a.~$t\in\mI_{s,x_0}$;
and that $\mI_{s,x_0}=\R$ if $x(t,s,x_0)$ is bounded.
The results establishing the existence, uniqueness and properties of this
maximal solution for $f$ satisfying \hyperlink{f1}{\bf f1}, \hyperlink{f2}{\bf f2}
and \hyperlink{f3}{\bf f3}
(less restrictive conditions, in fact) can be found in~\cite[Chapter 2]{cole}.
Recall also that the real map $x$, which is
defined on an open subset of $\R\times\R\times\R$ containing
$\{(s,s,x_0)\,|\;s,x_0\in\R\}$, satisfies $x(s,s,x_0)=x_0$ and
$x(t,l,x(l,s,x_0))=x(t,s,x_0)$ whenever all the involved terms
are defined. In addition, since the countable intersection of
subsets of $\R$ with full Lebesgue measure has full Lebesgue measure,
property \hyperlink{f7}{\bf f7}
ensures that, for $l$-a.a.~$t\in\R$, the map $x\mapsto f_x(t,x)$ is strictly
decreasing, and hence $f(t,(1-\rho)x_1+\rho x_2)>(1-\rho)f(t,x_1)+\rho f(t,x_2)$ for all $x_1,x_2\in\R$
and $\rho\in(0,1)$. In turn, this last property guarantees
the strict concavity of $x_0\mapsto x(t,s,x_0)$ for fixed $t>s$:
the statements and proof of ~\cite[Proposition 2.3]{lno} are also valid
in the more general setting here considered.
\subsection{The special solutions $\ma$ and $\mr$}
The coercivity property \hyperlink{f6}{\bf f6} has consequences on the properties of the sets
\[
\begin{split}
 \mB^-&:=\Big\{(s,x_0)\in\R^2\,\Big|\;\sup_{t\in(\alpha_{s,x_0},s]}x(t,s,x_0)
 <\infty\Big\}\,,\\
 \mB^+&:=\Big\{(s,x_0)\in\R^2\,\Big|\;\inf_{t\in[s,\beta_{s,x_0})}x(t,s,x_0)
 >-\infty\Big\}\,,\\
\end{split}
\]
which may be empty.
We fix any $\ep>0$ and take $m>0$ large enough to satisfy
$f(t,x)\le-\ep$ for all $t\in\mR$, and $|x|\ge m$.
By rewriting as $x(t)=x(s)+\int_s^t f(l,x(l))\,dl$, it is easy to check
that {\bf any solution remains upper bounded as time increases
and lower bounded as time decreases},
with $\limsup_{t\to(\beta_{s,x_0})^-}x(t)<m$ and
$\liminf_{t\to(\alpha_{s,x_0})^+}x(t)>-m$.
We will often use these properties
in the paper without further reference.
They imply that $\mB^-$ (resp.~$\mB^+$) is the set of pairs $(s,x_0)$
giving rise to solutions which remain bounded in the past (resp.~in the future).
Note that $\alpha_{s,x_0}=-\infty$ for all $(s,x_0)\in\mB^-$ and
$\beta_{s,x_0}=\infty$ for all $(s,x_0)\in\mB^+$.
The set $\mB:=\mB^-\cap\mB^+$ is the (possibly empty) set of
pairs $(s,x_0)$ giving rise to
(globally defined) bounded solutions of~\eqref{3.ecu}.
\par
The next result is proved by repeating the arguments leading to
\cite[Theorem 2.5]{lno} (in turn based on~\cite[Theorem 3.1]{lnor}).
Conditions \hyperlink{f4}{\bf f4}, \hyperlink{f5}{\bf f5}
and \hyperlink{f7}{\bf f7} are not required.
\begin{teor}\label{3.teoruno}
Let $\mB^\pm,\mB$ and $m$ be
the sets and constant above defined.
\begin{itemize}
\item[\rm(i)] If $\mB^-$ is nonempty, then there exist
a set $\mR^-$ coinciding with $\R$ or with
a negative open half-line and a maximal solution
$a\colon\mR^-\to(-\infty,m)$ of~\eqref{3.ecu} such that,
if $s\in\mR^-$, then
$x(t,s,x_0)$ remains bounded as $t\to-\infty$ if and only if $x_0\le a(s)$; and if
$\sup\mR^-<\infty$, then $\lim_{t\to(\sup\mR^-)^-} a(t)=-\infty$.
\item[\rm(ii)] If $\mB^+$ is nonempty, then there exist
a set $\mR^+$ coinciding with $\R$ or with
a positive open half-line and a maximal solution
$r\colon\mR^+\to(-m,\infty)$ of~\eqref{3.ecu} such that,
if $s\in\mR^+$, then $x(t,s,x_0)$ remains bounded as $t\to+\infty$
if and only if $x_0\ge r(s)$; and
if $\inf\mR^+>-\infty$, then $\lim_{t\to(\inf\mR^+)^+} r(t)=\infty$.
\item[\rm(iii)] Let $x$ be a solution defined on a maximal interval $(\alpha,\beta)$.
If it satisfies $\liminf_{t\to\beta^-}x(t)=-\infty$, then $\beta<\infty$; and if
$\limsup_{t\to\alpha^+}x(t)=\infty$, then $\alpha>-\infty$.
In particular, any globally defined solution is bounded.
\item[\rm(iv)] The set $\mB$ is nonempty if and only if $\mR^-=\R$ or $\mR^+=\R$, in
which case both equalities hold, $a$ and $ r$ are
globally defined and bounded solutions of~\eqref{3.ecu}, and
$\mB=\{(s,x_0)\in\R^2\,|\; r(s)\le x_0\le a(s)\}\subset\R\times[-m,m]$.
\item[\rm(v)]
Let the function $b\colon\R\to\R$ be bounded, continuous,
of bounded variation and with nonincreasing singular part on every compact interval of
$\R$. Assume that $b'(t)\le f\big(t,b(t)\big)$
for $l$-a.a.~$t\in\R$. Then, $\mB$ is nonempty, and $r\le b\le a$.
If, in addition, there exists $ t_0\in\R$ such that
$b'(t_0)<f\big(t_0,b(t_0)\big)$, then $r<a$.
And, if $b'(t)<f\big(t,b(t)\big)$ for $l$-a.a.~$t\in\R$,
then $r<b<a$.
\end{itemize}
\end{teor}
\begin{nota}\label{3.notabasta}
As explained in~\cite[Remark 2.6]{lno}, \eqref{3.ecu} has a bounded solution
if and only if there exists a time $t_0$
such that the solutions $a$ and $r$ defined in Theorem~\ref{3.teoruno}
are respectively defined at least on $(-\infty,t_0]$ and $[t_0,\infty)$,
with $a(t_0)\ge r(t_0)$; and the inequality $a(t_0)>r(t_0)$
is equivalent to the existence of at least two bounded solutions.
If $t_0$ with $a(t_0)<r(t_0)$ exists, then there are no bounded solutions.
\end{nota}
\subsection{Hyperbolic solutions, their persistence, and local pullback attractor and repellers}
\label{3.subsechyp}
In this subsection, conditions \hyperlink{f6}{\bf f6}
and \hyperlink{f7}{\bf f7} are not required.
A bounded solution $\wit b\colon\R\to\R$ of~\eqref{3.ecu} is said to be
{\em hyperbolic\/} if the corresponding variational equation $z'=f_x(t,\wit b(t))\,z$
has an exponential dichotomy on $\R$. That is (see~\cite{copp1}), if
there exist $k\ge 1$ and $\beta>0$ such that either
\begin{equation}\label{3.masi}
 \exp\int_s^t f_x(l,\wit b(l))\,dl\le k\,e^{-\beta(t-s)} \quad
 \text{whenever $t\ge s$}
\end{equation}
or
\begin{equation}\label{3.menosi}
 \exp\int_s^t f_x(l,\wit b(l))\,dl\le k\,e^{\beta(t-s)} \quad
 \text{whenever $t\le s$}
\end{equation}
holds. In the case~\eqref{3.masi}, the hyperbolic solution $\wit b$ is said to be
{\em (locally) attractive}, and in the case \eqref{3.menosi}, $\wit b$ is {\em (locally) repulsive}.
In both cases, we call $(k,\beta)$ a (non-unique)
{\em dichotomy constant pair\/} for the hyperbolic solution $\wit b$, or for the linear
equation $z'=f_x(t,\wit  b(t))\,z$.
\par
Condition \hyperlink{f5}{\bf f5} shows the existence of the Lipschitz coefficient $l^{f_x}=l^{f_x}(\rho)$
appearing in the statement of the next result.
\begin{prop}\label{3.proppersiste}
Assume the existence of an attractive (resp.~repulsive) hyperbolic solution
$\wit b_f$ of $x'=f(t,x)$ with dichotomy constant pair $(k_f,\beta_f)$.
Take $\rho>0$ and a constant $l^{f_x}\ge\sup_{x_1,x_2\in[-\rho,\rho],x_1\ne x_2} |f_x(t,\wit b_f(t)+x_1)-f_x(t,\wit b_f(t)+x_2)|/|x_1-x_2|\,$ $l$-a.a.
Then, given $\beta\in(0,\beta_f)$, there exists
$\delta^*=\delta^*(k_f,\beta,\rho,l^{f_x})>0$ such that, for any $\delta_*\in(0,\delta^*]$, the equation $x'=g(t,x)$ has
an attractive (resp.~repulsive) hyperbolic solution $\wit b_g$ with
$\big\|\wit b_f-\wit b_g\big\|_\infty\le\delta_*$
and dichotomy constant pair $(k_f,\beta)$ provided that
\begin{itemize}
\item[\rm 1.] $g$ satisfies \hyperlink{f1}{\bf f1}, \hyperlink{f2}{\bf f2},
\hyperlink{f3}{\bf f3}, \hyperlink{f4}{\bf f4} and \hyperlink{f5}{\bf f5},
\item[\rm 2.] $\big|g(t,\wit b_f(t))-f(t,\wit b_f(t))\big|\le \beta\,\delta_*/(2k_f)$ $l$-a.a.,
\item[\rm 3.] $\sup_{x\in[-\rho,\rho]}|f_x(t,\wit b_f(t)+x)-g_x(t,\wit b_f(t)+x)|<\beta_f-\beta$ $l$-a.a.,
\item[\rm 4.] $\sup_{x_1,x_2\in[-\rho,\rho],x_1\ne x_2} |g_x(t,\wit b_f(t)+x_2)-
g_x(t,\wit b_f(t)+x_1)|/|x_2-x_1|\le l^{f_x}$ $l$-a.a.
\end{itemize}
\end{prop}
\begin{proof}
The proof, similar to that of~\cite[Proposition 3.2]{lnor},
is based on~\cite[Lemma 3.3]{aloo}. We include the details in this Carath\'{e}odory setting for the reader's
convenience.
\par
Let us work in the attractive case. We call $\mB_\nu:=
\{\bar y\in C(\R,\R)\,|\;\n{\bar y}_\infty\le\nu\}$ for $\nu>0$ and,
to simplify the notation, $b:=\wit b_f$. We assume that $g$ satisfies
conditions 1, 3 and 4, and call $l^{f,g}_\rho:=\beta_f-\beta-
\sup_{|x|\le\rho}|f_x(t,\wit b_f(t)+x)-g_x(t,\wit b_f(t)+x)|$.
Condition \hyperlink{f5}{\bf f5} allows us to find $\delta\in(0,\rho]$
such that $|f_x(t,b(t)+\bar y(t))-f_x(t,b(t))|\le l^{f,g}_\rho$
$\,l$-a.a.~if $\bar y\in\mB_\delta$, and hence, using hypothesis 3,
\[
\begin{split}
 g_x(t,b(t)+\bar y(t))&=f_x(t,b(t)+\bar y(t))+(g_x(t,b(t)+\bar y(t))-f_x(t,b(t)+\bar y(t)))\\
 &\le
 f_x(t,b(t))+(\beta_f-\beta)
\end{split}
\]
$l$-a.a. This bound combined with relation \eqref{3.masi} for $b$ ensures that, if $\bar y\in\mB_\delta$, then
\begin{equation}\label{3.byhyper}
 \exp\int_s^t g_x(l,b(l)+\bar y(l))\,ds \le k_f\,e^{-\beta(t-s)}\quad \text{whenever $t\ge s$}\,.
\end{equation}
In particular, if $\wit b_g=b+\bar y$ is a solution of $x'=g(t,x)$, then it is hyperbolic attractive
with dichotomy pair $(k_f,\beta)$ and satisfies $\big\|\wit b_g-b\big\|_\infty=\n{\bar y}_\infty\le\delta$.
Our goal is hence to look for such a function $\bar y$, which will be obtained as the unique fixed
point of a contractive operator $T\colon\mB_{\delta_*}\to\mB_{\delta_*}$
if $\delta_*\in (0,\delta^*]$ for a suitable $\delta^*\in (0,\delta]$.
\par
Let us define $T$.
The change of variable $y=x-b(t)$ takes $x'=g(t,x)$ to
\[
 y'=g_x(t,b(t))\,y+h(t,y)\,,
\]
where $h(t,y):=g(t,b(t))-f(t,b(t))+r(t,y)$ for $r(t,y):=g(t,b(t)+y)-g(t,b(t))-g_x(t,b(t))\,y$.
Relation \eqref{3.byhyper} for $\bar y=0$ ensures that
\[
 u(t,s):=\exp\int_s^t g_x(l,b(l))\,dl\le k_f e^{-\beta(t-s)}\quad \text{whenever $t\ge s$}\,.
\]
Clearly, $t\mapsto u(t,s)$ solves $y'=g_x(t,b(t))\,y$ with $u(s,s)=0$. It is easy to check
(see~\cite[Section 3]{copp1}) that, for each $\bar y\in \mB_\delta$,
\[
 T\bar y(t):=\int_{-\infty}^t u(t,s)\,h(s,\bar y(s))\,ds
\]
is the unique bounded solution of $y'=g_x(t,b(t))\,y+h(t,\bar y(t))$, and that
\begin{equation}\label{3.Ty}
\begin{split}
 \n{T\bar y}_\infty&\le \n{h(\cdot,\bar y(\cdot))}_{L^\infty} k_f/\beta\,,\\
 \n{T\bar y_1-T\bar y_2}_\infty&\le \n{r(\cdot,\bar y_1(\cdot))-r(\cdot,\bar y_2(\cdot))}_{L^\infty}k_f/\beta\,.
\end{split}
\end{equation}
Let us fix $\delta^*\in(0,\delta]$ such that $l^{f_x}\delta^*< 0.9\,\beta/k_f$,
take any $\delta_*\in(0,\delta^*]$, and assume that $g$ also satisfies 2.
To check that $T\colon\mB_{\delta_*}\to\mB_{\delta_*}$ is well-defined, it suffices to check that
$|h(t,\bar y(t))|\le \beta\,\delta_*/k_f$ $l$-a.a. if $\bar y\in\mB_{\delta_*}$, and this bound follows from
$\big|g(t,b(t))-f(t,b(t))\big|\le \beta\,\delta_*/(2k_f)$ $l$-a.a.~(according to hypothesis 2) and
$|r(t,\bar y(t))|\le l^{f_x}|\bar y(t)|^2/2\le l^{f_x}(\delta_*)^2/2< \beta\,\delta_*/(2k_f)$ $l$-a.a.~(by
Taylor's theorem, hypothesis 4 and the choice of $\delta^*$). Similarly,
$|r(t,\bar y_1(t))-r(t,\bar y_2(t))|\le l^{f_x}|\bar y_1(t)-\bar y_2(t)|^2/2\le l^{f_x}\delta_*|\bar y_1(t)-\bar y_2(t)|\le
(0.9\,\beta/k_f)\,|\bar y_1(t)-\bar y_2(t)|$ $l$-a.a. if $\bar y_1,\,\bar y_2\in\mB_{\delta_*}$,
which combined with the second bound in \eqref{3.Ty} shows that $T\colon\mB_{\delta_*}\to\mB_{\delta_*}$
is contractive. We represent its unique fixed point by $\bar y_g\in\mB_{\delta_*}$. Then
$\wit b_g:=\wit b_f+\bar y_g$ solves $x'=g(t,x)$. As said before, this completes the proof.
\end{proof}
Let us define local {\em pullback\/} attractors and
repellers, which may appear also in the absence of hyperbolic solutions,
and which will play a fundamental role in the dynamical description of the next sections.
A solution $\bar a\colon(-\infty,\beta)\to\R$ (with $\beta\le\infty$) of~\eqref{3.ecu} is
{\em locally pullback attractive\/} if there exist $s_0<\beta$ and $\delta>0$ such that,
if $s\le s_0$
then the solutions $x(t,s,a(s)\pm\delta)$ are defined on
$[s,s_0]$ and, in addition,
\[
 \lim_{s\to-\infty}|\bar a(t)-x(t,s,\bar a(s)\pm\delta)|=0
 \quad \text{for all $t\le s_0$}\,.
\]
A solution $\bar r\colon(\alpha,\infty)\to\R$ (with $\alpha\ge-\infty$) of~\eqref{3.ecu} is
{\em locally pullback repulsive\/} if the solution $\bar r^*\colon(-\infty,-\alpha)\to\R$ of
$y'=-h(-t,y)$ given by $\bar r^*(t)=\bar r(-t)$ is locally pullback attractive.
That is, if there exist $s_0>\alpha$ and $\delta>0$ such that,
if $s\ge s_0$, then the solutions $x(t,s,\bar r(s)\pm\delta)$ are defined on
$[s_0,s]$ and, in addition,
\[
 \lim_{s\to\infty}|\bar r(t)-x(t,s,\bar r(s)\pm\delta)|=0
 \quad \text{for all $t\ge s_0$}\,.
\]
\subsection{Occurrence of an attractor-repeller pair}\label{3.subsecAR}
The main results in this paper rely on the fact that if the considered
equation has a local attractor, such attractor is in fact part of
a {\em classical attractor-repeller pair of hyperbolic solutions\/}.
Theorem \ref {3.teorhyp} (proved in Appendix \ref{appendix})
shows that
the solutions $a$ and $r$ associated to \eqref{3.ecu} by
Theorem~\ref{3.teoruno} are globally defined and uniformly separated
if and only if they are hyperbolic, in which case they provide the mentioned
attractor-repeller pair. The global dynamics in the case of existence of
such a pair is jointly described by Theorems~\ref{3.teoruno} and~\ref{3.teorhyp}.
\par
Recall that conditions \hyperlink{f1}{\bf f1}-\hyperlink{f7}{\bf f7}
are assumed on $f$, unless otherwise indicated.
\begin{teor}\label{3.teorhyp}
Assume that~\eqref{3.ecu} has bounded solutions,
and let $a$ and $r$ be the (globally defined) functions provided by
Theorem~{\em \ref{3.teoruno}}.
Then, the following assertions are equivalent:
\begin{itemize}
\item[\rm \rm(a)] The solutions $a$ and $r$ are uniformly separated: $\inf_{t\in\R}(a(t)-r(t))>0$.
\item[\rm \rm(b)] The solutions $a$ and $r$ are hyperbolic, with $a$ attractive and $r$ repulsive.
\item[\rm \rm(c)] The equation~\eqref{3.ecu} has two different hyperbolic solutions.
\end{itemize}
In this case,
\begin{itemize}
\item[\rm \rm(i)] let $(k_a,\beta_a)$ and $(k_r,\beta_r)$ be dichotomy constant pairs for
the hyperbolic solutions $a$ and $r$, respectively, and let us choose
any $\bar\beta_a\in(0,\beta_a)$ and any $\bar\beta_r\in(0,\beta_r)$.
Then, given $\ep>0$, there exist $k_{a,\ep}\ge 1$ and
$k_{r,\ep}\ge 1$ (respectively depending also on the choices of $\bar\beta_a$ and
of $\bar\beta_r$) such that
\[
\begin{split}
 &\quad\qquad |a(t)-x(t,s,x_0)|\le k_{a,\ep}\,e^{-\bar\beta_a(t-s)}|a(s)-x_0|
 \quad\text{if $x_0\ge r(s)+\ep$ and $t\ge s$}\,,\\
 &\quad\qquad |r(t)-x(t,s,x_0)|\le k_{r,\ep}\,e^{\bar\beta_r(t-s)}|r(s)-x_0|
 \quad\text{if $x_0\le a(s)-\ep$ and $t\le s$}\,.
\end{split}
\]
In addition,
\[
\begin{split}
 &\quad\qquad |a(t)-x(t,s,x_0)|\le k_a\,e^{-\beta_a(t-s)}|a(s)-x_0|
 \quad\text{if $x_0\ge a(s)$ and $t\ge s$}\,,\\
 &\quad\qquad |r(t)-x(t,s,x_0)|\le k_r\,e^{\beta_r(t-s)}|r(s)-x_0|
 \quad\text{if $x_0\le r(s)$ and $t\le s$}\,.
\end{split}
\]
\item[\rm \rm(ii)] The equation~\eqref{3.ecu} has no more
hyperbolic solutions, and
$a$ and $r$ are the unique bounded solutions of
\eqref{3.ecu} which are uniformly separated.
\end{itemize}
\end{teor}
\subsection{A nonautonomous saddle-node bifurcation pattern.}
Let us take $f$ as in \eqref{3.ecu} (i.e., satisfying
\hyperlink{f1}{\bf f1}-\hyperlink{f7}{\bf f7}),
and consider the parametric perturbations
\begin{equation}\label{3.ecuconlb}
 x'=f(t,x)+\lambda\,.
\end{equation}
Let $\mB_\lb$ be the (possibly empty) set of bounded solutions, and
$a_\lb$ and $r_\lb$ the upper and lower bounded solutions provided by
Theorem \ref{3.teoruno} when $\mB_\lb$ is nonempty.
Next, we aim to show the existence of a bifurcation value
$\lb^*$ for \eqref{3.ecuconlb}, which splits the space of parameters
in values where there are no bounded solutions and values
where two hyperbolic solutions exist. We will
talk hence about nonautonomous saddle-node bifurcation.
The proof of the next result, which is required for the dynamical
classification that we explain in Section \ref{4.sec},
is given in Appendix \ref{appendix}. The existence of the constants
$\rho$ and $m_\rho^f$ in the statement is ensured by
\hyperlink{f6}{\bf f6} and \hyperlink{f2}{\bf f2}.
\begin{teor}\label{3.teorlb*}
Let $\rho>0$ satisfy $f(t,x)<0$ if $|x|>\rho$ for all $t\in\mR$,
and let $m_\rho^f:=\big\|\sup_{|x|\le\rho}f({\cdot},x)\big\|_{L^\infty}$.
There exists a unique $\lb^*=\lb^*(f)\in[-\|m_\rho^f\|_{L^\infty},
\n{f({\cdot},0)}_{L^\infty}]$ such that
\begin{itemize}
\item[\rm \rm(i)] $\mB_\lb$ is empty if and only if $\lb<\lb^*$.
\item[\rm \rm(ii)] If $\lb^*\le\lb_1<\lb_2$,
then $\mB_{\lb_1}\varsubsetneq\mB_{\lb_2}$. More precisely,
\[
 r_{\lb_2}<r_{\lb_1}\le a_{\lb_1}<a_{\lb_2}\,.
\]
In addition,
$\lim_{\lb\to\infty}a_\lb(t)=\infty$ and $\lim_{\lb\to\infty}r_\lb(t)=-\infty$
uniformly on $\R$.
\item[\rm \rm(iii)] $\inf_{t\in\R}(a_{\lb^*}(t)-r_{\lb^*}(t))=0$,
and~\eqref{3.ecuconlb}$_{\lb^*}$ has no hyperbolic solutions.
\item[\rm(iv)] If $\lb>\lb^*$, then $a_\lb$ and $r_\lb$ are
uniformly separated and the unique hyperbolic solutions
of~\eqref{3.ecuconlb}$_\lb$.
\item[\rm \rm(v)] $\lb^*(f+\lb)=\lb^*(f)-\lb$ for any $\lb\in\R$.
\end{itemize}
\end{teor}
\section{More dynamical properties under an additional hypothesis}\label{4.sec}
Let $f\colon\R\to\R$ satisfy the conditions described in Section \ref{3.sec}.
From now on, we focus on equations of the type
\begin{equation}\label{4.ecDelta}
 x'=f(t,x-\Delta(t))\,,
\end{equation}
where $\Delta\colon\R\to\R$ is an $L^\infty$ function
with the fundamental property of existence of finite asymptotic limits
$\delta_\pm:=\lim_{t\to\pm\infty}\Delta(t)$.
It is easy to check that the map $f_\Delta(t,x):=f(t,x-\Delta(t))$
also satisfies conditions \hyperlink{f1}{\bf f1}-\hyperlink{f7}{\bf f7}.
Hence, Theorem \ref{3.teorlb*} shows that the dynamics induced
by \eqref{4.ecDelta} fits in one of the next situations:
\begin{itemize}
\item[\rm -] \hypertarget{CA}{{\sc Case A}}: \eqref{4.ecDelta}
has two hyperbolic solutions.
\item[\rm -] \hypertarget{CB}{{\sc Case B}}: \eqref{4.ecDelta}
has at least one bounded solution but no hyperbolic ones.
\item[\rm -] \hypertarget{CC}{{\sc Case C}}: \eqref{4.ecDelta}
has no bounded solutions.
\end{itemize}
According to Theorem~\ref{3.teorhyp}, \hyperlink{CA}{\sc Case A}
is equivalent to the existence of an attractor-repeller pair.
As in the quadratic case analyzed in~\cite{lno}, much more can be said in
any of the three situations if the next condition (assumed when indicated) holds
for the ``unperturbed" equation \eqref{3.ecu} (i.e., for $x'=f(t,x)$):
\begin{hipo}\label{4.hipo}
The equation $x'=f(t,x)$ has an attractor-repeller pair $(\wit a,\wit r)$.
\end{hipo}
This section is devoted to describe these ``extra"~properties,
which relate the dynamics of \eqref{4.ecDelta} to those
of its ``past" and ``future" instances,
\begin{eqnarray}
 x'&=f(t,x-\delta_-)\,,\label{4.ecu-}\\
 x'&=f(t,x-\delta_+)\,.\label{4.ecu+}
\end{eqnarray}
By {\em past\/} and {\em future instances\/} we mean that \eqref{4.ecDelta}
is approximated by \eqref{4.ecu-} when $-t$ is large enough and by \eqref{4.ecu+} when
$t$ is large enough.
Note that equation \eqref{4.ecDelta} can be understood as a transition from \eqref{4.ecu-} to
\eqref{4.ecu+} as time increases.
\par
The next technical result, based on Proposition \ref{3.proppersiste},
plays a fundamental role.
\begin{prop}\label{4.proppersiste}
Assume Hypothesis {\rm\ref{4.hipo}}, let $(k,\beta_f)$ be a common dichotomy
constant pair for $\wit a$ and $\wit r$, and fix $\beta\in(0,\beta_f)$. Given $\ep>0$ there exists
$\delta_\ep>0$ such that, if $\Delta$ is an
$L^\infty$ function with $\n{\Delta}_{L^\infty}<\delta_\ep$, then the equation
$x'=f(t,x-\Delta(t))$ has an attractor-repeler pair $(\wit a_\Delta,\wit r_\Delta)$
with common dichotomy constant pair $(k,\beta)$, and with
$\n{\wit a-\wit a_\Delta}_\infty\le\ep$ and $\n{\wit r-\wit r_\Delta}_\infty\le\ep$.
\end{prop}
\begin{proof}
We denote $g(t,x):=f(t,x-\Delta(t))$, so that $g_x(t,x)=f_x(t,x-\Delta(t))$.
We take $\rho>0$ and look for $m^{f_x}>0$ and $l^{f_x}>0$ such that
$\sup_{x\in[-\rho,\rho]} |f_x(t,\wit b(t)+x)|\le m^{f_x}$ and
$\sup_{x_1,x_2\in[-\rho,\rho],\,x_1\ne x_2}|f_x(t,\wit b(t)+x_1)-
f_x(t,\wit b(t)+x_1)|/|x_1-x_2|\le l^{f_x}$ $l$-a.a. for
$\wit b=\wit a$ and $\wit b=\wit r$.
Let $\delta^*=\delta^*(\beta,k_f,\rho,l^{f_x})$
be the constant determined by Proposition
\ref{3.proppersiste}, valid for $\wit b=\wit a$ and
$\wit b=\wit r$. Assume without restriction that $\ep\in(0,\delta^*]$.
We first impose $\n{\Delta}_{L^\infty}\le\rho/2$. Then,
\[
\begin{split}
 &\sup_{x_1,x_2\in[-\rho/2,\,\rho/2],\,x_1\ne x_2}\Frac{|g_x(t,\wit b(t)+x_1)-
 g_x(t,\wit b(t)+x_2)|}{|x_1-x_2|}\\
 &\quad\;
 =\sup_{x_1,x_2\in[-\rho/2,\,\rho/2],\,x_1\ne x_2}\Frac{|f_x(t,\wit b(t)+x_1-\Delta(t))-
 f_x(t,\wit b(t)+x_2-\Delta(t))|}{|x_1-x_2|}\le l^{f_x}\,,
\end{split}
\]
which shows that condition
4 of Proposition \ref{3.proppersiste} holds for $\rho/2$. We also impose
$\n{\Delta}_{L^\infty}\le\beta\ep/(2\,k_f\,m^{f_x})$, which yields
$\big|g(t,\wit b(t))-f(t,\wit b(t))\big|\le \beta\,\ep/(2k_f)$ $l$-a.a.:
condition 2 holds. Finally, we impose
$\n{\Delta}_{L^\infty}<(\beta_f-\beta)/(l^{f_x})$
to get condition 3 for $\rho/2$, namely
$\sup_{x\in[-\rho/2,\rho/2]}|f_x(t,\wit b(t)+x)-g_x(t,\wit b(t)+x)|<(\beta_f-\beta)$ $l$-a.a. The minimum of the three constants is $\delta_\ep$.
\end{proof}
\begin{nota}
Assume Hypothesis \ref{4.hipo} for $x'=f(t,x)$.
It is easy to check that, for any $d\in\R$, the equation $x'=f(t,x-d)$
also satisfies Hypothesis \ref{4.hipo}, with attractor-repeller pair
$(\wit a+d,\wit r+d)$;
and that the constant $\delta_\ep$ found in Proposition \ref{4.proppersiste}
is independent of the value of $d$ and hence valid for all these translated equations.
\end{nota}
In particular, under Hypothesis \ref{4.hipo},
the equations \eqref{4.ecu-} and \eqref{4.ecu+}
have respective attractor-repeller pairs $(\wit a +\delta_-,\wit r+\delta_-)$
and $(\wit a +\delta_+,\wit r+\delta_+)$. The already mentioned connections
among the dynamics of \eqref{4.ecDelta} and those of \eqref{4.ecu-} and \eqref{4.ecu+}
are described in the next result.
\begin{teor}\label{4.teorexist}
Assume Hypothesis~{\rm \ref{4.hipo}}. Then,
\begin{itemize}
\item[\rm(i)] there exist the functions $\ma_\Delta$ and $\mr_\Delta$
associated by Theorem~{\rm \ref{3.teoruno}} to \eqref{4.ecDelta}.
\end{itemize}
Let us represent by $x_\Delta(t,s,x_0)$ the solution with $x_\Delta(s,s,x_0)=x_0$. Then,
\begin{itemize}
\item[\rm(ii)] $\lim_{t\to-\infty}|\ma_\Delta(t)-(\wit a(t)+\delta_-)|=0$,
$\lim_{t\to-\infty}|x_\Delta(t,s,x_0)-(\wit r(t)+\delta_-)|=0$ whenever
$\ma_\Delta(s)$ exists and $x_0<\ma_\Delta(s)$,
$\lim_{t\to +\infty}|\mr_\Delta(t)-(\wit r(t)+\delta_+)|=0$, and
$\lim_{t\to +\infty}|x_\Delta(t,s,x_0)-(\wit a(t)+\delta_+)|=0$ whenever
$\mr_\Delta(s)$ exists and $x_0>\mr_\Delta(s)$.
\item[\rm(iii)] The solutions $\ma_\Delta$ and $\mr_\Delta$
are respectively locally pullback attractive and locally
pullback repulsive.
\item[\rm(iv)] $\ma_\Delta$ and $\mr_\Delta$ are uniformly separated if and only if
they are globally defined and different; i.e., if and only if
$(\wma_\Delta,\wmr_\Delta):=(\ma_\Delta,\mr_\Delta)$ is an attractor-repeller pair.
\item[\rm(v)] If the equation \eqref{4.ecDelta} has no hyperbolic solutions, then
it has at most one bounded solution $\ma_\Delta=\mr_\Delta$.
\end{itemize}
\end{teor}
\begin{proof}
The proof repeats step by step that of~\cite[Theorem 3.4]{lno}. The results used there
for the quadratic case have already been stated in this general concave case:
Proposition \ref{4.proppersiste} and Theorems \ref{3.teoruno} and \ref{3.teorhyp}.
\end{proof}
Let us summarize part of the information provided for \eqref{4.ecDelta}
by Theorem \ref{4.teorexist}
combined with Theorems \ref{3.teorlb*} and \ref{3.teorhyp}, under the assumptions
\hyperlink{f1}{\bf f1}-\hyperlink{f7}{\bf f7} and Hypothesis
\ref{4.hipo} on $f$, and for an $L^\infty$ function $\Delta$ with finite asymptotic
limits $\delta_\pm$.
\par
- \hyperlink{CA}{\sc Case A} holds for \eqref{4.ecDelta}
if and only if the equation has an attractor-repeller pair $(\wma_\Delta,\wmr_\Delta)$;
or, equivalently, if it has two different
bounded solutions. If so, the attractor-repeller pair connects
$(\wit a +\delta_-,\wit r +\delta_-)$ to $(\wit a +\delta_+,\wit a +\delta_+)$:
$\lim_{t\to\pm\infty}|\wma_\Delta(t)-(\wit a(t) +\delta_\pm)|=0$ and
$\lim_{t\to\pm\infty}|\wmr_\Delta(t)-(\wit r(t) +\delta_\pm)|=0$.
This situation is often called ({\em end-point}) {\em tracking}.
In addition, $\wma_\Delta(t)$ is the unique solution approaching $\wit a +\delta_-$ as time
decreases, and $\wmr_\Delta(t)$ is the unique solution approaching $\wit r +\delta_+$ as time
increases. Note also that \hyperlink{CA}{\sc Case A} holds if and only
$\lambda^*(f_\Delta)<0$: see Theorem \ref{3.teorlb*}, and recall that
$f_\Delta(t,x):=f(t,x-\Delta(t))$.
\par
- \hyperlink{CB}{\sc Case B} holds if and only if \eqref{4.ecDelta}
has a unique bounded solution $\mb_\Delta$. In this case, this
solution is locally pullback attractive and repulsive (see Subsection \ref{3.subsechyp}),
and it connects $\wit a +\delta_-$ to $\wit r +\delta_+$:
$\lim_{t\to-\infty}|\mb_\Delta(t)-(\wit a(t) +\delta_-)|=0$ and
$\lim_{t\to+\infty}|\mb_\Delta(t)-(\wit r(t) +\delta_+)|=0$.
And no other solution of \eqref{4.ecDelta} satisfies
any of these two properties. Note also that \hyperlink{CB}{\sc Case B}
holds if and only $\lambda^*(f_\Delta)=0$.
\par
- \hyperlink{CB}{\sc Case C} holds
if and only if the equation has no bounded solutions. In this case, there
exists a locally pullback attractive solution $\ma_\Delta$ which is the
unique solution bounded at $-\infty$ approaching $\wit a+\delta _-$
as time decreases (i.e., with
$\lim_{t\to-\infty}|\ma_\Delta(t)-(\wit a(t) +\delta_-)|=0$); and it exists a locally pullback
repulsive solution $\mr_\Delta\ne\ma_\Delta$ which is the unique solution
bounded at $+\infty$ approaching $\wit r+\delta_+$
as time increases (i.e., with $\lim_{t\to+\infty}|\mr_\Delta(t)-(\wit r(t) +\delta_+)|=0$).
Note also that \hyperlink{CC}{\sc Case C}
holds if and only $\lambda^*(f_\Delta)>0$.
This situation of loss of connection is sometimes called {\em tipping}.
\par
Some drawings showing the dynamical behavior in each one of these three cases
can be found in~\cite[Figures 1-6]{lnor} (with a typo: the
graphs of \hyperlink{CA}{\sc Cases A} and \hyperlink{CC}{C} are interchanged).
\section{Rigorous estimates of tracking and tipping for piecewise constant transitions}\label{5.sec}
From now, $f$ is assumed to satisfy conditions
\hyperlink{f1}{\bf f1}-\hyperlink{f7}{\bf f7} and
$\G\colon\R\to\R$ is a bounded and continuous function
such that there exist $\gamma_\pm:=\lim_{t\to\pm\infty}\G(t)$ in $\R$.
Observe that these hypotheses on $\G$ are stronger than those
assumed on $\Delta $ on Section \ref{4.sec}. We consider the differential equations
\begin{equation}\label{5.eqGamma}
 x'=f(t,x-\G(t))
\end{equation}
always assuming Hypothesis \ref{4.hipo} on $x'=f(t,x)$. Our global purpose is to determine
if the dynamical situation of \eqref{5.eqGamma} fits in
\hyperlink{CA}{\sc Case A}-tracking or \hyperlink{CC}{C}-tipping: see the
end of Section \ref{4.sec}.
Later on, we will be interested also in parametric perturbations of the type
$x'=f(t,x-d\G(ct))$, where $c>0$ and $d>0$: the parameter $c$ acts as the {\em rate\/}
of the transition $\G$, that is, of the change from the past
$d\gamma_-$ to the future $d\gamma_+$, and $d$ determines the
{\em size\/} of this change; and we will talk of {\em rate-induced
tipping} or {\em size-induced tipping} when a small change in $c$ or $d$ causes
the global dynamics to jump from \hyperlink{CA}{\sc Case A} to \hyperlink{CC}{\sc Case C}.
But, for our first results, we work
with fixed values of $c$ and $d$; that is, which a generic function $\G$
satisfying the assumed conditions.
\par
Our results for a continuous $\G$ as in \eqref{5.eqGamma} are, in fact, obtained
in Section \ref{6.sec}, just for the case in which $f$ is quadratic. The techniques
that we use require the discretization with respect to the time
of $\G$, as already presented in~\cite{lno}: the function $\G$
is replaced in \eqref{5.eqGamma} by a piecewise constant right-continuous
function which coincides with $\G$ at infinite equispaced points, which we describe now.
Let us define
\begin{equation}\label{5.defGch}
\begin{split}
 \G^h(t)&
 :=\G(jh)
 \quad\;\text{if }t\in\,[\,jh\,,(j+1)h\,)\text{ for }j\in\Z\text{ and }h>0\,,\\
 \G^0(t)&:=\G(t)
 \qquad \text{for }t\in\R\,,
\end{split}
\end{equation}
and consider the differential equations
\begin{equation}\label{5.ecutro}
 x'=f(t,x-\G^h(t))\,.
\end{equation}
We use \eqref{5.ecutro}$_h$ to refer to this equation for $h\ge 0$,
and we represent by $x_h(t,s,x_0)$ the maximal solution of
\eqref{5.ecutro} with value $x_0$ at $t=s$.
Our objective in this section is to determine the dynamical situation
for \eqref{5.ecutro}$_{h_0}$ for small $h_0>0$.
\begin{notas}\label{5.notatrans}
1.~Let us call $\gamma_{h,j}:=\G(jh)$ for $h>0$ and $j\in\Z$,
and consider
\begin{equation}\label{5.ecuchj}
 x'=f(t,x-\gamma_{h,j})\,.
\end{equation}
Recall once more that $x(t,s,x_0)$ is the maximal solution of the unperturbed equation
$x'=f(t,x)$ satisfying $x(s)=x_0$. Note that the solution of \eqref{5.ecuchj}$_{h,j}$ with value $x_0$ at $t=s$ is
$x(t,s,x_0-\gamma_{h,j})+\gamma_{h,j}$.
By comparing equations \eqref{5.ecutro}$_h$ and~\eqref{5.ecuchj}$_{h,j}$ we observe
that, if $h>0$ and $j\in\Z$, then
\begin{equation}\label{5.atrozos}
\begin{split}
 x_h(t,jh,x_0)=x(t,jh,x_0-\gamma_{h,j})+\gamma_{h,j}
\end{split}
\end{equation}
at those points $t$ of $[jh,(j+1)h]$ at which $x_h(t,jh,x_0)$ exists and
\begin{equation}\label{5.atrozos2}
\begin{split}
 x_h(t,jh,x_0)=x(t,jh,x_0-\gamma_{h,j-1})+\gamma_{h,j-1}
\end{split}
\end{equation}
at those points $t$ of $[(j-1)h,jh]$ at which $x_h(t,jh,x_0)$ exists.
\par
2.~Theorem~\ref{4.teorexist} ensures the existence
of the special solutions $\mah$ and $\mrh$ of \eqref{5.ecutro}$_h$
determined by Theorem \ref{3.teoruno} for
$h\ge 0$ under Hypothesis~\ref{4.hipo}.
\end{notas}
Let $(\wit a,\wit r)$ be the attractor-repeller pair provided by
Hypothesis \ref{4.hipo}. In the statements and proofs
of many of our results, the convex combinations
\begin{equation}\label{5.defbnu}
 b^\nu:=\nu\,\wit a+(1-\nu)\,\wit r
\end{equation}
for $\nu\in(0,1)$, as well as their properties, play a fundamental role.
The strict concavity assumption \hyperlink{f7}{\bf f7}
makes it easy to check that, if $\nu\in(0,1)$,
then $(b^\nu)'(t)\le f(t,b^\nu(t))$
for $l$-a.a.~$t\in\R$. Adapting the
argument of~\cite[Theorem 2]{olop}, we get
\begin{equation}\label{5.bcomp}
 x(t,s,b^\nu(s))> b^\nu(t) \quad\text{and}\quad
 x(s,t,b^\nu(t))< b^\nu(s)
 \quad\text{whenever $\;t>s$}\,.
\end{equation}
By adapting the arguments used in \cite[Proposition 4.3]{nono} (which require
to use the compactness of the set $\W(f)$ proved in Theorem \ref{A.teorcont}),
we can check that, given $\ep>0$ and $\nu\in(0,1)$, there exists
$\delta^\nu_\ep>0$ such that
\begin{equation}\label{5.defdelta}
 x(s+\ep,s,b^\nu(s))\ge b^\nu(s+\ep)+\delta^\nu_\ep
 \quad\text{and}\quad
 x(s-\ep,s,b^\nu(s))\le b^\nu(s-\ep)-\delta^\nu_\ep
\end{equation}
for all $s\in\R$.
Note also that the inequalities \eqref{5.bcomp} are equalities for $\nu=0,1$.
\par
More properties of these lower solutions are described in the next two technical
lemmas, required for the main results.
\begin{lema}\label{5.lemanu}
Take $\nu_1\le\nu_2$ in $(0,1)$. Then,
$h_{\nu_1,\nu_2}:=\inf\{h\ge 0\,|\;x(t,s,b^{\nu_1}(s))\ge b^{\nu_2}(t)$ if $t-s\ge h$\}
is well-defined
and, in addition, the map $\{(\nu_1,\nu_2)\in(0,1)^2\,|\;\nu_1\le\nu_2\}\to[0,\infty),
\;(\nu_1,\nu_2)\mapsto h_{\nu_1,\nu_2}$ is continuous, with $h_{\nu,\nu}=0$ for all $\nu\in(0,1)$.
\end{lema}
\begin{proof}
Let us check that the conditions \lq\lq$x(t,s,b^{\nu_1}(s))\ge b^{\nu_2}(t)$ if $t-s\ge h$" and
``$x(s+h,s,b^{\nu_1}(s))\ge b^{\nu_2}(s+h)$ for $s\in\R$" are equivalent: clearly, the first one
implies the second one; and, if the second one holds and $t\ge s+h$, then
$x(t,s,b^{\nu_1}(s))=x(t,s+h,x(s+h,s,b^{\nu_1}(s))\ge x(t,s+h,b^{\nu_2}(s+h))
\ge b^{\nu_2}(t)$, by \eqref{5.bcomp}.
\par
Theorem \ref{3.teorhyp}(i) ensures that, for certain $k\ge 1$ and $\beta>0$,
$x(s+h,s,b^{\nu_1}(s))\ge \wit a(s+h)-k\,(1-\nu_1)\,e^{-\beta h}\n{\wit a-\wit r}_\infty$ if $h\ge 0$.
Let us take a constant $h_0>0$ such that $k\,(1-\nu_1)\,e^{-\beta h_0}\n{\wit a-\wit r}_\infty
= \inf_{t\in\R} (\wit a(t)-b^{\nu_2}(t))\; (=
(1-\nu_2)\inf_{t\in\R} (\wit a(t)-\wit r(t)))$. Then,
$h_0\in\mI_{\nu_1,\nu_2}:=\{h\ge 0\,|\;x(s+h,s,b^{\nu_1}(s))\ge b^{\nu_2}(s+h)$ for $s\in\R\}$:
this set is nonempty (and a positive half line, as deduced from the equivalent
definition of the statement of the Lemma).
In particular, $h_{\nu_1,\nu_2}:=\inf\mI_{\nu_1,\nu_2}$ is well-defined. Clearly,
$h_{\nu,\nu}=0$ for all $\nu\in(0,1)$. Note also that
$h_{\nu_1,\nu_2}\le h_{\bar\nu_1,\bar\nu_2}$ if $\bar\nu_1\le\nu_1\le\nu_2\le\bar\nu_2$.
\par
Let us fix $\nu\in(0,1)$, $\ep>0$, and $\delta_\ep^\nu$ satisfying \ref{5.defdelta}.
We take $\nu_1<\nu$ and $\nu_2>\nu$, both in $(0,1)$ and with
$\nu-\nu_1$ and $\nu_2-\nu$ small enough to ensure that
$b^\nu-\delta^\nu_\ep\le b^{\nu_1}<b^\nu<b^{\nu_2}\le b^\nu+\delta^\nu_\ep$. Then,
\[
\begin{split}
 x(s+2\ep,s,b^{\nu_1}(s))&\ge x(s+2\ep,s,b^\nu(s)-\delta^\nu_\ep)\ge
 x(s+2\ep,s,x(s,s+\ep,b^\nu(s+\ep)))\\
 &=x(s+2\ep,s+\ep,b^\nu(s+\ep))\ge
 b^\nu(s+2\ep)+\delta_\ep^\nu\ge b^{\nu_2}(s+2\ep)
\end{split}
\]
for all $s\in\R$, which shows that $h_{\nu_1,\nu_2}\le 2\ep$.
This ensures the continuity of $(\nu_1,\nu_2)\mapsto h_{\nu_1,\nu_2}$ at the points
$(\nu,\nu)$.
\par
Let us now take $\nu_1<\nu_2$ (always in $(0,1)$)
and a sequence $((\nu_{1,n},\nu_{2,n}))$ with limit
$(\nu_1,\nu_2)$. There is no restriction in assuming $\nu_0\le\nu_{1,n}<\nu_{2,n}\le1-\nu_0$ for all
$n\in\N$ for an $\nu_0\in(0,1/2]$. Then, $h_n:=h_{\nu_{1,n},\nu_{2,n}}\le h_{\nu_0,1-\nu_0}$.
Hence, a given subsequence of $(h_n)$ has a convergent subsequence, say
$(h_k)$, with limit $h_*$. The goal is checking that $h_*=h_{\nu_1,\nu_2}$.
Taking limit as $k\to\infty$ in the inequality
$x(s+h_k,s,b^{\nu_{1,k}}(s))\ge b^{\nu_{2,k}}(s+h_k)$ we get
$x(s+h_*,s,b^{\nu_1}(s))\ge b^{\nu_2}(s+h_*)$ for all $s\in\R$, which implies
$h_*\ge h_{\nu_1,\nu_2}$. Now we take $\ep>0$ and associate $\delta^{\nu_2}_\ep$
to $\ep$ and $\nu_2$ as in \eqref{5.defdelta}. Then,
\[
\begin{split}
 &x(s+h_{\nu_1,\nu_2}+\ep,s,b^{\nu_1}(s))=x(s+h_{\nu_1,\nu_2}+\ep,s+h_{\nu_1,\nu_2},
 x(s+h_{\nu_1,\nu_2},s,b^{\nu_1}(s)))\\
 &\qquad\qquad\ge x(s+h_{\nu_1,\nu_2}+\ep,s+h_{\nu_1,\nu_2},b^{\nu_2}(s))
 \ge b^{\nu_2}(s+h_{\nu_1,\nu_2}+\ep)+\delta_\ep^{\nu_2}
\end{split}
\]
for all $s\in\R$. This ensures that, if $k$ is large enough, then
$x(s+h_{\nu_1,\nu_2}+\ep,s,b^{\nu_{1,k}}(s))\ge b^{\nu_{2,k}}(s+h_{\nu_1,\nu_2}+\ep)$
for all $s\in\R$, so that $h_{\nu_1,\nu_2}+\ep\ge h_k$ and hence
$h_{\nu_1,\nu_2}+\ep\ge h_*$. Since this is valid for all $\ep>0$,
$h_{\nu_1,\nu_2}\ge h_*$. The proof is complete.
\end{proof}
\begin{lema}\label{5.lemat*}
Assume Hypothesis {\rm\ref{4.hipo}} and fix $\nu_1,\nu_2\in(0,1)$ and $h\ge 0$.
Then, there exist $t^i_*=t_*^i(\nu_i,h)\in\mathbb R$ for $i=1,2$
such that $\mah(t)-\G(t)>b^{\nu_1}(t)$
if $t\le t^1_*$ and $\mrh(t)-\G(t)<b^{\nu_2}(t)$ if $t\ge t^2_*$.
\end{lema}
\begin{proof}
We define
$\ep_{\nu_1}:=(1-\nu_1)\inf_{t\in\R}(\wit a(t)-\wit r(t))$ and
$\ep_{\nu_2}:=\nu_2\inf_{t\in\R}(\wit a(t)-\wit r(t))$.
Let us write
\[
 \mah(t)-\G(t)=(\mah(t)-(\wit a(t)+\gamma_-))+(\gamma_--\G(t))+(1-\nu_1)\,(\wit a(t)-\wit r(t))+b^{\nu_1}(t)\,.
\]
Theorem \ref{4.teorexist}(ii) and the definition of $\gamma_-$ provide $t_*^1$ such that
$\mah(t)-(\wit a(t)+\gamma_-)>-\ep_{\nu_1}/2$ and $\gamma_--\G(t)>-\ep_{\nu_1}/2$,
which shows the first assertion for all $t\le t_*^1$. Similarly, we get
$t_*^2\in\R$ such that, if $t\ge t_*^2$, then
\[
 \mrh(t)-\G(t)=(\mrh(t)-(\wit r(t)+\gamma_+))+(\gamma_+-\G(t))-\nu_2(\wit a(t)-\wit r(t))+b^{\nu_2}(t)<b^{\nu_2}(t)\,,
\]
which completes the proof.
\end{proof}
In what follows, we will develop some criteria guaranteing
\hyperlink{CA}{\sc Cases A} or \hyperlink{CC}{\sc C} for certain values
of $h$, always under Hypothesis \ref{4.hipo}.
The fundamental idea is given by Remark \ref{3.notabasta}: to get solutions
$\mah$ and $\mrh$ respectively defined on $(-\infty,t_0]$ and $[t_0,\infty)$, and
to compare their values at $t_0$. The information provided by
Remark \ref{5.notatrans}.2, regarding the existence of these solutions under
Hypothesis \ref{4.hipo}, will be constantly used, without further reference.
\subsection{Criteria for tracking}
Theorem \ref{5.teortrack} and Corollaries \ref{5.corotrack1} and \ref{5.corotrack2}
establish criteria for tracking. Lemma \ref{5.lemat*} shows how to accomplish
some of the hypotheses that the first two results require: the existence of $j_0$ and $n\ge 1$.
\begin{teor}\label{5.teortrack}
Assume Hypothesis {\rm\ref{4.hipo}} and fix $h>0$.
Assume the existence of $j_0\in\Z$, $n\in\N$ and
$\nu_{j_0},\nu_{j_0+1},\cdots,\nu_{j_0+n}$ in $[0,1]$ such that
$\mah(j_0h)-\G(j_0h)\ge b^{\nu_{j_0}}(j_0h)$, $\mrh((j_0+n)h)-\G((j_0+n)h)
<b^{\nu_{j_0+n}}((j_0+n)h)$, and
\begin{equation}\label{5.chaintrack}
 x((j+1)h,jh,b^{\nu_j}(jh))
 -b^{\nu_{j+1}}\!((j+1)h)\ge \G((j+1)h)-\G(jh)
\end{equation}
for $j\in\{j_0,\ldots,j_0+n-1\}$.
Then, \eqref{5.ecutro}$_h$ has an attractor-repeller pair.
\end{teor}
\begin{proof}
Let us call $\gamma_j:=\Gamma(jh)$. By hypothesis,
\begin{equation}\label{5.induc1}
 \mah(jh)-\gamma_j\ge b^{\nu_j}(jh)
\end{equation}
for $j=j_0$. We will prove the following assertion: if, for a $j\in\{j_0,\ldots,j_0+n-1\}$,
the map $\mah$ exists on $(-\infty,jh]$ and \eqref{5.induc1} holds, then
$\mah$ exists on $(-\infty,(j+1)h]$ and in addition~\eqref{5.induc1} holds
for $j+1$ instead of $j$.
\par
So, we assume these hypotheses for $j$.
As explained in Remark \ref{5.notatrans}.1,
at those points $t$ of $[jh,(j+1)h]$ at which $\mah(t)$ exists,
\begin{equation}\label{5.para5}
\begin{split}
 &\mah(t)=x(t,jh,\mah(jh)-\gamma_j)+\gamma_j
 \ge x(t,jh,b^{\nu_j}(jh))+\gamma_j\ge b^{\nu_j}(t)+\gamma_j\,.
\end{split}
\end{equation}
We have used \eqref{5.induc1} and \eqref{5.bcomp}
in the last two inequalities. Therefore,
$\mah(t)$ exists on $[jh,(j+1)h]$ (since it is bounded from below), and
\[
 \mah((j+1)h)\ge x((j+1)h,jh,b^{\nu_j}(jh))+\gamma_j
 \ge b^{\nu_{j+1}}((j+1)h)+\gamma_{j+1}\,.
\]
We have used the bound in the statement.
Our assertion is proved. In particular, it ensures the existence of
$\mah$ on $(-\infty,(j_0+n)h]$,
and \eqref{5.induc1} for $j=j_0+n$ yields
\[
 \mah((j_0+n)h)\ge b^{\nu_{j_0+n}}((j_0+n)h)+\gamma_{j_0+n}>\mrh((j_0+n)h)\,.
\]
This inequality and Theorem \ref{4.teorexist}(iv)
ensure that $(\mah,\mrh)$ is an attractor-repeller
pair for \eqref{5.ecutro}$_h$
(see Remark~\ref{3.notabasta}). The proof is complete.
\end{proof}
To successfully apply Theorem \ref{5.teortrack} for a particular problem requires an
adequate choice of the parameters $\nu_{j_0},\nu_{j_0+1},\cdots,\nu_{j_0+n}\in[0,1]$.
The next corollaries show that a quite simple choice suffices to
show the existence of attractor-repeller pair in some interesting situations.
The first one extends \cite[Proposition 3.12(ii)]{lno},
which was proved with different arguments for a much less general setting.
A {\em nonincreasing function} $\G$ satisfies
$\G(t_1)\ge\G(t_2)$ if $t_1<t_2$.
\par
\begin{coro}\label{5.corotrackdecr}
Assume Hypothesis~{\rm\ref{4.hipo}}. If $\Gamma$ is nonincreasing,
then \eqref{5.ecutro}$_h$ has an attractor-repeller pair for all $h>0$.
\end{coro}
\begin{proof}
We look for $j_0\in\Z$ and $n\in\N$ such that
$\mah(j_0h)-\G(j_0h)\ge b^{1/2}(j_0h)$, $\mrh((j_0+n)h)-\G((j_0+n)h)
<b^{1/2}((j_0+n)h)$, take $\nu_{j_0}=\nu_{j_0+1}=\ldots=\nu_{j_0+n}=1/2$,
and observe that \eqref{5.bcomp} and the nonincreasing character of $\G$
yield \eqref{5.chaintrack} for
$j\in\{j_0,\ldots,j_0+n-1\}$. Theorem \ref{5.teortrack} proves the assertion.
\end{proof}
\begin{coro}\label{5.corotrack1}
Assume Hypothesis~{\rm\ref{4.hipo}}, and take $\nu_1<\nu_2$ in $(0,1)$ and
$h\ge h_{\nu_1,\nu_2}$ with $h_{\nu_1,\nu_2}$ provided by Lemma~{\rm\ref{5.lemanu}}.
Look for $j_0\in\Z$ and $n\in\N$ such that
$\mah(j_0h)-\G(j_0h)\ge b^{\nu_1}(j_0h)$ and $\mrh((j_0+n)h)-\G((j_0+n)h)<b^{\nu_2}((j_0+n)h)$.
Assume that
\begin{equation}\label{5.hipjcoro}
 (\nu_2-\nu_1)(\wit a((j+1)h)-\wit r((j+1)h))\ge\G((j+1)h)-\G(jh)
\end{equation}
for $j\in\{j_0,\ldots,j_0+n-1\}$. Then,
\eqref{5.ecutro}$_h$ has an attractor-repeller pair.
\par
In addition, if $n\ge 2$, $\mrh((j_0+n)h)-\G((j_0+n-1)h)<b^{\nu_2}((j_0+n)h)$ and
\eqref{5.hipjcoro} holds for $j\in\{j_0,\ldots,j_0+n-2\}$, then
\eqref{5.ecutro}$_h$ has an attractor-repeller pair.
\end{coro}
\begin{proof}
We call $\gamma_j:=\G(jh)$. Observe that, since
$b^{\nu_2}-b^{\nu_1}=(\nu_2-\nu_1)(\wit a-\wit r)$,
the choice of $h$ (see Lemma \ref{5.lemanu}) ensures that
\[
 x((j+1)h,jh,b^{\nu_1}(jh))+\gamma_j\ge b^{\nu_2}((j+1)h)+\gamma_j\ge b^{\nu_1}((j+1)h)+\gamma_{j+1}
\]
as long as $j$ satisfies \eqref{5.hipjcoro}. Hence the first assertion follows from
Theorem~\ref{5.teortrack}, since all its hypotheses are fulfilled if we take
$\nu_j=\nu_1$ for $j\in\{j_0,\ldots,(j_0+n-1)\}$. When \eqref{5.hipjcoro} is
satisfied for $j\in\{j_0,j_0+1,\ldots,j_0+n-2\}$, we can repeat the proof of
Theorem~\ref{5.teortrack} applying \eqref{5.hipjcoro}
to get $\ma_h((j_0+n-1)h)\ge b^{\nu_1}((j_0+n-1)h)+\gamma_{j_0+n-1}$. By \eqref{5.atrozos},
$\ma_h((j_0+n)h)\ge x((j_0+n)h,j_0+n,b^{\nu_1}((j_0+n-1)h))+\gamma_{j_0+n-1}\ge
b^{\nu_2}((j_0+n)h)+\gamma_{j_0+n-1}$. The last hypothesis yields
$\ma_h((j_0+n)h)>\mr_h((j_0+n)h)$ which, according to
Remark~\ref{3.notabasta} and Theorem \ref{4.teorexist}(iv), shows the assertion.
\end{proof}
\begin{nota}\label{5.notah}
We point out that, for the previous result,
we do not need all the information provided by Lemma \ref{5.lemanu}: we just need
$x((j+1)h,jh,b^{\nu_1}(jh))\ge b^{\nu_2}((j+1)h)$ for $j\in\{j_0,\ldots,j_0+n-1\}$.
We will use this fact in Subsection \ref{5.subsecnum}.
\end{nota}
\begin{coro}\label{5.corotrack2}
Assume Hypothesis~{\rm\ref{4.hipo}} and that $\wit a(0)-\wit r(0)>\G(0)-\gamma_-$.
\begin{itemize}
\item[\rm(i)] If $\liminf_{t\to\infty}(\wit a(t)-\wit r(t))>\gamma_+-\G(0)$,
then there exists $h_0>0$ such that \eqref{5.ecutro}$_h$ has an attractor-repeller
pair for every $h\ge h_0$.
\item[\rm(ii)] If $\limsup_{t\to\infty}(\wit a(t)-\wit r(t))>\gamma_+-\G(0)$,
then there exists a sequence $(h_j)\uparrow\infty$ such that
\eqref{5.ecutro}$_{h_j}$ has an attractor-repeller
pair for every $n\in\N$.
\end{itemize}
\end{coro}
\begin{proof}
Let us check (i). We observe that the hypotheses ensure that there exists
$h_*>0$ such that, if $h\ge h_*$, then
\begin{equation}\label{5.hh}
 (\wit a(0)-\wit r(0))/2>\G(0)-\G(-h)\quad\text{and}\quad (\wit a(h)-\wit r(h))/2>\G(h)-\G(0)\,.
\end{equation}
In addition, Lemma \ref{5.lemat*} provides $h^*>0$ large enough to ensure that
\begin{equation}\label{5.hhh}
\mah(-h)-\G(-h)> b^{1/4}(-h) \quad\text{and}\quad
\mrh(h)-\G(h)< b^{1/4}(h)
\end{equation}
if $h\ge h^*$. Let $h_{1/4,3/4}$ be provided by Lemma~{\rm\ref{5.lemanu}}.
Assertion (i) holds for $h_0:=\max(h_{1/4,3/4},h_*,h^*)$, as ensured by Corollary \ref{5.corotrack1}
for $\nu_1=1/4$, $\nu_2=3/4$, $j_0=-1$, and $n=2$.
\par
\par
To prove (ii), we observe that, under its hypothesis, there exists
a sequence $(h_j)\uparrow\infty$ with $h_j\ge h_0$
such that \eqref{5.hh} and \eqref{5.hhh} hold for all $h=h_j$.
\end{proof}
\subsection{Criteria for tipping}
Our next goal is to establish criteria for tipping, similar in their formulations
to the previous results, and valid for functions $\G$ which increase with $t$.
The next lemma will play a fundamental role. By {\em nondecreasing function $\G$},
we mean $\G(t_1)\le\G(t_2)$ if $t_1<t_2$.
\begin{lema}\label{5.lemainc}
Assume that Hypothesis~{\rm\ref{4.hipo}} hold. Fix $h>0$, and
assume also that $\G(t)$ is nondecreasing. Then,
\begin{itemize}
\item[\rm(i)] if $\mah$ is defined on $(-\infty,t]$
and $t\in[jh,(j+1)h)]$ for $j\in\Z$, then
\[
 \mah(t)\le\wit a(t)+\G(jh)\,.
\]
\item[\rm(ii)] If $\mrh$ is defined on $[t,\infty)$
and $t\in[jh,(j+1)h)]$ for $j\in\Z$, then
\[
 \mrh(t)\ge\wit r(t)+\G(jh)\,.
\]
\end{itemize}
\end{lema}
\begin{proof}
Let us prove (i).
We call $\gamma_j:=\G(jh)$ and $\gamma_\pm:=\Gamma(\pm\infty)$,
and fix $\rho>0$. Since $\inf_{t\in\R}(\wit a(t)-\wit r(t))>0$,
Theorem~\ref{4.teorexist}(ii) ensures the existence of
$j_\rho$ such that $\mah$ is defined on $(-\infty,j_\rho h]$ and
satisfies $\mah(t)-\gamma_-\le (1+\rho)\,\wit a(t)-\rho\,\wit r(t)$
for all $t\le j_\rho h$; and hence, since $(\gamma_j)\downarrow\gamma_-$ as $j\to-\infty$,
for any $j\le j_\rho-1$,
\begin{equation}\label{5.igutodo}
 \mah(t)-\gamma_j\le (1+\rho)\,\wit a(t)-\rho\,\wit r(t)
 \qquad\text{if $\;t\in[jh,(j+1)h]$}\,.
\end{equation}
Let us assume that~\eqref{5.igutodo} holds for a certain
$j-1\in\Z$ and prove it for $j$, assuming that $\mah$ is defined
on $(-\infty,t]$ for $t\in[jh,(j+1)h]$. Note first that our induction hypothesis
and the nondecreasing character of $\G$ yield
$\mah(jh)-\gamma_j\le\mah(jh)-\gamma_{j-1}
\le (1+\rho)\,\wit a(jh)-\rho\,\wit r(jh)$.
Hence, \eqref{5.atrozos} for $t\in[jh, (j+1)h]$
combined with the concavity of $x_0\mapsto x(t,s,x_0)$ yields
\[
\begin{split}
 \mah(t)&=x(t,jh,\mah(jh)-\gamma_j)+\gamma_j\\
 &\le x(t,jh,(1+\rho)\,\wit a(jh)-\rho\,\wit r(jh))+\gamma_j
 \le (1+\rho)\,\wit a(t)-\rho\,\wit r(t)+\gamma_j\,.
\end{split}
\]
Hence, \eqref{5.igutodo} holds as long as $\mah$ is defined. Taking
limit as $\rho\to 0$ proves (i).
\par
To prove (ii), we substitute~\eqref{5.igutodo} by
\[
 \mrh(t)-\gamma_j\ge (1+\rho)\,\wit r(t)-\rho\,\wit a(t)\,.
\]
for $j\ge j_\rho$ (for a possibly different $j_\rho$)
and $t\in[jh,(j+1)h]$. In order to prove it for any
$j\in\Z$, we assume that it is true for a certain $j$, and
that $\mrh$ is defined on $[t,\infty)$ for
a $t\in[(j-1)h,jh]$. Using \eqref{5.atrozos2}, we get
\[
\begin{split}
 \mrh(t)&=x(t,jh,\mrh(jh)-\gamma_{j-1})+\gamma_{j-1}
 \ge x(t,jh,\mrh(jh)-\gamma_j)+\gamma_{j-1}\\
 &\ge x(t,jh,(1+\rho)\,\wit r(jh)-\rho\,\wit a(jh))+\gamma_{j-1}
 \ge (1+\rho)\,\wit a(t)-\rho\,\wit r(t)+\gamma_{j-1}\,.
\end{split}
\]
So, we can conclude the proof of (ii) as that of (i).
\end{proof}
Recall that the functions $b^\nu$ of the next statement are defined by \eqref{5.defbnu}.
\begin{teor}\label{5.teortipping}
Assume Hypothesis {\rm\ref{4.hipo}}. Fix $h>0$, and
assume also that $\G(t)$ is nondecreasing,
and the existence of $j_0\in\Z$,
$n\in\N$ and $\nu_{j_0}=1,\nu_{j_0+1},\cdots,\nu_{j_0+n}=0$
in $[0,1]$ such that
\begin{equation}\label{5.hipj}
 x((j+1)h,jh,b^{\nu_j}(jh))
 -b^{\nu_{j+1}}\!((j+1)h)\le\G((j+1)h)-\G(jh)
\end{equation}
for $j\in\{j_0,\ldots,j_0+n-1\}$.
Then, \eqref{5.ecutro}$_h$ does not have an attractor-repeller
pair; and it has no bounded
solutions if, in addition, one of the $n$ inequalities \eqref{5.hipj} is strict.
\end{teor}
\begin{proof}
Let us denote $\gamma_j:=\G(jh)$ for $j\in\Z$.
We assume that $\mah$ is defined on
$(-\infty, (j_0+n)h]$ and $\mrh$ is defined on
$[(j_0+n)h,\infty)$: otherwise, there is nothing to prove
(see Theorem \ref{4.teorexist}(iv)).
We will check that
\begin{equation}\label{5.desas}
 \mah(jh)\le b^{\nu_j}(jh)+\gamma_{j}
\end{equation}
for $j\in\{j_0,\ldots,j_0+n\}$. Lemma~\ref{5.lemainc}(i)
ensures that $\mah(j_0h)\le\wit a(j_0h)+\gamma_{j_0}= b^{\nu_{j_0}}(j_0h)+\gamma_{j_0}$.
This shows \eqref{5.desas}
for $j=j_0$. We assume that it is true for a $j\in\{j_0,\ldots,j_0+n-1\}$ and will prove it
for $j+1$.
According to \eqref{5.atrozos} and using \eqref{5.desas} and \eqref{5.hipj},
\[
\begin{split}
 \mah((j+1)h)&=x((j+1)h,jh,\mah(jh)-\gamma_j)+\gamma_j\\
 &\le x((j+1)h,jh,b^{\nu_j}(jh))+\gamma_j
 \le b^{\nu_{j+1}}((j+1)h)+\gamma_{j+1}\,,
\end{split}
\]
as asserted. In particular, \eqref{5.desas} for $j=j_0+n$ and Lemma \ref{5.lemainc}(ii) yield
\[
 \mah((j_0+n)h)\le b^{\nu_{j_0+n}}((j_0+n)h)+\gamma_{j_0+n}= \wit r((j_0+n)h)+\gamma_{j_0+n}
 \le\mrh((j_0+n)h)\,.
\]
This precludes the existence of attractor-repeller pair
(see Remark \ref{3.notabasta}).
\par
By reviewing the proof we observe that the strict character of
one of the inequalities \eqref{5.hipj} ensures that
$\mah((j_0+n)h)<\mrh((j_0+n)h)$, and this precludes
the existence of bounded solutions (see again Remark \ref{3.notabasta}).
The proof is complete.
\end{proof}
\begin{teor}\label{5.corotipping}
Assume Hypothesis~{\rm\ref{4.hipo}}. Fix $h>0$, and
assume that $\G(t)$ is nondecreasing.
Assume also one of the next three conditions:
\begin{itemize}[itemsep=2pt]
\item[\rm(i)] (A one-step criterion for tipping.) There is $j_0\in\Z$ such that
\[
 \wit a((j_0+1)h)-\wit r((j_0+1)h)\le\G((j_0+1)h)-\G(j_0h)\,.
\]
\item[\rm(ii)] (A first two-steps criterion for tipping.)
There are $j_0\in\Z$ and $\nu\in[0,1]$ such that
\[
\begin{split}
  (1-\nu)\,\big(\wit a((j_0+1)h)-\wit r((j_0+1)h)\big)
  &\le \G((j_0+1)h)-\G(j_0h)\,,\quad \text{and}\\
  x\big((j_0+2)h,(j_0+1)h,b^\nu((j_0+1)h)\big)&-\wit r((j_0+2)h)\\
  &\le \G((j_0+2)h)-\G((j_0+1)h)\,.
\end{split}
\]
\item[\rm(iii)] (A second two-steps criterion for tipping.) There is $j_0\in\Z$ such that
\[
\begin{split}
  &\qquad x\big((j_0+2)h,(j_0+1)h,\wit a((j_0+1)h)+\G(j_0h)-\G((j_0+1)h)\big)\\
  &\qquad\qquad\qquad\qquad\qquad -\wit r((j_0+2)h)\le \G((j_0+2)h)-\G((j_0+1)h)\,.
\end{split}
\]
\item[\rm(iv)] (A several-steps criterion for tipping.) There are $j_0\in\Z$ and $n\in\N$ such that
\[
 \qquad x_h((j_0+n)h,j_0h,\wit a(j_0h)+\G(j_0h))\le\widetilde r(j_0+n)h+\G((j_0+n)h)\,.
\]
\end{itemize}
Then, \eqref{5.ecutro}$_h$ does not have an
attractor-repeller pair. If, in addition, the inequality in {\rm(i)}, in {\rm(iii)}, in
{\rm (iv)}, or one of those in {\rm (ii)} is strict, then \eqref{5.ecutro}$_h$ has no bounded
solutions.
\end{teor}
\begin{proof}
Let us call $\gamma_j=\G(jh)$. Note that there is no restriction in
assuming that $\mah$ is defined on $(-\infty,(j_0+2)h]$ and $\mrh$ is defined on
$[j_0h,\infty)$, since otherwise there is nothing to prove (see Theorem \ref{4.teorexist}(iv)).
Assertion (i) is a trivial consequence of
Theorem \ref{5.teortipping}: we have $n=1$, $\nu_{j_0}=1$ and $\nu_{j_0+1}=0$,
and the unique inequality \eqref{5.hipj} is exactly that in (i).
Assertion (ii) also follows from Theorem \ref{5.teortipping}, with $n=2$,
$\nu_{j_0}=1$, $\nu_{j_0+1}=\nu$ and $\nu_{j_0+2}=0$: the two
inequalities \eqref{5.hipj} are those of (ii).
To prove (iii), we use \eqref{5.atrozos} and combine it with Lemma \ref{5.lemainc}(i),
the assumed inequality and Lemma \ref{5.lemainc}(ii) to check that
\[
\begin{split}
 \mah((j_0+2)h)
 &\le x\big((j_0+2)h,(j_0+1)h,\wit a((j_0+1)h)+\gamma_{j_0}-\gamma_{j_0+1}\big)+\gamma_{j_0+1}\\
 &\le \wit r((j_0+2)h)+\gamma_{j_0+2}\le\mrh((j_0+2)h)\,.
\end{split}
\]
According to Remark \ref{3.notabasta}, this inequality means the existence of at most
one bounded solution for \eqref{5.ecutro}$_h$. The same happens under the hypotheses in (iv),
which combined with  Lemma \ref{5.lemainc} yields
$\ma_h((j_0+n)h)\le\mr_h((j_0+n)h)$. The last assertion is checked
by reviewing the proof under the additional hypothesis, and using again
the information of Remark \ref{3.notabasta}.
\end{proof}
\subsection{Numerical evidence}\label{5.subsecnum}
We close the section by numerically validating some
of the criteria for tracking and tipping developed above.
Our benchmark will be a concave quadratic problem
\begin{equation}\label{eq:quad_concave}
x'=-\big(x-\Gamma^h(ct)\big)^2+p(t), \quad x\in\R,\; h\in\N,\; c>0
\end{equation}
for $\Gamma^h(t)$ defined as in \eqref{5.defGch} from
$\G(t):=(2/\pi)\arctan(t)$ and $p(t):=0.962-\sin(t/2)-\sin(\sqrt{5}\,t)$.
Fixing a value of $h\ge 0$ and letting $c$ vary, we
analyze the possible occurrence of rate-induced tipping.
It is proved in~\cite[Theorem 5.3]{lno} that the bifurcation map
$\lambda_*:[0,\infty]\times[0,\infty)$, $(c,h)\mapsto\lambda_*(c,h)$, with
$\lambda_*(c,h)$ associated to \eqref{eq:quad_concave} by Theorem \ref{3.teorlb*},
is bounded and continuous. Recall that the sign of
$\lambda_*(c,h)$ determines the dynamical situation of \eqref{eq:quad_concave}.
Specifically, we will highlight the pairs $(c,h)$ for which
the criteria of Corollary \ref{5.corotrack1} and Theorem \ref{5.corotipping}
allow us to correctly identify tracking or tipping.
The results are appreciable in Figure \ref{fig:num_track_tipp_sec5},
where the surface given by the numerical approximation of $\lambda_*(c,h)$ for $c,h\in[0,6]$
is complemented by chromatic
dots corresponding to the pairs $(c,h)$ where our criteria rigorously guarantee tracking (in green),
and tipping in one, two or more steps (in magenta, red and orange, respectively).
\begin{figure}
    \centering
    \includegraphics[width=0.48\textwidth,trim={1.3cm 0.6cm 1.2cm 1cm},clip]{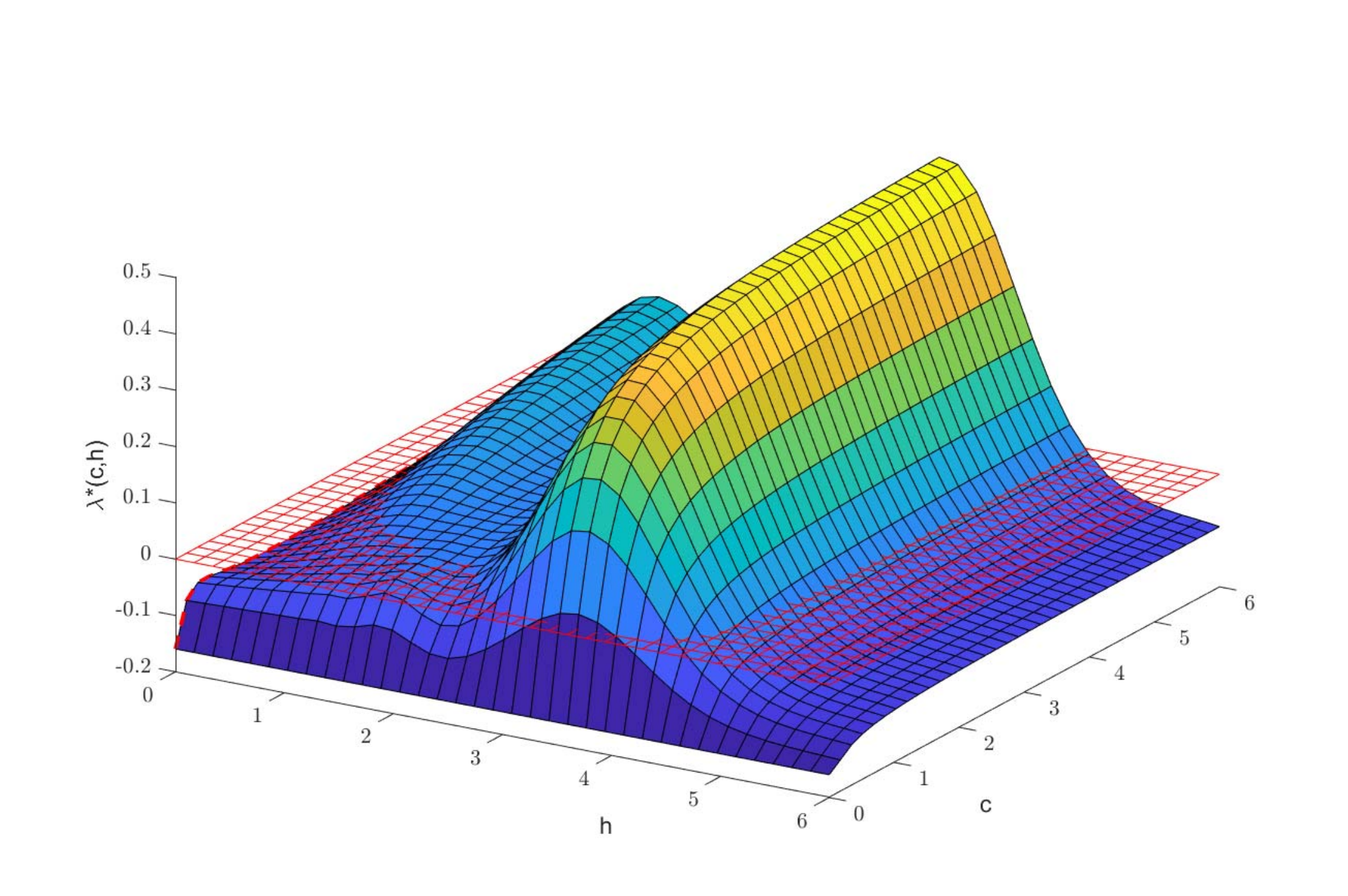}
    \includegraphics[width=0.48\textwidth,trim={1.3cm 0.6cm 1.2cm 1cm},clip]{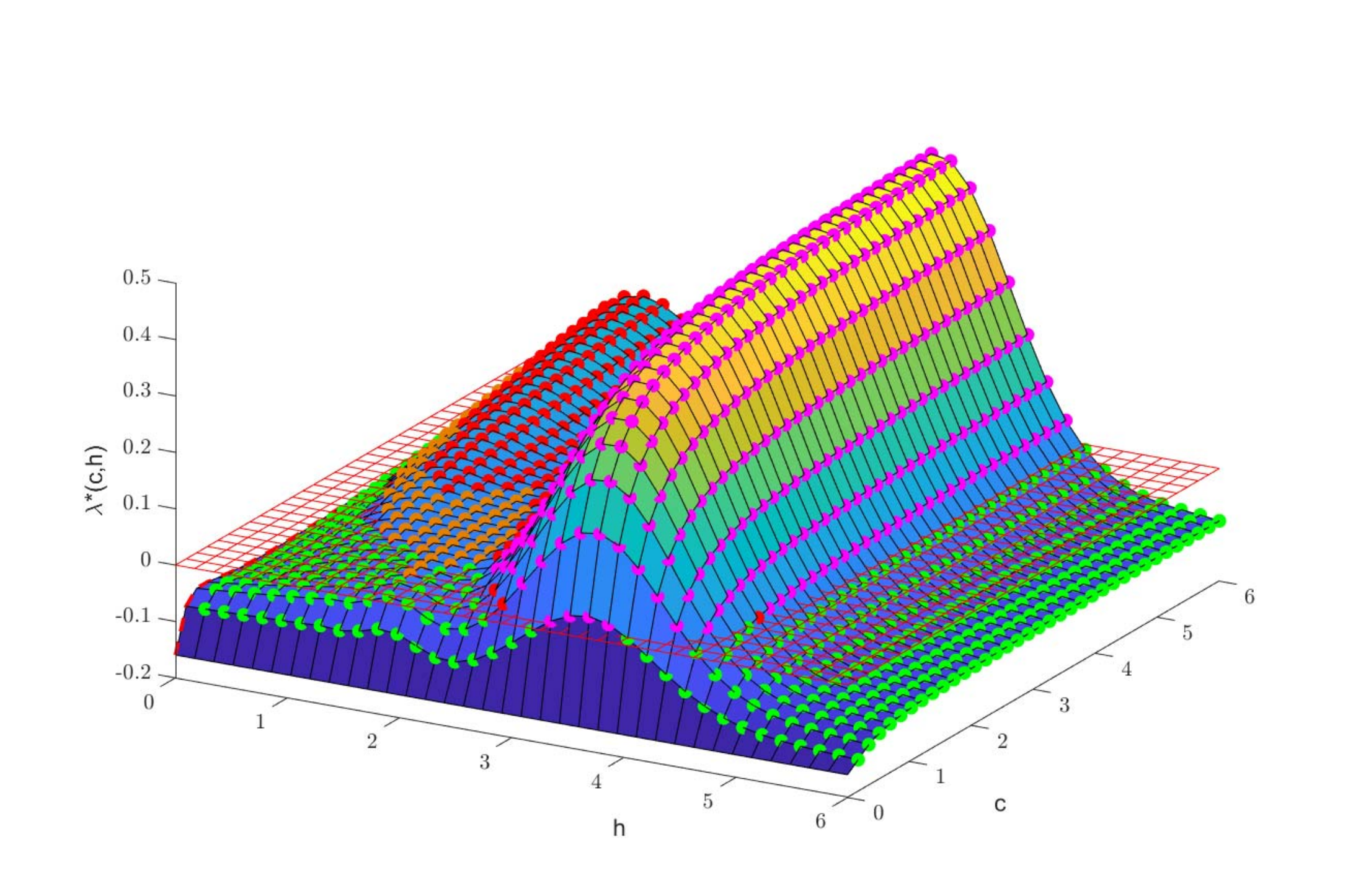}
    \includegraphics[width=0.48\textwidth,trim={1.3cm 0.6cm 1.2cm 1cm},clip]{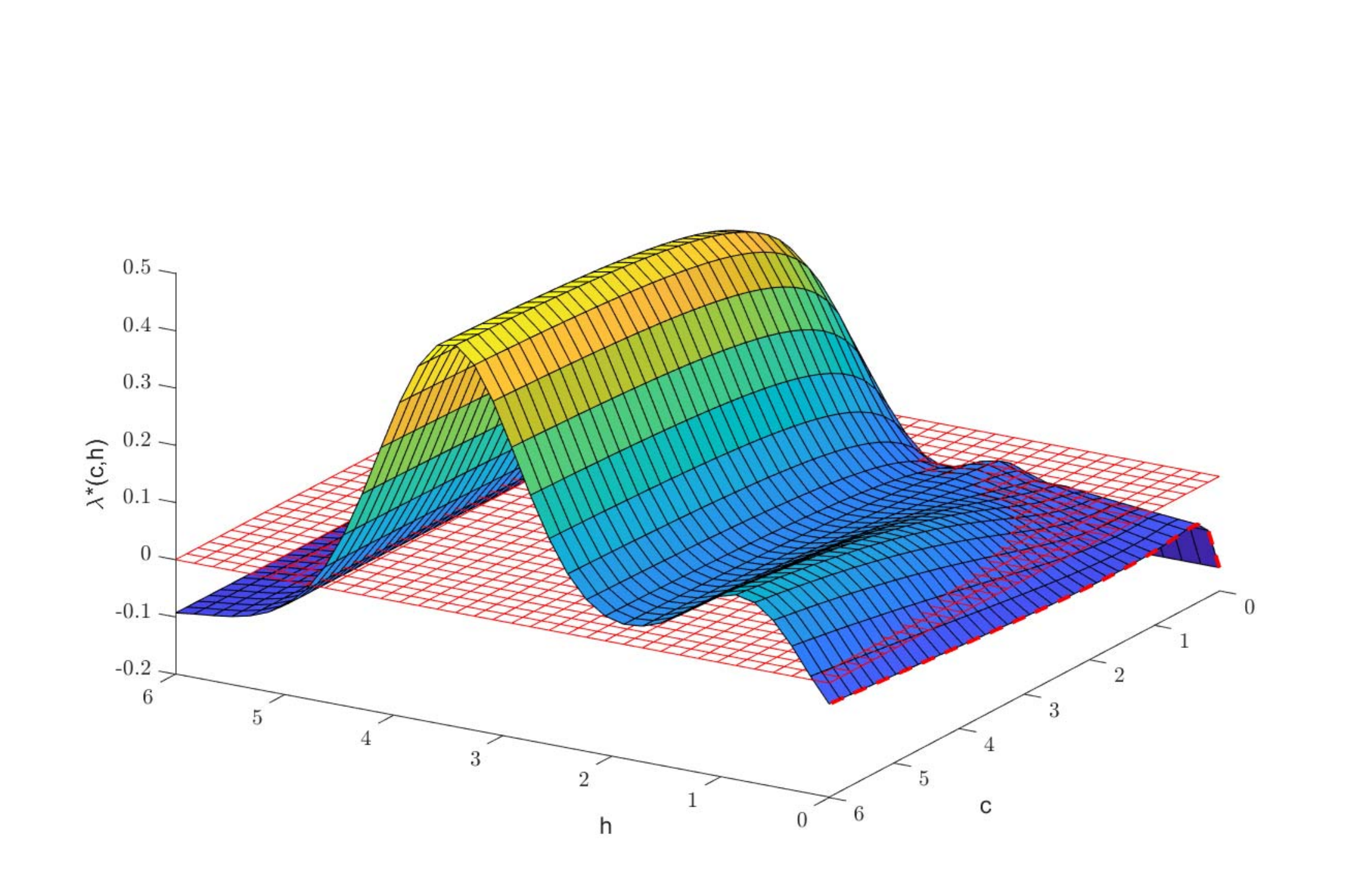}
    \includegraphics[width=0.48\textwidth,trim={1.3cm 0.6cm 1.2cm 1cm},clip]{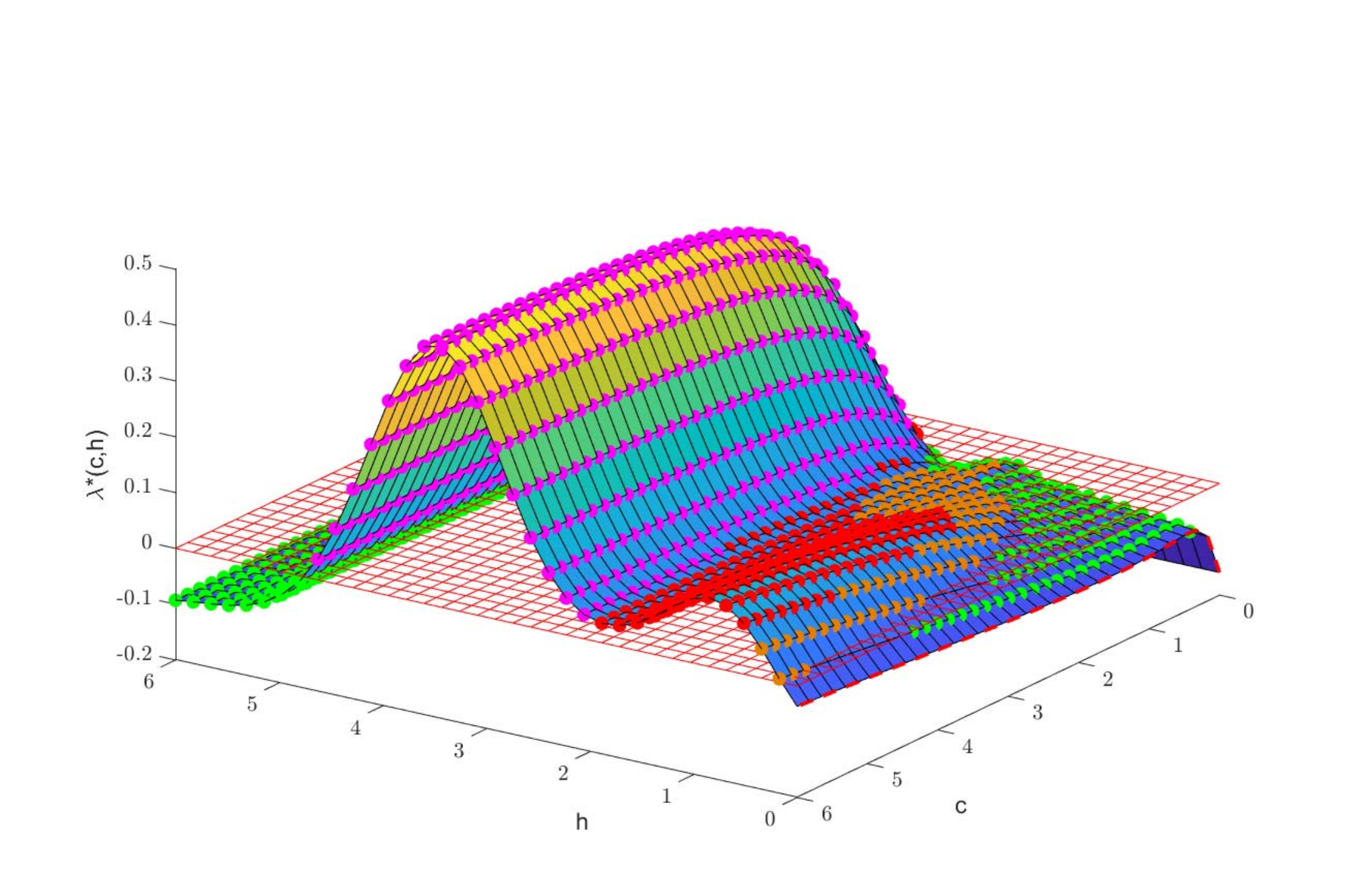}
    \caption{Visual representation of the effectiveness of the criteria of tracking and tipping of Corollary
    \ref{5.corotrack1} and Theorem \ref{5.corotipping}. The four pictures showcase the same surface
    $\lambda_*(c,h)$ (for the quadratic problem \eqref{eq:quad_concave}) numerically approximated
    using the algorithm presented in~\cite{lno}: the two rows correspond to different views.
    As explained in \cite{lno}, the part of the surface which lies below the red grid ($\lambda_*=0)$
    of left-hand side correspond to pairs $(c,h)$ providing end-point tracking, while the ones above provide
    tipping. On the right-hand side the surface is complemented by chromatic dots corresponding to pairs $(c,h)$
    where our methods can rigorously guarantee tracking (in green), tipping in one step (in magenta),
    tipping in two steps (in red), and tipping in more than two steps (in orange).}
    \label{fig:num_track_tipp_sec5}
\end{figure}
\par
Let us say a few words on the way the simulations are set up. The $(c,h)$-surface $\lambda_*(c,h)$
for the two-parametric problem \eqref{eq:quad_concave} is constructed via the algorithm presented in~\cite[Appendix B]{lno}.
A grid of $36\times 36$ values covering $[0,6]^2$ is use. For each one of these points we check our criteria of
tracking and tipping. For the points of tracking (in green),
we verify the criterion of Corollary \ref{5.corotrack1} and Remark \ref{5.notah}. We fix
$(c,h)$ with $h>0$ and take $I_{c,h}=[\texttt{floor}(-80/(ch)),\texttt{ceil}(80/ch)]$ as a suitable interval of integration.
Then, we look for a value $\nu\in(0,1/2)$ such that: the solutions $x(t+h,s,b^\nu(t))$ of $x'=-x^2+p(t)$ satisfy
$x(t+h,t,b^\nu(t))>b^{1-\nu}(t+h)$ whenever $t=jh\in I_{c,h}$ for some $j\in\Z$; there exist
$j_0\in \Z$ and $n\in\N$ such that $[j_0h,(j_0+n)h]\subseteq I_{c,h}$, and with
$\mah(j_0h)-\G(cj_0h)>b^{\nu}(j_0h)$ and $\mrh((j_0+n)h)-\G(c(j_0+n)h)<b^{\nu}((j_0+n)h)$; and
the (strict) inequalities \eqref{5.hipjcoro} hold
for all $j=j_0,\ldots, j_0+n-1$ for $\nu_1=\nu$ and $\nu_2=1-\nu$.
If such a value of $\nu$ is found, a green dot is then plotted on $\lambda_*(c,h)$.
\par
For the points of tipping, we fix once more a pair $(c,h)$
and take $k:=\texttt{floor}(30/h)$. Then,
we check if the one-step criterion for tipping is satisfied, i.e., if the inequality
$\wit a((j_0+1)h)-\wit r((j_0+1)h)<\G(c(j_0+1)h)-\G(cj_0h)$
of Theorem \ref{5.corotipping}(i) holds for some $j_0\in\{0,\ldots,2k\}$.
If so, a magenta
point is assigned to $\lambda_*(c,h)$. If not, we check a two-steps criterion:
if the inequality of Theorem \ref{5.corotipping}(iii)
(with $\G(ct)$ instead of $\G(t)$) holds for some $j_0\in\Z$, we
plot a red point on $\lambda_*(c,h)$.
If none of the previous inequalities is satisfied, we check the several steps
criterion of Theorem \ref{5.corotipping}(iv): we
numerically check the existence of $j\in\{0,\ldots,2k\}$ such that
$x_{c,h}((j-k)h,-kh,\wit a(-kh)+\G(-ckh))<\widetilde (r(j-k)h)+\Gamma(c(j-k)h)$,
which proves tipping; and,
in this case, an orange dot is plot on $\lambda_*(c,h)$.
Observe finally that the net of green (resp.~magenta, red and orange) points practically covers
the negative (resp.~positive) part of the surface of $\lb_*$; i.e., the part corresponding
to pairs $(c,h)$ for which there is tracking (resp.~tipping).
\par
We point out that our results allow us to carry out a similar analysis of phase-induced
tipping (see Section \ref{6.sec}) instead of rate-induced tipping.
\section{Rigorous estimates for continuous transitions}\label{6.sec}
In Section \ref{5.sec}, we have obtained several results concerning equations
\eqref{5.eqGamma} with $\G$ replaced by the piecewise constant function
$\G^h$ for $h>0$. The case $h=0$, which we consider in this section, is quite
more difficult to deal with. In fact,
for simplicity and to formulate
statements that are not too abstract,
our results for this case are restricted to a particular type of
$f$, namely $f(t,x)=-x^2+p(t)$ where $p\colon\R\to\R$ is an $L^\infty$ function.
That is, for a fixed bounded and continuous function $\G\colon\R\to\R$
with $\gamma_\pm:=\lim_{t\to\pm\infty}\G(t)\in\R$ and a fixed $L^\infty$-map
$p\colon\R\to\R$, we deal with the equation
\begin{equation}\label{6.ecGp}
 x'=-\big(x-\Gamma(t)\big)^2+p(t)\,,
\end{equation}
understood as a perturbation of
\begin{equation}\label{6.ecp}
 x'=-x^2+p(t)
\end{equation}
as well as a transition from the past equation \eqref{4.ecu-} to the
future one \eqref{4.ecu+}.
Let us reformulate Hypothesis \ref{4.hipo} for this setting:
\begin{hipo}\label{6.hipo}
The equation \eqref{6.ecp} has an attractor-repeller pair $(\wit a,\wit r)$,
\end{hipo}
Our objective is to establish
criteria ensuring that the dynamics of the transition equation
fits \hyperlink{CA}{\sc Cases A} or
\hyperlink{CC}{\sc C}. It is easy to check
that $f$ satisfies conditions \hyperlink{f1}{\bf f1}-\hyperlink{f7}{\bf f7}.
Note that the information provided by Proposition
\ref{4.proppersiste}, Theorem~\ref{4.teorexist} and Remarks~\ref{5.notatrans}
is valid for this formulation (as well as the rest of the results
of Section \ref{5.sec}).
\par
Throughout this section, we represent by $\lb_*(\G,p)$ the special
value of the parameter associated to \eqref{6.ecGp} by Theorem \ref{3.teorlb*}.
In particular, Hypotheses \ref{6.hipo} is equivalent to $\lb_*(0,p)<0$.
We will apply part of the results of \cite{lno}: please be advised that our
present $\lb_*(\G,p)$ is represented by $\lb_*(2\G,p-\G^2)$ in that paper.
Recall that a negative, zero, or positive $\lb_*(\G,p)$ corresponds to
{\sc Cases} \hyperlink{CA}{\sc A} (tracking), \hyperlink{CB}{\sc B} or \hyperlink{CC}{\sc C}
(tipping), respectively: see Theorem \ref{3.teorlb*} and Section \ref{4.sec}.
\subsection{Total and partial tracking on the hull}\label{4.subsec_partial_tipping}
Notions of total and partial tracking/tipping suitable for our framework where introduced in
\cite{lno}. Let us recall the idea. The properties which we mention here are
proved in Appendix \ref{appendix} for a more
general setting. Let $\W_p$ be the hull of $p$,
which is a compact metric space defined as the closure of
$\{p_s\,|\;s\in\R\}$, where $p_s(t):=p(s+t)$ and
the closure is taken in the set $L^\infty(\R)$
provided with the topology
$\sigma$ defined by the family of seminorms
\[
 n_{[r_1,r_2]}(q)=\Big|\int_{r_1}^{r_2}q(t)\,dt\,\Big|\quad
 \text{for $r_1,r_2\in\Q\:$ with $\;r_1<r_2$}\,.
\]
Then, the time shift $t\mapsto q_t$ defines a continuous global flow on
$\W_p$; and, if $u(t,q,x_0)$ represents the solution of
\begin{equation}\label{6.ecqhull}
 x'=-x^2+q(t)
\end{equation}
for $q\in\W_p$ with $u(0,q,x_0)=x_0$, then
the map $(t,q,x_0)\mapsto (q_t,u(t,q,x_0))$ is a local continuous {\em skew-product} flow
on $\W_p\times\R$, defined of an open subset $\mU\subseteq\R\times\W_p\times\R$.
\par
By reviewing the proof of Theorem \ref{3.teorhyp}, sketched in
Appendix \ref{appendix}, we observe
that Hypothesis \ref{6.hipo} ensures the existence two hyperbolic copies
of the base $\W_p$ for this skew-product flow: the graphs of two continuous real maps
on $\W_p$, which are exponentially asymptotically stable sets as time increases
(the upper one) or decreases (the lower one). This property ensures that each equation
\eqref{6.ecqhull} has an attractor-repeller pair
(given by the corresponding sections of the hyperbolic copies of the base).
Since Hypothesis \ref{6.hipo} ensures the existence of an attractor-repeller
pair for $y'=-(y-\gamma_\pm)^2+p(t)$, the
previous argument shows the same property for $y'=-(y-\gamma_\pm)^2+q(t)$
for any $q\in\W_p$. So, a natural question arises:
are all the equations
\begin{equation}\label{6.eqGq}
 y'=-(y-\G(t))^2+q(t)
\end{equation}
for $q\in\W_p$ in the same dynamical case?
The global occurrence of \hyperlink{CA}{\sc Case A} is called
{\em total tracking on the hull $\W_p$ for \eqref{6.ecGp}}, and we talk about
{\em partial tracking} ({\em or tipping}) {\em on the hull} when
\hyperlink{CA}{\sc Cases A} and \hyperlink{CC}{C}
coexist for different functions in $\W_p$.
\par
Recall that the sign of the bifurcation value $\lb_*(\G,q)$ associated to \eqref{6.eqGq}
by Theorem \ref{3.teorlb*} determines its dynamical situation.
As shown in \cite[Remark 2.13.2]{lno}, the map $\W_p\to\R\,\;q\mapsto\lb_*(\G,q)$ is
not continuous, in general. The next result shows some continuity properties
and provides criteria for total and partial tacking.
\begin{teor}\label{6.teorconthull}
Let $\W_p$ be the hull of $p$. Then,
\begin{itemize}
\item[\rm(i)] the map $\R\to\R\,,\;s\mapsto\lb_*(\G,q_s)$ is uniformly
continuous for all $q\in\W_p$.
\item[\rm(ii)] The map $\W_p\to\R\,,\;q\mapsto \lb_*(\G,q)$ is lower semicontinuous.
Consequently, $\lb_*(\G,q)\le\lb_*(\G,p)$ for all $q\in\W_p$.
\item[\rm(iii)] If there exists $\rho>0$ with $\lb_*(\G,p_s)\le-\rho$ for all
$s\in\R$, then there is total tracking on the hull for \eqref{6.ecGp}.
\item[\rm (iv)] If $p$ is almost periodic, then the map $\W_p\to\R\,,\;q\mapsto \lb_*(\G,q)$ is
continuous.
\item[\rm(v)] If $p$ is almost periodic and there exist $s_1,s_2\in\R$
with $\lb_*(\G,p_{s_1})<0<\lb_*(\G,p_{s_2})$, then there exist four
sequences $(s_{1,n}^\pm)$, $(s_{2,n}^\pm)$ with $\lim_{n\to\infty}s_{i,n}^\pm=\pm\infty$
and $\pm s_{1,n}^\pm\le\pm s_{2,n}^\pm\le\pm s_{1,n+1}^\pm$ for all $n\in\N$ such that
$\lb_*(\G,p_{s_{1,n}^\pm})<0<\lb_*(\G,p_{s_{2,n}^\pm})$ for all $n\in\N$.
\end{itemize}
\end{teor}
\begin{proof}
To begin, let us check that
\begin{equation}\label{6.common}
 \n{q}_{L^\infty}\le\n{p}_{L^\infty}  \quad\text{ for all $q\in\W_p$}\,.
\end{equation}
We take $q\in\W_p$ and $(s_n)$ with $p_{s_n}\xrightarrow{\sigma}q$. If $r,h\in\R$
and $\ep>0$,
\[
 \frac{1}{h}\int_r^{r+h} p_{s_n}(t)\,dt-\ep\le
 \frac{1}{h}\int_r^{r+h} q_{s_n}(t)\,dt\le \frac{1}{h}\int_r^{r+h} p(t)\,dt+\ep
\]
for large enough $n\in\N$ (see Lemma \ref{A.lema}). Since $\n{p_s}_{L^\infty}=\n{p}_{L^\infty}$,
we get $-\n{p}_{L^\infty}-\ep\le (1/h)\int_r^{r+h} q(t)\,dt \le \n{p}_{L^\infty}+\ep$.
Lebesgue's Differentation Theorem ensures that $|q(r)|\le \n{p}_{L^\infty}+\ep$
for $l$-a.a.~$r\in\R$, which proves our assertion. We will use \eqref{6.common}
to prove (ii) and (iv).
\smallskip\par
(i) For each $s\in\R$, the change $y(t)=x(t-s)$ takes $x'=-(x-\G(t))^2+q_s(t)$ to
$y'=-(y-\G_{-s}(t))^2+q(t)$, and preserves the number of bounded solutions, which means
that $\lb_*(\G,q_s)=\lb_*(\G_{-s},q)$. According to \cite[Theorem 2.12]{lno}
(whose proof can be repeated for $L^\infty$ maps using Theorem \ref{3.teorlb*}),
there exists $m>0$ such that $|\lb_*(\G,q_{s_1})-\lb_*(\G,q_{s_2})|
\le m \n{\G_{-s_1}-\G_{-s_2}}_\infty$.
Let us define $\theta(h):=\sup\{|\G(t+s)-\G(t)|\,|\;t\in\R,\,s\in[-h,h]\}$ and
deduce from the uniform continuity of $\G$ that $\theta$
is a continuous increasing map on $[0,\infty)$ with $\theta(0)=0$. Then,
$|\lb_*(\G,q_{s_1})-\lb_*(\G,q_{s_2})|\le m\,\theta(|s_2-s_1|)$,
which proves the assertion.
\smallskip\par
(ii) We fix $\ep>0$, define $m:=\n{\G}_\infty+(\n{p}_{L^\infty}+\ep)^{1/2}$, and deduce from
\eqref{6.common} that
$f_q(t,x)\le-\ep$ for all $q\in\W_p$ if $|x|\ge m$, where $f_q(t,x)=
-(x-\G(t))^2 + q(t)$. This fact, \eqref{6.common}, and Theorem \ref{3.teorlb*} show
the existence of a common bound $m_1\ge|\lb_*(\G,q)|$ for all $q\in\W_p$.
Let us take an element $q$ and a sequence $(q_n)$, all in $\W_p$, with
$q_n\xrightarrow{\sigma} q$, assume that $\bar\lb=\lim_{k\to\infty}
\lb_*(\G,q_k)$ for a subsequence $(q_k)$, and check that
$\lb_*(\G,q)\le \bar\lb$. Let $\mb_k$ be a bounded solution
of $x'=-(x-\G(t))^2+q_k(t)+\lb_*(\G,q_k)$ (see Theorem~\ref{3.teorlb*}).
Theorem \ref{3.teoruno} shows that
$\n{\mb_k}_\infty\le m_2$ for all $k\in\N$,
where $m_2:=\n{\G}_\infty+(\n{p}_{L^\infty}+m_1+\ep)^{1/2}$. Therefore, there exists
$m_3$ with $\n{\mb'_k}_{L^\infty}\le m_3$ for all $k\in\N$. Arzel\'{a}-Ascoli's
Theorem provides a subsequence $(\mb_j)$ which converges to a new map $\mb_0$
on the compact subsets of $\R$.
It follows from Theorem \ref{A.teorcont} that the bounded map $\mb_0$ solves
$x'=-(x-\G(t))^2+q(t)+\bar\lb$, and Theorem \ref{3.teorlb*}
shows that $\lb_*(\G,q)\le \bar\lb$. The definition of $\W_p$
proves the last assertion.
\smallskip\par
(iii) According to (ii), $\lb_*(\G,q)\le-\rho$ for all $q\in\W_p$ if
$\lb_*(\G,p_s)\le-\rho$ for all $s\in\R$, which proves assertion.
\smallskip\par
(iv) Property \eqref{6.common} and \cite[Theorem 2.12]{lno} (see the proof of (i))
provide $k=k(\n{p}_{L^\infty})>0$ such that
$|\lb_*(\G,q_1)-|\lb_*(\G,q_2)|\le k\n{q_1-q_2}_{L^\infty}$ if $q_1,q_2\in\W_p$.
If $p$ is almost periodic, the topology $\sigma$ on $\W_p$ coincides
with the topology of the uniform convergence
on $\R$: see \cite[Chapter 1]{fink}. These facts prove (iv).
\smallskip\par
(v) If $p$ is almost periodic, then its hull coincides with the
alpha-limit set and with the omega-limit set of $\{p_t\,|\;t\in\R\}$.
Hence, there are sequences $(s_{i,n}^\pm)$
with $\lim_{n\to\infty}s_{i,n}^\pm=\pm\infty$ and
$\lim_{n\to\pm\infty}p_{s_{i,n}^\pm}=p_{s_i}$
for $i=1,2$. There is no restriction in assuming that $\pm s_{1,n}^\pm\le\pm s_{2,n}^\pm
\le\pm s_{1,n+1}^\pm$ for all $n\in\N$. So, (v) follows from (iv).
\end{proof}
A {\em phase-induced critical transition} occurs when
a change in the initial phase $s$ of $p$
(i.e., replacing $p$ by $p_s$ in \eqref{6.ecGp}, or, equivalently,
replacing $\G$ by $\G_{-s}$) causes a critical transition for
\eqref{6.ecGp} as $s$ varies. If so, we have partial tracking on the hull.
In addition, Theorem \ref{6.teorconthull}(iv) implies that partial tracking on the
hull ensures infinitely many phase-induced critical transitions in the
case of an almost periodic map $p$.
\subsection{Tracking and tipping for the piecewise constant and the unperturbed
cases}\label{6.sec51}
Let us define $\G^h\colon\R\to\R$ for $h\ge 0$ by \eqref{5.defGch},
and consider the equations
\begin{equation}\label{6.ectrozos}
 x'=-\big(x-\Gamma^h(t)\big)^2+p(t)\,.
\end{equation}
Their special shape allows us to go deeper in the criteria for the existence
or absence of an attractor-repeller for small values of $h$,
which we will do under Hypothesis \ref{6.hipo}. In addition, unlike in
Section \ref{5.sec}, here the criteria will also be valid for $h=0$;
that is, for \eqref{6.ecGp}.
\par
Recall that $\lambda_*(\G,p)$ is the bifurcation value
of the parameter associated to \eqref{6.ecGp} by Theorem \ref{3.teorlb*},
and that $\lambda_*(0,p)$ (associated to \eqref{6.ecp})
is negative under Hypothesis {\rm\ref{6.hipo}}.
Recall also that $x(t,s,x_0)$ stands for the solution of \eqref{6.ecp}
with $x(s,s,x_0)=x_0$. Let
$x_h(t,s,x_0)$ be the solution of \eqref{6.ectrozos}$_h$ with
$x_h(s,s,x_0)=x_0$. We will also make use of the functions $b^\nu$ defined by
\eqref{5.defbnu}.
\par
Our first result shows that, in fact,
tracking for $h=0$ ensures tracking for small positive $h$, and
a uniform property of propagation to the hull.
\begin{prop} \label{6.prophull}
Let $\W_p$ be the hull of $p$. Then,
\begin{itemize}
\item[\rm(i)] if $\lb_*(\G,p)<0$, then there exists $h_*>0$ such that
$\lb_*(\G^h,p)<0$ (i.e., \eqref{6.ectrozos}$_h$ has an attractor-repeller pair)
for all $h\in[0,h_*]$.
\item[\rm(ii)] If there exists $\rho>0$ with $\lb_*(\G,p_s)\le-\rho$ for
all $s\in\R$, then there exists $h_*>0$ such that
$\lb_*(\G^h,q)<0$ for all $q\in\W_p$ and $h\in[0,h_*]$: there is
total tracking on the hull $\W_p$ for \eqref{6.ectrozos}$_h$ for all $h\in[0,h_*]$.
\end{itemize}
\end{prop}
\begin{proof}
(i) As explained in the proof of Theorem \ref{6.teorconthull}(i), there exists
a positive constant $m$ such that $|\lb_*(\G,p)-\lb_*(\G^h,p)|
\le m \n{\G-\G^h}_\infty$. Therefore, the assertion follows
from the uniform convergence of $\G^h$ to $\G$ as $h\downarrow 0$.
\smallskip\par
(ii) Theorem \ref{6.teorconthull}(ii) ensures that
$\lb_*(\G,q)\le-\rho$. As above, there exists a positive constant $m$ such that
$\lb_*(\G^h,q)\le\lb_*(\G,q)+|\lb_*(\G,q)-\lb_*(\G^h,q)|\le
-\rho+m\n{\G-\G^h}_\infty$ for all $q\in\W_p$.
So, (ii) follows from $\lim_{h\downarrow 0}\n{\G-\G^h}_\infty=0$.
\end{proof}
Following similar ideas to those of Section \ref{5.sec},
the next results, Proposition \ref{6.propcambio} and Theorems
\ref{6.teortrack1} and \ref{6.teortrack2}, provide sufficient conditions for the
existence of an attractor-repeller pair for the quadratic equation
\eqref{6.ectrozos}$_h$ if $h\ge 0$ is small enough. Lemma \ref{6.lemaconverg},
which is based on results of~\cite{lno}, is fundamental to complete their proofs.
The maps $\mah$ and $\mrh$ are those associated to \eqref{6.ectrozos}$_h$
by Theorem \ref{3.teoruno} for $h\ge 0$.
\begin{lema}\label{6.lemaconverg}
Assume Hypothesis {\rm\ref{6.hipo}}.
\begin{itemize}
\item[\rm(i)] If $\ma_0$ exists on $(-\infty,\bar t\,]$, then there exists
$h_0=h_0(\bar t)>0$ such that $\mah$ exists on $(-\infty,\bar t\,]$ for all $h\in[0,h_0]$,
and $\lim_{h\to 0^+}\mah(t)=\ma_0(t)$ uniformly on $(-\infty, \bar t\,]$.
\item[\rm(ii)] If $\mr_0$ exists on $[\bar t,\infty)$, then there exists
$h_0=h_0(\bar t)>0$ such that $\mrh$ exists on $[\bar t,\infty)$ for all $h\in[0,h_0]$,
and $\lim_{h\to 0^+}\mrh(t)=\mr_0(t)$ uniformly on $[\bar t,\infty)$.
\end{itemize}
\end{lema}
\begin{proof}
Let us fix $\ep>0$. By reviewing the proof of Theorem \ref{4.teorexist}(i)\&(ii)
(in~\cite{lno}), we observe that there exists $t_*>0\in\mathbb R$ such that, if $h\in[0,1]$,
then $|\mah(-t)-\ma_0(-t)|\le\ep$ and $|\mrh(t)-\mr_0(t)|\le-\ep$ if $t\ge t_*$.
Now, to prove (i), we assume that $\ma_0$ exists in $(-\infty,\bar t\,]$ for a $\bar t>-t_*$.
According to~\cite[Theorem A.3]{lno} (whose proof can be repeated without any change for
an $L^\infty$ function $p$), $\lim_{h\to 0^+}\mah(s)=\lim_{h\to 0^+}x_h(s,-t_*,\ma_h(-t_*))=
x_0(s,-t_*,\ma_0(-t_*))=\ma_0(s)$ uniformly in $s\in[-t_*,\bar t\,]$,
which proves (i). And we check (ii) with the same argument.
\end{proof}
\begin{prop}\label{6.propcambio}
Assume Hypothesis {\rm\ref{6.hipo}}, that $\G$ is $C^1$ with $\G'$ bounded on $\R$,
and that $\lambda_*(0,p)\le -\sup_{t\in\R}\G'(t)$. Then, there exists $h_*>0$
such that \eqref{6.ectrozos}$_h$ has an attractor-repeller pair
for all $h\in[0,h_*]$. In addition, there is total tracking on the hull $\W_p$ of $p$
for \eqref{6.ectrozos}$_{0}$ (i.e., for \eqref{6.ecGp}).
\end{prop}
\begin{proof}
Let us check the result for $h=0$.
The change of variable $y=x-\G(t)$ takes the equation \eqref{6.ectrozos}$_0$ to
$y'=-y^2+p(t)-\G'(t)$ without changing the dynamics (i.e., \hyperlink{CA}{\sc Case A},
\hyperlink{CB}{\sc B} or \hyperlink{CC}{\sc C} persists). Let $b$ be
the unique bounded solution of $x'=-x^2+p(t)+\lb_*(0,p)$ (see Theorem \ref{3.teorlb*}).
Then $b'(t)\le -b^2(t)+p(t)+\lambda_*(0,p)\le-b^2(t)+p(t)-\G'(t)$. In addition,
since $\lambda_*(0,p)<0$ and $\G$ has finite asymptotic limits, it cannot
be $\lambda_*(0,p)=-\G'(t)$ for all $t\in\R$, and hence there exists
$t_0$ such that $b'(t_0)< -b^2(t_0)+p(t_0)-\G'(t_0)$. Theorem
\ref{3.teoruno}(v) ensures that $y'=-y^2+p(t)-\G'(t)$ (and hence
\eqref{6.ectrozos}$_0$) has two different bounded solutions,
and Theorems \ref{4.teorexist}(v) and \ref{3.teorlb*}
ensure that \eqref{6.ectrozos}$_0$ has an attractor-repeller pair.
Proposition \ref{6.prophull}(i) ensures the existence of
an $h_*>0$ such that \eqref{6.ectrozos}$_h$ has an attractor-repeller pair
for all $h\in[0,h_*]$, as asserted. The last property follows from
Theorem \ref{6.teorconthull}(iii), which states that
$\lambda_*(0,q)\le -\sup_{t\in\R}\G'(t)$ for all $q\in\W_p$,
and from the first assertion.
\end{proof}
\begin{lema}\label{6.lemader}
Assume Hypothesis {\rm\ref{6.hipo}}. There exists a constant $\kappa>0$
such that any bounded solution $b$ of \eqref{6.ecp} satisfies
\begin{equation}\label{6.igulema}
 b(t+h)=b(t)-b^2(t)\,h+\int_{t}^{t+h}p(s)\,ds+\rho_b(t,t+h)
\end{equation}
for any $t\in\R$ and $h\ge0$, with $|\rho_b(t,t+h)|\le k\,h^2$.
\end{lema}
\begin{proof}
Let $\kappa_1$ satisfy $-\kappa_1\le\inf_{t\in\R}\wit r(t)<\sup_{t\in\R}\wit a(t)<\kappa_1$
and take $\kappa_2\ge\n{p}_{L^{\infty}}$.
It is easy to deduce from the equation solved by $b$ that, if $s\in[t,t+h]$, then
\[
 |b(s)-b(t)|\le \int_{t}^s (|b^2(l)|+|p(l)|)\,dl\le (\kappa_1^2+\kappa_2)\,h
 =:\kappa_3\,h\,;
\]
and also that \eqref{6.igulema} holds for
\[
 \rho_b(t+h,t):=\int_{t}^{t+h}(-b^2(s)+b^2(t))\,ds=
 \int_{t}^{t+h}(-b(s)+b(t))(b(s)+b(t))\,ds\,.
\]
Hence, $|\rho_b(t,t+h)|\le\int_{t}^{t+h}2\kappa_1\kappa_3\,h\,ds=:
\kappa\,h^2$.
\end{proof}
Lemma \ref{5.lemat*} ensures the existence of the times $\bar t_1$ and $\bar t_2$
satisfying the initial requirements of the next two theorems.
Observe that
the statement of Theorem \ref{6.teortrack1}(i) could be included in Theorem \ref{6.teortrack2}
(taking $\alpha=0$). We separate them since the requirements regarding $(\wit a(t)-\wit r(t))$
in Theorem \ref{6.teortrack2} are easier to verify, while that
of Theorem \ref{6.teortrack2} is less restrictive and more suitable for small $\inf_{t\in\R}(\wit a(t)-\wit r(t))$.
\begin{teor}\label{6.teortrack1}
Assume Hypothesis {\rm\ref{6.hipo}}.
\begin{itemize}
\item[\rm(i)]
Let $\bar t_1<\bar t_2$ satisfy
$\ma_0(\bar t_1)-\G(\bar t_1)>b^{1/2}(\bar t_1)$ and
$\mr_0(\bar t_2)-\G(\bar t_2)<b^{1/2}(\bar t_2)$.
Assume that $\G$ is $C^2$ on $[\bar t_1,\bar t_2]$, with
\begin{equation}\label{6.deri}
 \frac{1}{4}\big(\wit a(t)-\wit r(t)\big)^2>\G'(t)
 \quad \text{for all $t\in[\bar t_1,\bar t_2]$}\,.
\end{equation}
Then, there exists $h_*>0$ such that
\eqref{6.ectrozos}$_h$ has an attractor-repeller pair
for all $h\in[0,h_*]$.
\item[\rm(ii)]
Assume that $\G$ is $C^2$ on $\R$ and that
\begin{equation}\label{6.deri2}
 \frac{1}{4}\inf_{t\in\R}\big(\wit a(t)-\wit r(t)\big)^2>\sup_{t\in\R}\G'(t)\,.
\end{equation}
Then, there exists $h_*>0$ such that there is total tracking on the hull
$\W_p$ of $p$ for \eqref{6.ectrozos}$_h$ for all $h\in[0,h_*]$.
\end{itemize}
\end{teor}
\begin{proof}
(i) Recall that $b^{1/2}=(1/2)(\wit a+\wit r)$: see \eqref{5.defbnu}.
Applying Lemma \ref{6.lemader} to the solutions $x(\cdot,t,b^{1/2}(t))$,
$\wit a$ and $\wit r$ of \eqref{6.ecp}, we get
\[
\begin{split}
 x(t+h,t,b^{1/2}(t))&=\frac{1}{2}\,(\wit a(t)+\wit r(t))
 +\int_{t}^{t+h}p(s)\,ds-\frac{1}{4}\,(\wit a(t)+\wit r(t))^2\,h+O(h^2)\,,\\
 b^{1/2}(t+h)&=\frac{1}{2}\,(\wit a(t)+\wit r(t))
 +\int_{t}^{t+h}p(s)-\frac{1}{2}\,(\wit a^2(t)+\wit r^2(t))\,h+O(h^2)
\end{split}
\]
for $h\ge 0$. Hence,\vspace{-.2cm}
\[
 x(t+h,t,b^{1/2}(t))-b^{1/2}(t+h)=\frac{1}{4}(\wit a(t)-\wit r(t))^2\,h+O(h^2)
\]
for $h\ge 0$.
Since $\G''(t)$ is bounded on $[\bar t_1,\bar t_2]$,
we have $\G(t+h)-\G(t)=\G'(t)\,h+ O(h^2)$ for $t\in[\bar t_1,\bar t_2-h]$.
Therefore, the condition \eqref{6.deri}
allows us to determine $h_0>0$ such that, if $h\in[0,h_0]$, then
\[
 x(t+h,t,b^{1/2}(t))-b^{1/2}(t+h)\ge\G(t+h)-\G(t) \quad
 \text{for $\,t\in[\bar t_1,\bar t_2-h]$}\,.
\]
Applying Lemma \ref{6.lemaconverg}, we deduce the  existence of $\delta>0$
with $2\delta<\bar t_2-\bar t_1$ and $h_1\in(0,h_0]$ such that,
for all $h\in[0,h_1]$,
$\mah(t)-\G(t)>b^{1/2}(t)$ for all $t\in[\bar t_1,\bar t_1+\delta]$ and
$\mrh(t)-\G(t)<b^{1/2}(t)$ for all $t\in[\bar t_2-\delta,\bar t_2]$. We choose
$m_1<m_2$ in $\Z$ and $\tilde h\in(0,h_1]$ such that
$\tilde t_1:=m_1\tilde h\in [\bar t_1,\bar t_1+\delta]$ and
$\tilde t_2:=m_2\tilde h\in [\bar t_2-\delta,\bar t_2]$.
We also define $t_j:=j\tilde h$ for $j=m_1,\ldots,m_2$.
With all these choices and definitions, we get
\[
 \mah(t_{m_1})-\G(t_{m_1})>b^{1/2}(t_{m_1})\,,\quad
 \mrh(t_{m_2})-\G(t_{m_2})<b^{1/2}(t_{m_2})
\]
and
\[
 x(t_j+\tilde h,t_j,b^{1/2}(t_j))-b^{1/2}(t_j+\tilde h)\ge\G(t_j+\tilde h)-\G(t_j)
 \quad\text{\,for $j=m_1,\ldots,m_2-1$}\,.
\]
We call $\nu_j=1/2$ for all $j\in\Z$ and
observe that all the conditions of Theorem \ref{5.teortrack}
are fulfilled (for $j_0=m_1$ and $n=m_2-m_1$).
This proves that \eqref{6.ectrozos}$_{\tilde h}$ has an attractor-repeller pair.
\par
Let us deduce this property for $h=0$. For each $k\in\N$, we observe that
the previous procedure can be repeated for $h_k:=\tilde h/k=\tilde t_1/(km_1)
=\tilde t_2/(km_2)$.
Relation \eqref{5.para5} for $j=km_2$ ensures that $\ma_{h_k}(\tilde t_2)\ge
b^{1/2}(\tilde t_2)+\G(\tilde t_2)$. According to Lemma \ref{6.lemaconverg},
$\ma_0(\tilde t_2)=\lim_{k\to\infty}\ma_{h_k}(\tilde t_2)$, and hence
$\ma_0(\tilde t_2)\ge b^{1/2}(\tilde t_2)+\G(\tilde t_2)>\mr_0(\tilde t_2)$.
Hence, there exists $\ma_0(\tilde t_2)>\mr_0(\tilde t_2)$, which according
to Remark \ref{3.notabasta} and Theorem \ref{4.teorexist}(iv)
ensures that \eqref{6.ectrozos}$_0$ has an attractor-repeller pair.
\par
The existence of an $h_*>0$ such that \eqref{6.ectrozos}$_h$
has an attractor-repeller pair for all $h\in[0,h_*]$ follows
from Proposition \ref{6.prophull}(i).
\smallskip\par
(ii) Proposition \ref{3.proppersiste}, Hypotheses \ref{6.hipo} and
relation \eqref{6.deri2} ensure
the existence of $\rho>0$ such that $x'=-x^2+(p(t)+\rho)$ has
an attractor repeller par $(\wit a_{\rho},\wit r_{\rho})$ with
\begin{equation}\label{6.deri3}
 \frac{1}{4}\min_{t\in\R}\big(\wit a_{\rho}(t)-\wit r_{\rho}(t)\big)^2
 >\sup_{t\in\R}\G'(t)\,.
\end{equation}
Since $((\wit a_\rho)_s, (\wit r_\rho)_s)$ is an attractor repeller par
for $x'=-x^2+(p_s+\rho)(t)$ and \eqref{6.deri3} also holds for this pair,
we conclude from Lemma \ref{5.lemat*} and (i) that $\lb_*(\G,p_s+\rho)<0$.
Hence, by Theorem \ref{3.teorlb*}(v), $\lb_*(\G,p_s)<-\rho$,
and Proposition \ref{6.prophull}(ii) proves the assertion.
\end{proof}
\begin{teor}\label{6.teortrack2}
Assume Hypothesis {\rm \ref{6.hipo}}.
Let us fix $\alpha\in(0,1/2)$ and let $\bar t_1<\bar t_2$ satisfy
$\ma_0(\bar t_1)-\G(\bar t_1)>b^{1/2+\alpha}(\bar t_1)$ and
$\mr_0(\bar t_2)-\G(\bar t_2)<b^{1/2-\alpha}(\bar t_2)$.
Assume that $t\mapsto\G(t)$ is $C^2$ on $[\bar t_1,\bar t_2]$,
and that
\[
 \left(\frac{1}{4}-\frac{\alpha^2\,(\bar t_2+\bar t_1-2t)^2}{(\bar t_2-\bar t_1)^2}
 \right)(\wit a(t)-\wit r(t))^2+
 \frac{2\alpha}{\bar t_2-\bar t_1}\,(\wit a(t)-\wit r(t))> \G'(t)\quad
 \text{if $t\in[\bar t_1,\bar t_2]$}\,.
\]
Then, there exists $h_*>0$ such that \eqref{6.ectrozos}$_h$
has an attractor-repeller pair for all $h\in[0,h_*]$.
\end{teor}
\begin{proof}
Let us look for $\delta>0$ with $2\delta<\bar t_2-\bar t_1$ such that, if
$t_1\in[\bar t_1,\bar t_1+\delta]$, $t_2\in[\bar t_2-\delta,\bar t_2]$ and
$t\in[t_1,t_2]$,
\begin{equation}\label{6.desalpha}
 \left(\frac{1}{4}-\frac{\alpha^2\,(t_2+t_1-2t)^2}{(t_2-t_1)^2}
 \right)(\wit a(t)-\wit r(t))^2+
 \frac{2\alpha}{t_2-t_1}\,(\wit a(t)-\wit r(t))> \G'(t)\,.\end{equation}
Applying Lemma \ref{6.lemaconverg} and taking a smaller $\delta>0$ if needed, we find
$h_0>0$ such that, for all $h\in[0,h_0]$,
$\mah(t)-\G(t)>b^{\alpha+1/2}(t)$ for all $t\in[\bar t_1,\bar t_1+\delta]$
and $\mrh(t)-\G(t)<b^{\alpha-1/2}(t)$ for all $t\in[\bar t_2-\delta,\bar t_2]$.
We choose $m_1<m_2$ in $\Z$ and $\tilde h\in(0,h_0]$ such that
$\tilde t_1:=m_1\tilde h\in [\bar t_1,\bar t_1+\delta]$ and
$\tilde t_2:=m_2\tilde h\in [\bar t_2-\delta,\bar t_2]$.
Let us define
\begin{equation}\label{6.nujtj}
 \nu_j:=\frac{1}{2}+\frac{m_2+m_1-2j}{m_2-m_1}\,\alpha
 \quad\text{and}\quad t_j:=j\,\tilde h
\end{equation}
for $j=m_1,\ldots,m_2$, so that $\nu_{m_1}=1/2+\alpha$, $\nu_{m_2}=1/2-\alpha$,
$t_{m_1}=\tilde t_1$ and $t_{m_2}=\tilde t_2$. It is easy to check that
\begin{equation}\label{6.ineqs}
 \nu_j-\nu_{j+1}= \frac{2\alpha}{\tilde t_2-\tilde t_1}\,\tilde h\quad\text{\;and\;}
 \quad\nu_j(1-\nu_j)=\frac{1}{4}-
 \frac{\alpha^2\,(\tilde t_2+\tilde t_1-2t_j)^2}{(\tilde t_2-\tilde t_1)^2}\,.
\end{equation}
Let us fix $j\in\{m_1,\ldots,m_2-1\}$ and
calculate $x(t_j+\tilde h,t_j,b^{\nu_j}(t_j))-b^{\nu_{j+1}}(t_{j+1})$ as the sum
of $s1$, $s2$ and $s3$, described below. We will apply Lemma \ref{6.lemader} and
the first equality in \eqref{6.ineqs}:
\[
\begin{split}
 s1:=&\;x(t_j+\tilde h,t_j,b^{\nu_j}(t_j))-b^{\nu_j}(t_j)\\
 =&\;-(\nu_j\,\wit a(t_j)+(1-\nu_j)\,\wit r(t_j))^2\,\tilde h+\int_{t_j}^{t_{j+1}}\!p(s)\,ds + O(\tilde h^2)\,,\\
 s2:=&\;b^{\nu_j}(t_j)-b^{\nu_j}(t_{j+1})\\
 =&\;\big(\nu_j\,\wit a^2(t_j)+(1-\nu_j)\,\wit r^2(t_j)\big)\,\tilde h-\int_{t_j}^{t_{j+1}}\!p(s)\,ds + O(\tilde h^2)\,,\\
 s3:=&\;b^{\nu_j}(t_{j+1})-b^{\nu_{j+1}}(t_{j+1})\\
 =&\;(\nu_j-\nu_{j+1})\big(\wit a(t_j)-\wit r(t_j)-\tilde h(\wit a^2(t_j)-\wit r^2(t_j))\big)+ O(\tilde h^2)\\
 = &\;\frac{2\alpha}{\tilde t_2-\tilde t_1}\,(\wit a(t_j)-\wit r(t_j))\,\tilde h+
 O(\tilde h^2)\,.
\end{split}
\]
That is, using the second equality in \eqref{6.ineqs},
\[
\begin{split}
 &x(t_j+\tilde h,t_j,b^{\nu_j}(t_j))-b^{\nu_{j+1}}(t_{j+1})\\
 &\quad\;=\left(\nu_j\,(1-\nu_j) (\wit a(t_j)-\wit r(t_j))^2+
 \frac{2\alpha}{\tilde t_2-\tilde t_1}\,(\wit a(t_j)-\wit r(t_j))\right)\tilde h+O(\tilde h^2)\\
 &\quad\;= \left(\!\left(\frac{1}{4}-
 \frac{\alpha^2\,(\tilde t_2+\tilde t_1-2t_j)^2}{(\tilde t_2-\tilde t_1)^2}
 \right)(\wit a(t_j)-\wit r(t_j))^2+
  \frac{2\alpha}{\tilde t_2-\tilde t_1}\,(\wit a(t_j)-\wit r(t_j))\right)\tilde h+O(\tilde h^2)\,.
\end{split}
\]
Since $\G''(t)$ is bounded on $[-t_\alpha,t_\alpha]$ and hence
$\G(t_{j+1})-\G(t_j)=\G'(t_j)\,\tilde h+ O(\tilde h^2)$ for $j\in\{m_1,\ldots,m_2-1\}$,
we conclude from \eqref{6.desalpha} that we can take from the beginning
$\tilde h>0$ small enough (i.e., $|m_1|$ and $|m_2|$ large enough) to ensure that
\[
 x(t_j+\tilde h,t_j,b^{\nu_j}(t_j))-b^{\nu_{j+1}}(t_{j+1})\ge \G(t_{j+1})-\G(t_j)
\]
for $j\in\{m_1,\ldots,m_2-1\}$, with $\nu_j$ and $t_j$ defined by \eqref{6.nujtj}.
\par
Recall also that $\ma_{\tilde h}(m_1\tilde h)>
b^{\nu_{m_1}}(m_1\tilde h)+\G(m_1\tilde h)$ and
$\mr_{\tilde h}(m_2\tilde h)<b^{\nu_{m_2}}(m_2\tilde h)-\G(m_2\tilde h)$,
since $\tilde h\le h_0$, $m_1\tilde h=\tilde t_1$, $m_2\tilde h=\tilde t_2$,
$\nu_{m_1}=1/2+\alpha$ and $\nu_{m_1}=1/2-\alpha$.
Altogether, we have checked all the hypotheses of Theorem~\ref{5.teortrack}
(for $j_0=m_1$ and $n=m_2-m_1$),
and hence \eqref{6.ectrozos}$_{\tilde h}$ has an attractor-repeller pair.
\par
To deduce the result first for $h=0$ and then for all $h\in[0,h_*]$, where $h_*>0$,
we repeat the arguments used to complete the proofs of Proposition \ref{6.propcambio}
and Theorem \ref{6.teortrack1}. In this case, we get
$\ma_0(\tilde t_2)\ge b^{1/2-\alpha}(\tilde t_2)+\G(\tilde t_2)>\mr_0(\tilde t_2)$.
\end{proof}
Theorem \ref{6.teortip1}, based on Theorem \ref{5.teortipping} and on some ideas of
the proof of Theorem \ref{6.teortrack2}, establishes a condition guaranteing tipping
if $h$ is small enough.
\begin{teor}\label{6.teortip1}
Assume Hypothesis {\rm\ref{6.hipo}}.
Assume also that $t\mapsto\G(t)$ is nondecreasing, and that there exists
$\bar t_1<\bar t_2$ such that $t\mapsto\G(t)$ is $C^2$ on $[\bar t_1,\bar t_2]$ and
\[
 \left(\frac{1}{4}-\frac{(\bar t_1+\bar t_2-2t)^2}{4\,(\bar t_2-\bar t_1)^2}\right)
 (\wit a(t)-\wit r(t))^2+\frac{1}{\bar t_2-\bar t_1}\,(\wit a(t)-\wit r(t))<\G'(t)
 \quad\text{for all $t\in[\bar t_1,\bar t_2]$}\,.
\]
Then, there exists $h_*>0$
such that \eqref{6.ectrozos}$_h$ has no bounded solutions
for all $h\in[0,h_*]$.
\end{teor}
\begin{proof}
Let us look for $\delta>0$ with $2\delta<\bar t_2-\bar t_1$ such that, if
$t_1\in[\bar t_1,\bar t_1+\delta]$ and $t_2\in[\bar t_2-\delta,\bar t_2]$ and
$t\in[t_1,t_2]$, then
\[
 \left(\frac{1}{4}-\frac{(t_1+t_2-2t)^2}{4\,(t_2-t_1)^2}\right)
 (\wit a(t)-\wit r(t))^2+\frac{1}{t_2-t_1}\,(\wit a(t)-\wit r(t))<\G'(t)
 \quad\text{for all $t\in[t_1,t_2]$}\,.
\]
We choose $m_1,m_2\in\Z$ and $\tilde h>0$ such that
$\tilde t_1:=m_1\tilde h\in [t_1,t_1+\delta]$ and
$\tilde t_2:=m_2\tilde h\in [t_2-\delta,t_2]$.
We repeat the initial steps of the proof of Theorem \ref{6.teortrack2}
taking $\alpha=1/2$,
with the same definitions \eqref{6.nujtj} to conclude
that there exists $\rho>0$ such that
\[
 x(t_j+\tilde h,t_j,b^{\nu_j}(t_j))-b^{\nu_{j+1}}(t_{j+1})\le \G(t_{j+1})-\G(t_j)-\rho
\]
if $\tilde h$ is small enough and $t_j\in\{m_1,\ldots,m_2-1\}$. Note that
$\nu_{m_1}=1$ and $\nu_{m_2}=0$. Theorem \ref{5.teortipping} ensures that
\eqref{6.ectrozos}$_{\tilde h}$ does not have an attractor-repeller pair.
\par
To prove it for $h=0$, we take $k\in\N$ and observe that
the previous procedure can be repeated for $h_k:=\tilde h/k=\tilde t_1/(km_1)
=\tilde t_2/(km_2)$.
Let us assume for contradiction the
global existence of $\ma_0$ and $\mr_0$ with $\ma_0(\tilde t_2)
\ge\mr_0(\tilde t_2)$. This means that $\ma_{\tilde h/k}(\tilde t_2)$ and
$\mr_{\tilde h/k}(\tilde t_2)$ exist if $k$ is large enough, with $\ma_0(\tilde t_2)=\lim_{k\to\infty}\ma_{\tilde h/k}(\tilde t_2)$
and $\mr_0(\tilde t_2)=\lim_{k\to\infty}\mr_{\tilde h/k}(\tilde t_2)$: see Lemma
\ref{6.lemaconverg}. By reviewing the proof of
Theorem \ref{5.teortipping}, we observe that
$\ma_{\tilde h/k}(\tilde t_2)\le\mr_{\tilde h/k}(\tilde t_2)-\rho$ if $n$
is large enough, and
we obtain the contradiction by taking limits as $k\to\infty$.
Finally, we deduce the existence of an $h_*>0$ such that \eqref{6.ectrozos}$_h$
has no bounded solutions
for all $h\in[0,h_*]$ from the inequality $\ma_0(\tilde t_2)<\mr_0(\tilde t_2)$
combined with Lemma \ref{6.lemaconverg} (applied to $\ma_0$ on $(-\infty, \tilde t_2]$
and to $\mr_0$ on $[\tilde t_2,\infty)$) and with Remark \ref{3.notabasta}.
\end{proof}
The applicability of these criteria is shown in Figure \ref{6.figSec62} for a suitable example.

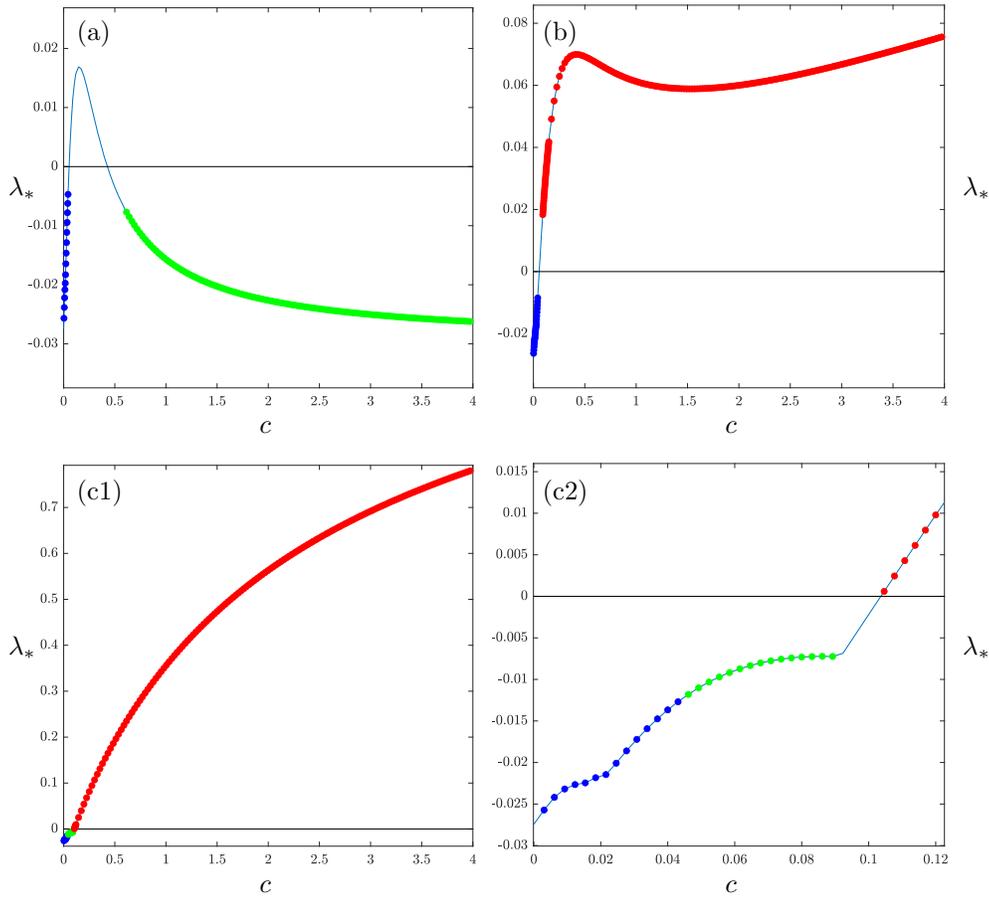
\begin{figure}
    \centering
\begin{overpic}[trim={0.8cm 0.6cm 1cm 0.8cm},clip,width=0.48\textwidth]{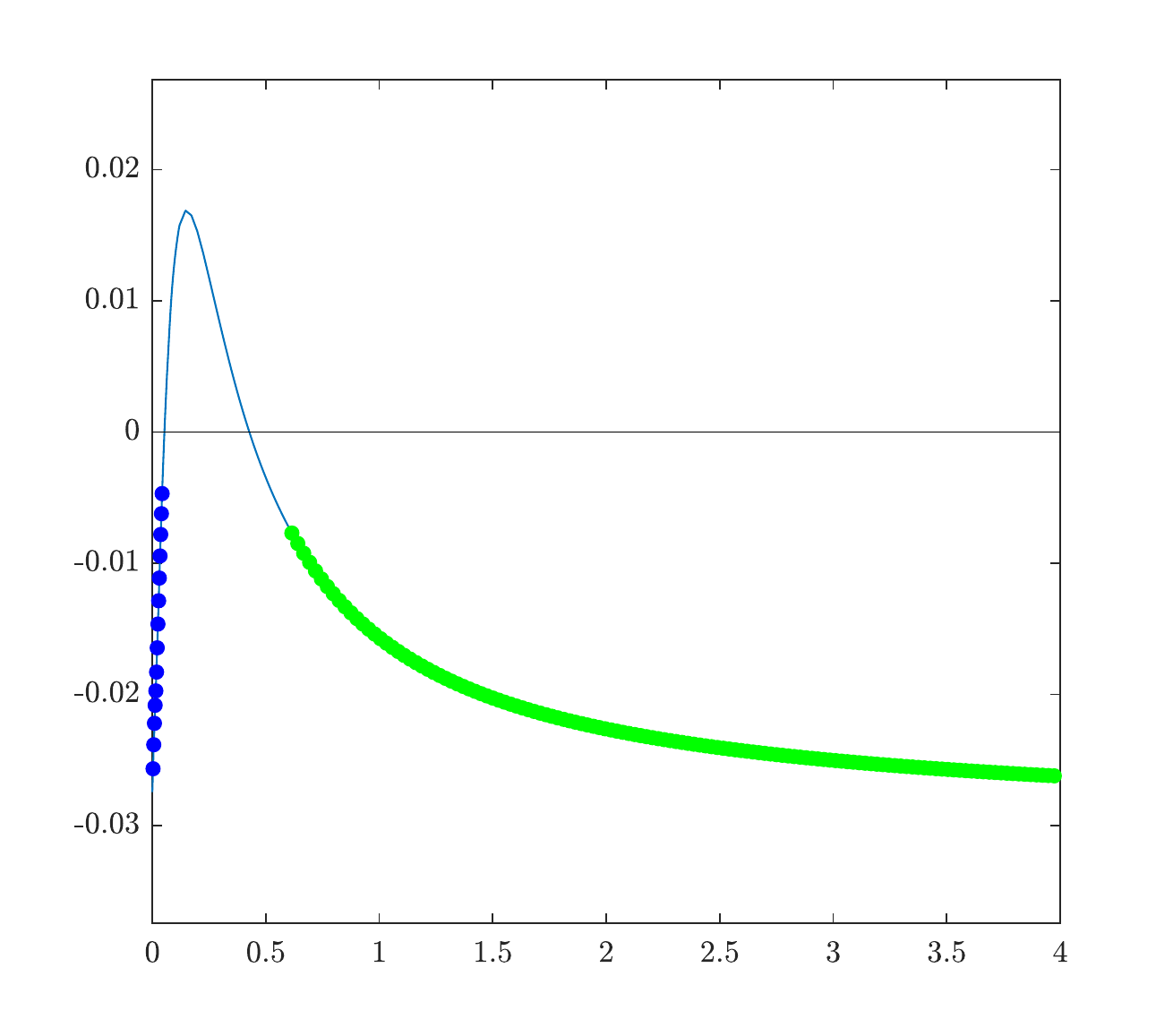}
\put(11,82){(a)}
\put(51,-4){$c$}	
\put(-4,48){$\lambda_*$}
\end{overpic}
\begin{overpic}[trim={0.7cm 0.6cm 1.15cm 0.8cm},clip,width=0.48\textwidth]{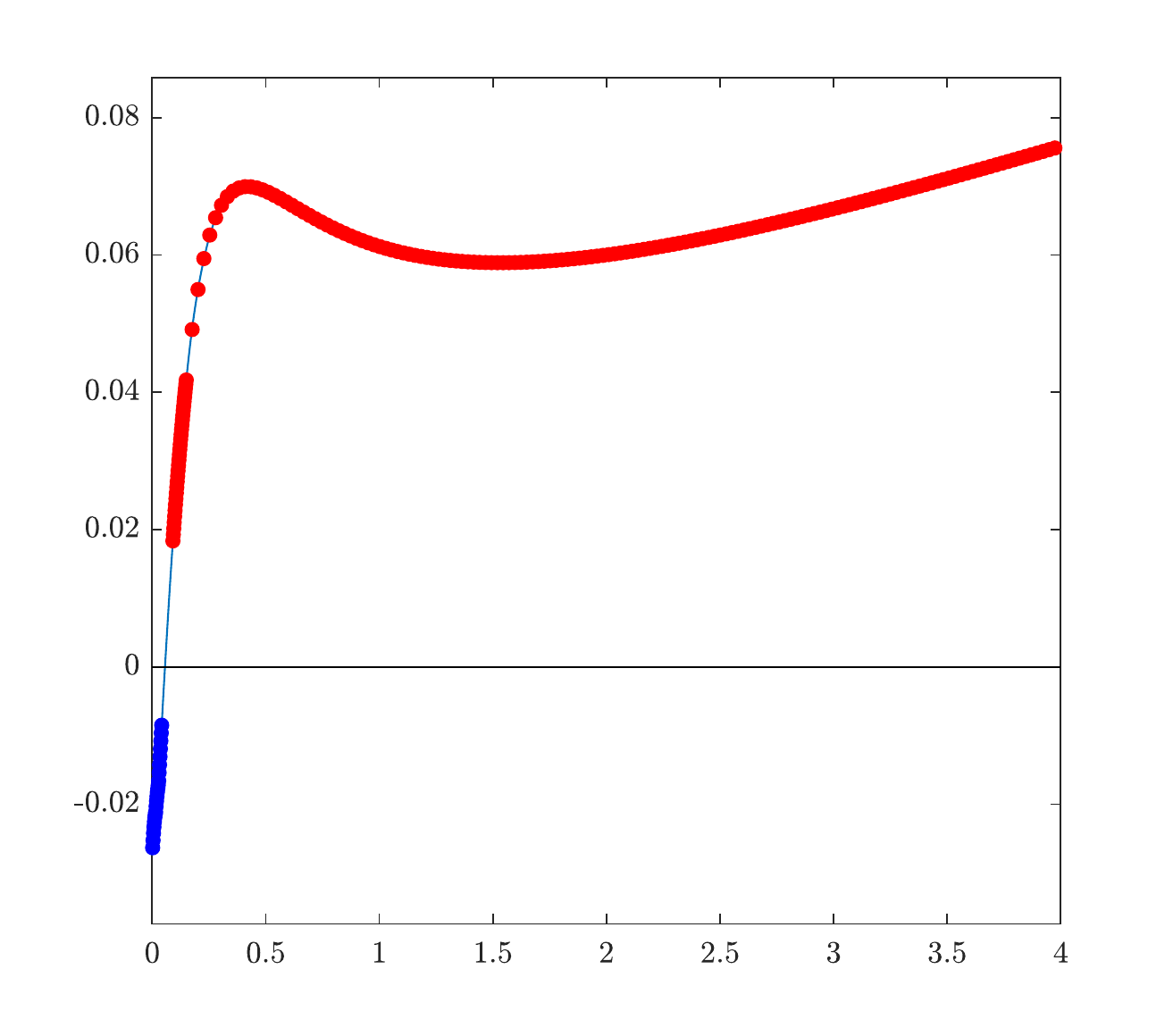}
\put(11,82){(b)}
\put(103,48){$\lambda_*$}
\put(51,-4){$c$}
\end{overpic}\\[4ex]
\begin{overpic}[trim={0.8cm 0.6cm 1cm 0.8cm},clip,width=0.48\textwidth]{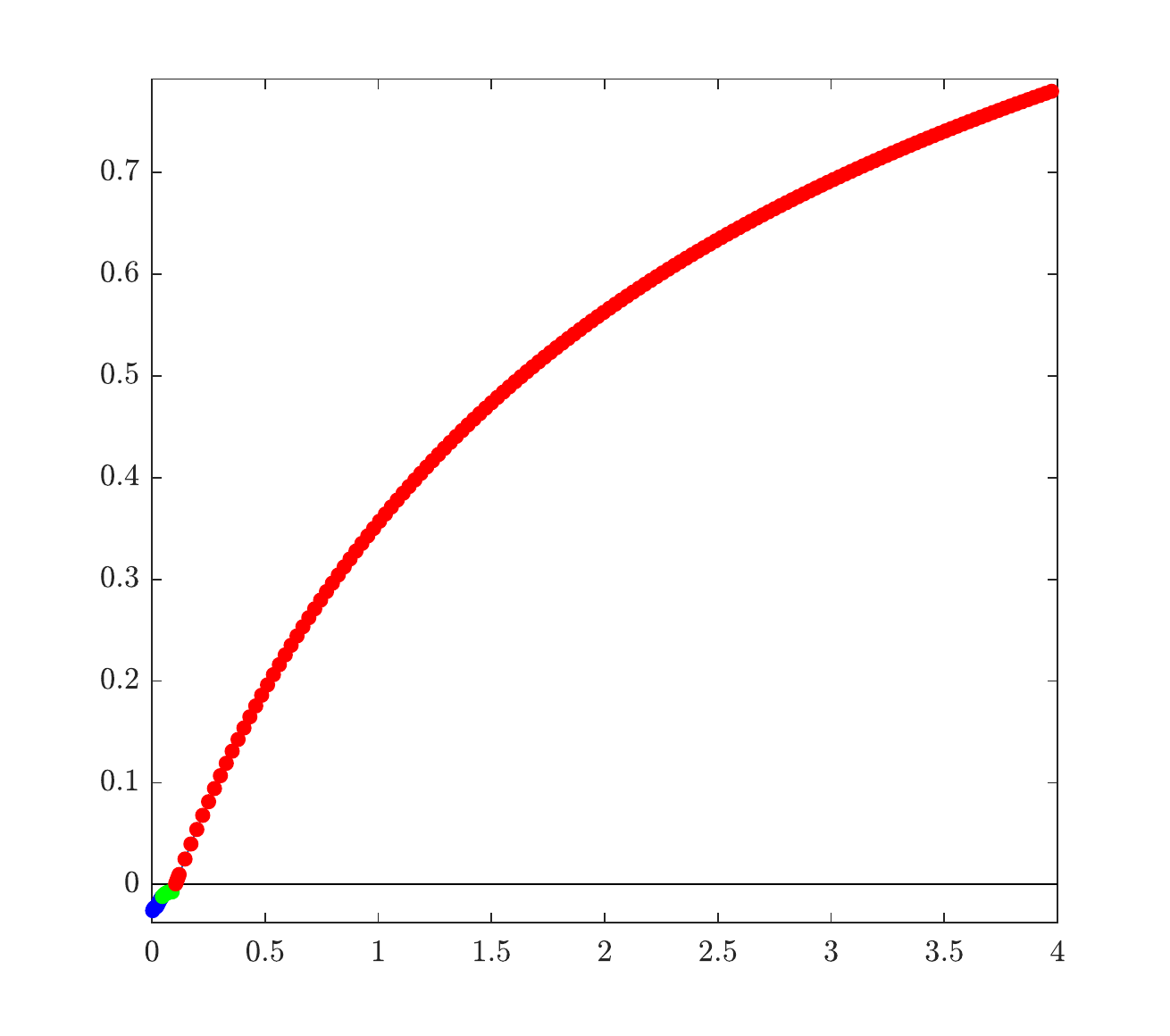}
\put(11,82){(c1)}
\put(-4,48){$\lambda_*$}
\put(51,-4){$c$}
\end{overpic}
\begin{overpic}[trim={0.7cm 0.6cm 1.15cm 0.8cm},clip,width=0.48\textwidth]{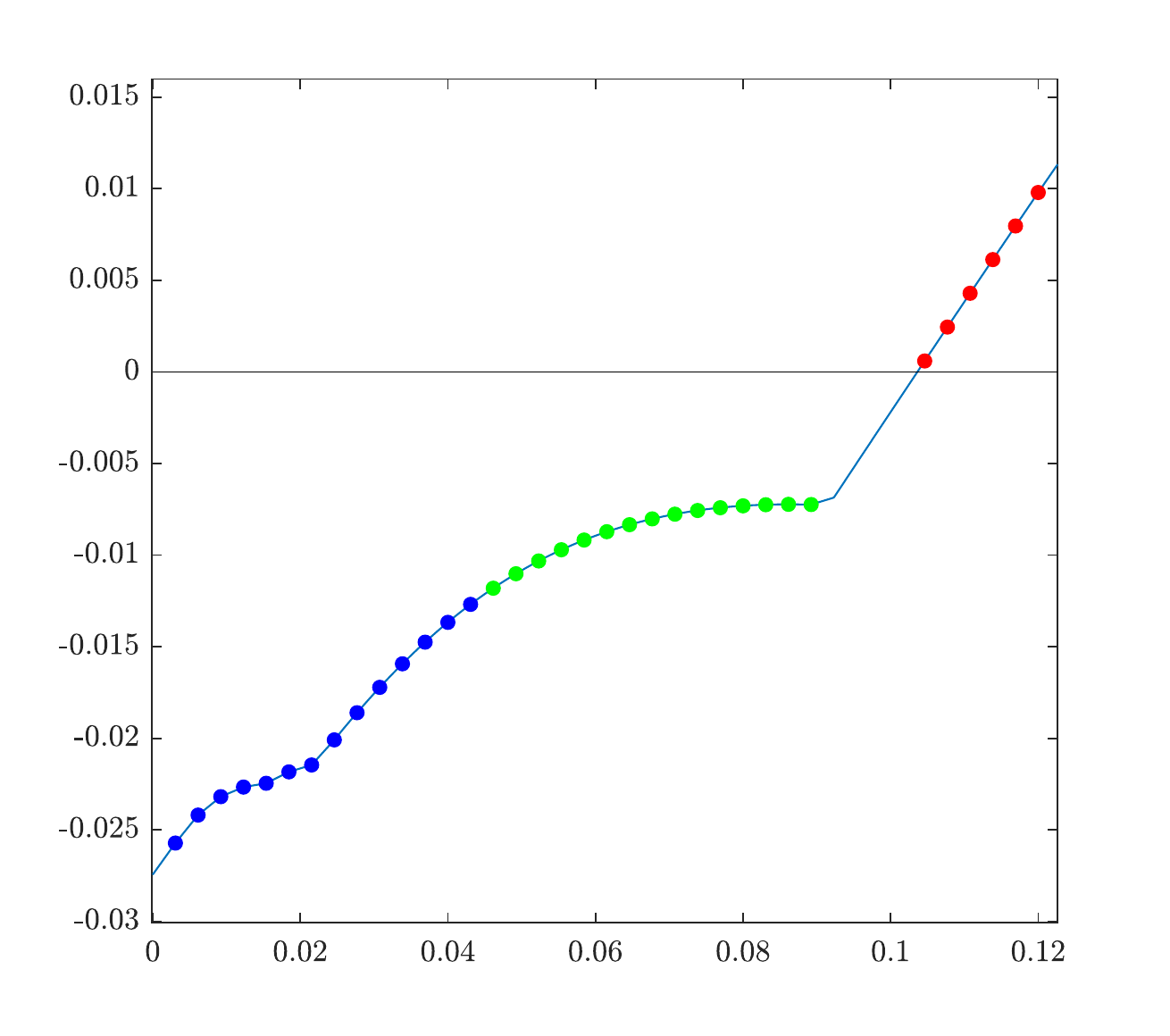}
\put(11,82){(c2)}
\put(103,48){$\lambda_*$}
\put(51,-4){$c$}
\end{overpic}\\[2ex]
\caption{The panels above show the applicability of the rigorous criteria of tracking and tipping developed in Theorems \ref{6.teortrack1}, \ref{6.teortrack2} and \ref{6.teortip1}. In all the panels, the light-blue continuous curve is the graph of the bifurcation curve $\lambda_*(c)$
of $x'=-(x-\G(ct))^2+p_s(t)$ for $c\in[0,4]$, where $\G(t):=(2/\pi)\arctan(t)$ and $p_s(t)=0.83-\sin((t+s)/2)-\sin(\sqrt{5}(t+s))$ (see also Subsection \ref{5.subsecnum}). Particularly, $s=-15$ in panel (a), $s=0$ in panel (b), and $s=15$ in panels (c1) and (c2). Panel (c2) is a magnification of panel (c1) for $c\in[0,0.12]$. The black horizontal lines correspond to the value $\lambda_*=0$. Note that the scale on the y-axis changes in each panel. In each panel, dark blue points correspond to values of $c$ for which the criterion in Theorem \ref{6.teortrack1} guarantees tracking. Green points correspond to values of $c$ for which the criterion in Theorem \ref{6.teortrack2} guarantees tracking even though the criterion in Theorem \ref{6.teortrack1} is not verified. And red points correspond to values of $c$ for which the criterion in Theorem \ref{6.teortip1} guarantees tipping. Given the high slope of the bifurcation map for values close to zero, a denser grid of points has been considered in $[0,0.15]$ (80 points), whereas only 150 evenly spaced points have been analyzed for $c\in (0.15,4]$.
}
\label{6.figSec62}
\end{figure}

\subsection{Sharp estimates for asymptotically constant polygonal transition functions}
In this section, the transition function $\G$ will be a simple continuous piecewise-linear map:
we define
\begin{equation}\label{6.defpol}
\G(t):=\begin{cases}
-1&\text{if }t<-1\,,\\
\;\;\,t&\text{if }t\in[-1,1]\,,\\
\;\;\,1&\text{if }t>1\,.
\end{cases}
\end{equation}
So, if $c>0$ and $d>0$, $\G_{c,d}(t):=d\,\G(ct)$ is the polygonal map
going from the constant
value $-d$ for $t\le 1/c$ to $d$ for $t\ge 1/c$ at speed $cd$ on $[-1/c,1/c]$.
In what follows, we will describe conditions determining the dynamical
situations of the equations of the two-parametric family
\begin{equation}\label{6.ecucd}
 x'=-(x-\G_{c,d}(t))^2+p(t)\,,
\end{equation}
where $p\colon\R\to\R$ is a fixed $L^\infty$ function, $c$ and $d$ are positive constants,
and always under Hypothesis \ref{6.hipo}.
\begin{lema}\label{6.lemamamr}
Assume Hypothesis {\rm\ref{6.hipo}} and fix $c>0$ and $d\ge 0$. Then, there exist
the functions $\ma_{c,d}$ and $\mr_{c,d}$ associated to \eqref{6.ecucd}$_{c,d}$
by Theorem {\rm\ref{3.teoruno}}. In addition,
$\ma_{c,d}(t)=\wit a(t)-d$ for
$t\le -1/c$ and $\mr_{c,d}(t)=\wit r(t)+d$ for $t\ge 1/c$. In particular,
if $\nu\in(0,1)$, then $\ma_{c,d}(-1/c)-\G_{c,d}(-1/c)>b^\nu(-1/c)$
and $\mr_{c,d}(1/c)-\G_{c,d}(1/c)<b^\nu(1/c)$.
\end{lema}
\begin{proof}
Theorem \ref{4.teorexist} shows the first assertion. Note also that
the equation \eqref{6.ecucd}$_{c,d}$ coincides with $x'=-(x+d)^2+p(t)$
(with attractor-repeller pair $(\wit a-d,\wit r-d)$)
on $(-\infty,-1/c]$ and with $x'=-(x-d)^2+p(t)$
(with attractor-repeller pair $(\wit a+d,\wit r+d)$)
on $[1/c,\infty)$. Therefore, the solution $x_{c,d}(t,-1/c,\wit a(-1/c)-d)$
coincides with $\wit a(t)-d$ for $t\le -1/c$ and hence its graph on $(-\infty,-1/c]$
bounds from above the set of the solutions which are bounded as time decreases,
which is the definition of $\ma_{c,d}$; and, analogously, $\mr_{c,d}(t)=\wit r(t)+d$
for $t\in[1/c,\infty)$. The third assertion is an almost immediate consequence of these
equalities, since $\G_{c,d}(\pm 1/c)=\pm d$.
\end{proof}
\begin{teor}\label{6.teortrack3}
Assume Hypothesis {\rm\ref{6.hipo}} and fix $c>0$ and $d\ge 0$. Then,
\eqref{6.ecucd}$_{c,d}$ has an
attractor-repeller pair if one of the following
inequalities holds for all $t\in[-1/c,1/c]$:
\begin{itemize}
\item[\rm(i)] $2d<(1/2c)(\wit a(t)-\wit r(t))^2$\,,
\item[\rm(ii)] $2d<(\wit a(t)-\wit r(t))+
((1-c^2t^2)/2c)(\wit a(t)-\wit r(t))^2$,
\item[\rm(iii)] $2d<\nu(\wit a(t)-\wit r(t))+((1-\nu)/(2c))
(\wit a(t)-\wit r(t))^2$ for a value $\nu\in[0,1]$.
\end{itemize}
In addition,
\begin{itemize}
\item[\rm(iv)] if $2d<\max(\inf_{t\in\R}\big((\wit a(t)-\wit r(t)),(1/2c)\inf_{t\in\R}
(\wit a(t)-\wit r(t))^2\big)$, then there
is total tracking on the hull $\W_p$ of $p$ for \eqref{6.ecucd}$_{c,d}$.
\end{itemize}
\end{teor}
\begin{proof}
According to Lemma \ref{6.lemamamr}, we can take $t_*:=1/c$ to apply
Theorem \ref{6.teortrack1}, which proves the assertion if (i) holds.
\par
Under the assumption of (ii), there exists $\alpha\in(0,1/2)$ such that
\begin{equation}\label{6.inteor}
 \frac{(1/c-2\alpha t)(2\alpha t+1/c)}{4/c^2}\,(\wit a(t)-\wit r(t))^2+
 \frac{\alpha}{1/c}\,(\wit a(t)-\wit r(t))>cd=\G'_{c,d}(t)
\end{equation}
for all $t\in[-1/c,1/c]$. According to Lemma \ref{6.lemamamr},
we can take $t_*:=1/c$ to apply Theorem \ref{6.teortrack2}, which
proves the result.
\par
Let us check (iii). For $\nu=0$, the assertion is that of (i), and
it follows from (ii) if $\nu=1$. So that we assume that the condition
in (iii) holds for $\nu\in(0,1)$, and take $\alpha\in(\nu/2,\sqrt{\nu}/2)$.
Then, $\nu<2\alpha$ and $1-\nu<1-4\alpha^2 c^2 t^2$ for all
$t\in[-1/c,1/c]$, which ensures that
$2d<2\alpha (\wit a(t)-\wit r(t))+((1-4\alpha^2 c^2 t^2)/(2c))
(\wit a(t)-\wit r(t))^2$. In turn, this last inequality ensures
\eqref{6.inteor}, and hence we can reason as in (ii).
\par
To prove (iv), we reason as in the proof of Theorem \ref{6.teortrack1}(ii),
using (i) and (ii).
\end{proof}
\begin{prop}\label{6.proptip2}
Assume Hypothesis {\rm\ref{6.hipo}} and fix $c>0$ and $d\ge 0$. Then,
\eqref{6.ecucd}$_{c,d}$ has no bounded solutions
if $2d>(\wit a(t)-\wit r(t))+((1-c^2t^2)/2c)(\wit a(t)-\wit r(t))^2$ for all
all $t\in[-1/c,1/c]$.
\end{prop}
\begin{proof}
The assumed inequality ensures that, for all $t\in[-1/c,1/c]$,
\[
 \frac{(1/c-t)(t+1/c)}{4/c^2}\,(\wit a(t)-\wit r(t))^2+
 \frac{1}{1/2c}\,(\wit a(t)-\wit r(t))<cd=\G'_{c,d}(t)\,,
\]
and hence the result is an immediate consequence of Theorem \ref{6.teortip1}.
\end{proof}
\begin{figure}
    \centering
\begin{overpic}[trim={0.8cm 0.6cm 1cm 0.6cm},clip,width=0.47\textwidth]{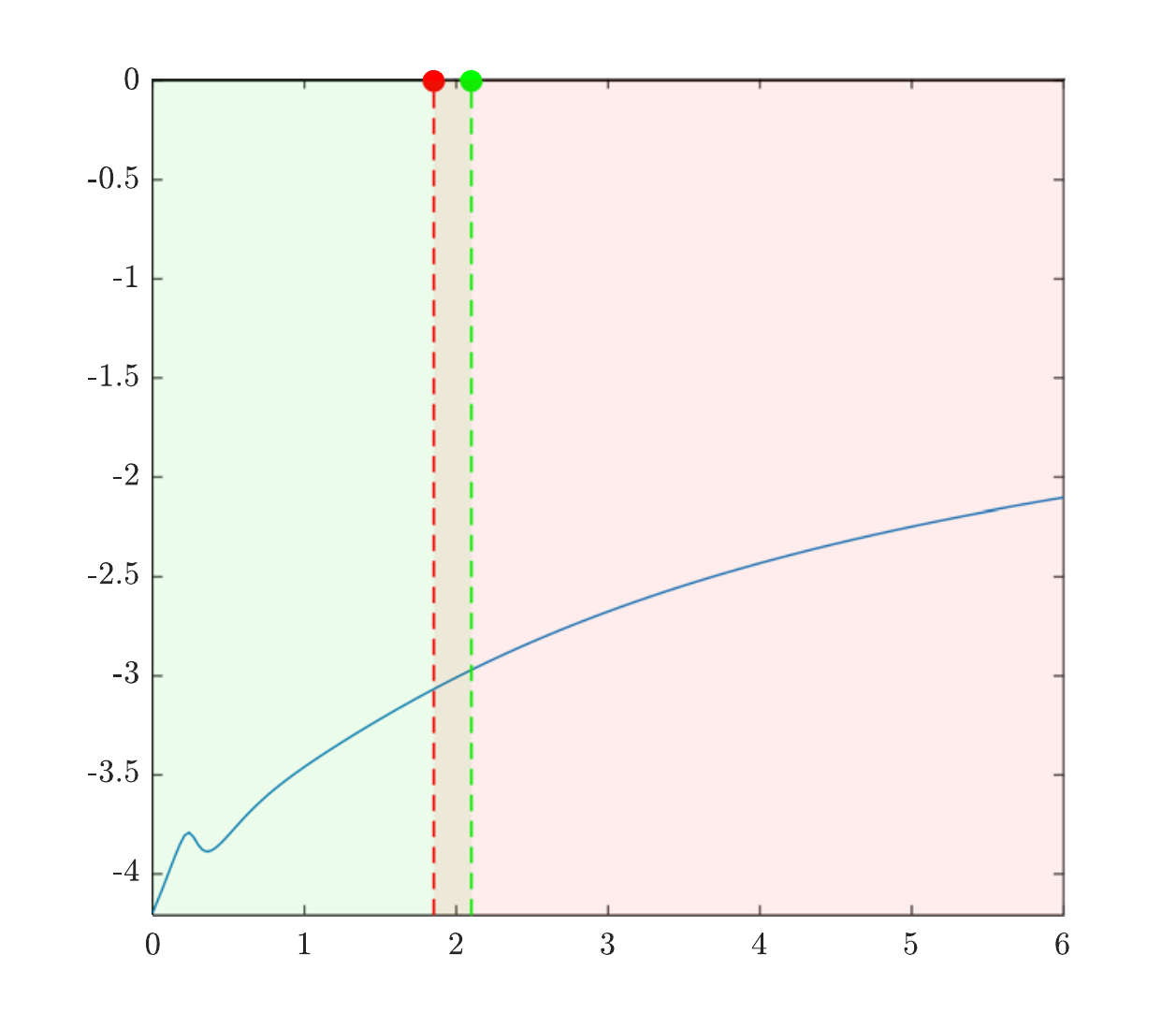}
\put(51,-3){$c$}	
\put(-4,48){$\lambda_*$}
\put(30.5,70){\color{green}$\longleftarrow$}	
\put(30.5,50){\color{green}$\longleftarrow$}
\put(30.5,30){\color{green}$\longleftarrow$}
\put(30.5,10){\color{green}$\longleftarrow$}
\put(35,80){\color{red}$\longrightarrow$}
\put(35,60){\color{red}$\longrightarrow$}
\put(35,40){\color{red}$\longrightarrow$}
\put(35,20){\color{red}$\longrightarrow$}
\end{overpic}
\begin{overpic}[trim={0.8cm 0.6cm 1cm 0.6cm},clip,width=0.47\textwidth]{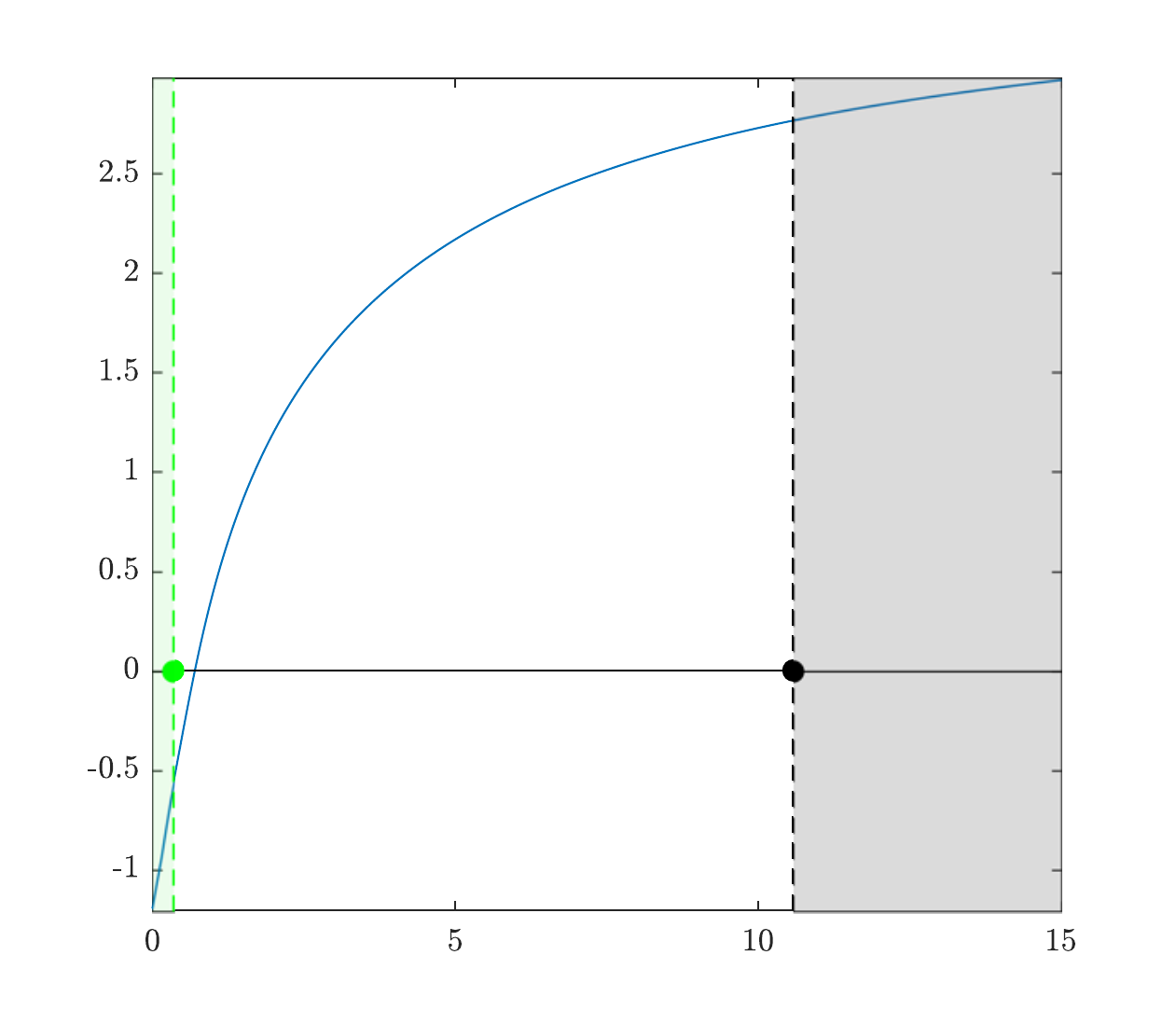}
\put(51,-3){$c$}	
\put(71,75){$\longrightarrow$}
\put(71,55){$\longrightarrow$}
\put(71,35){$\longrightarrow$}
\put(71,15){$\longrightarrow$}
\end{overpic}
\caption{The panels above exemplify the applicability of Theorem \ref{6.teortrack3}(i) and Corollary \ref{6.corofin}(i)\&(iii) on the tipping analysis of equation \eqref{6.ecucd},
with $\G_{c,1}(t)=\G(ct)$ for $\G$ defined by \eqref{6.defpol}. We plot in blue, on the left, the map $c\mapsto\lb_*(\G_{c,1},p)$, for $p(t):=0.962-\sin(t/2)-\sin(\sqrt{5}\,t)$; and, on the right, $c\mapsto\lb(\G_{c,1},p_{15})$, with $p_{15}(t)=p(t+15)$ (see also Subsection \ref{5.subsecnum}).
In both panels, the green dot and dashed green vertical lines identify the highest value of rate $c$ for which Theorem \ref{6.teortrack3}(i) is satisfied. So, the smaller values of $c$ (green area) guarantee tracking. In the left panel, the red dot and dashed red vertical line identify the value $c$ provided by Corollary \ref{6.corofin}(i). Higher values of $c$ (red area) also guarantee tracking in accordance with the result. Notably, Theorem \ref{6.teortrack3}(i) and Corollary \ref{6.corofin}(i) together show that tracking holds for every $c>0$ for this choice of $p(t)$. In the right panel, the black dot and dashed black vertical line identify the value $c$ provided by Corollary \ref{6.corofin}(iii). Higher values of $c$ (grey area) also guarantee tipping ({\sc Cases B} or {\sc C}) in accordance with the result.
}
\label{6.figSec63}
\end{figure}

\begin{coro}\label{6.corofin}
Assume Hypothesis {\rm\ref{6.hipo}}.
\begin{itemize}
\item[\rm(i)] If there exists $t_d>0$ such that $2d<\wit a(t)-\wit r(t)$
for all $t\in[-t_d,t_d]$, then \eqref{6.ecucd}$_{c,d}$ has an
attractor-repeller pair for all $c\ge 1/t_d$.
\item[\rm(ii)] If there exist $\nu\in[0,1)$, $\mu>0$ and $t_d\ge 2\mu/(1-\nu)$ such that
$2d<\nu(\wit a(t)-\wit r(t))+\mu(\wit a(t)-\wit r(t))^2$
for all $t\in[-t_d,t_d]$, then \eqref{6.ecucd}$_{c,d}$ has an
attractor-repeller pair for all $c\in [1/t_d,(1-\nu)/2\mu]$.
\item[\rm(iii)] If there exist $\mu>0$ and $t_d>0$ such that
$2d>(\wit a(t)-\wit r(t))+\mu(\wit a(t)-\wit r(t))^2$
for all $t\in[-t_d,t_d]$, then \eqref{6.ecucd}$_{c,d}$ has no
bounded solutions for all $c\ge\max(1/t_d,1/(2\mu))$.
\end{itemize}
\end{coro}
\begin{proof}
Assertion (i) follows from Theorem \ref{6.teortrack3}(ii), since $1/c\le t_d$.
Assertion (ii) is a consequence of Theorem \ref{6.teortrack3}(iii), since
$1/c\le t_d$ and $\mu\le(1-\nu)/(2c)$. Assertion (iii) follows from
Proposition \ref{6.proptip2}, since $1/c\le t_d$ and $\mu\ge 1/(2c)\ge
(1-c^2t^2)/(2c)$ for all $t\in[-1/c,1/c]$.
\end{proof}
Figure \ref{6.figSec63} shows an application of
these criteria to determine the dynamical case
for some examples of equation \eqref{6.ecucd}.
\par
Let us finally consider the one-parametric family of equations
\begin{equation}\label{6.ecudinfty}
 x'=-(x-d\,\G_\infty(t))^2+p(t)\,,\quad \text{where\;\;}
 \G_\infty(t):=\begin{cases}
 -1&\text{if }t<0\,,\\
 \;\;\,0&\text{if }t= 0\,,\\
 \;\;\,1&\text{if }t> 0\,,
 \end{cases}
 \end{equation}
which, for each fixed $d\ge 0$, can be understood as the limiting equation
of \eqref{6.ecucd}$_{c,d}$ as $c\to\infty$. Let us fix $d>0$ and assume that
$2d<\wit a(0)-\wit r(0)$ (resp.~$2d>\wit a(0)-\wit r(0)$).
Reasoning as in~\cite[Theorem 4.5]{lno}, we can prove that \eqref{6.ecudinfty}$_d$ has an
attractor-repeller pair (resp. has no bounded solutions) and that, in this case,
there exists a minimum $c_d$ such that \eqref{6.ecucd} is in the same dynamical situation
for $c>c_d$. Hence, Corollary \ref{6.corofin}(i) (resp.~(iii)) shows a way to determine
an upper-bound for the value of $c_d$.
\vspace{.4cm}\par\noindent
{\bf DATA AVAILABILITY STATEMENT:}
Data sharing not applicable to this article as no datasets were generated or analysed during the current study.
\appendix
\section{Skewproduct setting and proofs of some results}\label{appendix}
Let us define $\mF$ as the quotient space given by $\{f\colon\R^2\to\R\,
\text{\;satisfying \hyperlink{f1}{\bf f1}, \hyperlink{f2}{\bf f2},
\hyperlink{f3}{\bf f3}}\}$ and the equivalence relation
which identifies $f$ and $g$ if, for $l$-a.a.~$t\in\R$,
$f(t,x)=g(t,x)$ for all $x\in\R$. By ``$f\in\mF$" we mean that
$f$ is any element of its equivalence class.
Together with the countable family of seminorms
\[
 n_{[p,q],x}(f)=\Big|\int_{p}^{q}f(t,x)\,dt\,\Big|\quad
 \text{for $p,\,q,\,x\in\Q\:$ with $\;p<q$}\,,
\]
$\mF$ is a locally convex vector space. We denote by $\sigma$ the induced
topology on $\mF$. The set $\mF$ is metrizable, with the distance defined by
\[
 d(f,g)=\sum_{i,j}^\infty\,\frac{1}{2^{i+j}}\;\min\{1,n_{[p_i,q_i],x_j}(f-g)\}\,,
\]
where $(p_i,q_i)$ is a sequence of $\Q^2$ which is dense in $\R^2$
and satisfies $p_i<q_i$ for all $i\in\N$, and $(x_j)$ is a sequence of $\Q$ which is dense in $\R$.
Note that if $f\in\mF$, then its time-translation $f_t$, given by $f_t(s,x):=f(t+s,x)$, belongs to
$\mF$ for all $t\in\R$. Note also that $(f,f_x)\in\mF\times\mF$ if
$f$ satisfies \hyperlink{f1}{\bf f1}-\hyperlink{f5}{\bf f5}. In this case, we define
\[
 \W(f,f_x):=\text{\rm{closure}}_{(\mF\times\mF,\sigma\times\sigma)}\{(f_t,(f_x)_t)\,|\;t\in\R\}\subset\mF\times\mF\,.
\]
Given $(g,G)\in\W(f,f_x)$, we represent by $(x(t,g,x_0),y(t,g,G,x_0,y_0))$
the solution of the {\em triangular\/} system
\begin{equation}\label{eq:triangEQ}\begin{cases}
 x'=g(t,x)\,, & x(0)=x_0\,,\\
 y'=G(t,x)y\,, &y(0)=y_0\,,
\end{cases}
\end{equation}
and by $\mI_{g,x_0}$ the maximal interval of definition of $x(t,g,x_0)$.
Recall that the results about existence, uniqueness and regularity
properties of the maximal solutions for Carath\'{e}odory differential equations
can be found in~\cite[Chapter 2]{cole}.
\begin{lema}\label{A.lema}
Take $f,g\in\mF$ and assume that $f_{t_n}\xrightarrow{\sigma}g$ for a sequence $(t_n)$ in $\R$. Then,
$\lim_{n\to\infty}\int_p^q (f_{t_n}(t,x)-g(t,x))\,dt=0$
for all $p,q,x\in\R$ with $p<q$.
\end{lema}
\begin{proof}
By reviewing the proof of \cite[Proposition 2.26]{thesis:iacopo},
we see that the positive constants $m_j^f$ and $l_j^f$ provided by conditions
\hyperlink{f2}{\bf f2} and \hyperlink{f3}{\bf f3} for $f$ are also $m$-bounds and $l$-bounds
for all the maps $f_{t_n}$ and $g$.
We proceed in two steps. First, we take for $p<q$ in $\Q$, $x\in\R$ with
$|x|<j\in\N$, and $y\in[-j,j]\cap\Q$. Then
\[
 \left|\int_p^q \big(f_{t_n}(t,x)-g(t,x)\big)\,dt\,\right|\le
 \left|\int_p^q \big(f_{t_n}(t,y)-g(t,y)\big)\,dt\,\right|+ 2\,l^f_j(q-p)\,|y-x|\,,
\]
which combined with the convergence $f_{t_n}\xrightarrow{\sigma}g$ and the density of $\Q$ in $\R$ proves the assertion.
Now we take $p<q$ in $\R$ and $\bar p<\bar q$ in $\Q$. Then, if $|x|\le j\in\N$,
\[
 \left|\int_p^q \big(f_{t_n}(t,x)-g(t,x)\big)\,dt\,\right|\le
 \left|\int_{\bar p}^{\bar q}\big(f_{t_n}(t,x)-g(t,x)\big)\,dt\,\right| +2\,m_j^f(|\bar p-p|+|\bar q-q|)\,,
\]
and we use the previously proved property and the density of $\Q$ in $\R$.
\end{proof}
\begin{teor} \label{A.teorcont}
Let $f\colon\R^2\to\R$ satisfy \hyperlink{f1}{\bf f1}-\hyperlink{f5}{\bf f5}. Then,
\begin{itemize}
\item[\rm(i)] $\W(f,f_x)$ is a compact metric space.
\item[\rm(ii)] If $g\in\mF$ and $f_{t_n}\xrightarrow{\sigma}g$ for a sequence $(t_n)$ in $\R$, then $g$ satisfies
\hyperlink{f1}{\bf f1}-\hyperlink{f5}{\bf f5}, $(g,g_x)\in\W(f,f_x)$ and
$G=g_x$ in $\mF$ for all $(g,G)\in \W(f,f_x)$.
\item[\rm(iii)] The set
\[
 \qquad\mU=\bigcup_{(g,g_x)\in\W(f,f_x),\,x_0\in\R}\big\{ (t,g,g_x,x_0,y_0)\mid t\in \mI_{g,x_0},\,y_0\in\R\big\}
\]
is open on $\R\times\W(f,f_x)\times\R\times\R$, and the map
\[
\begin{array}{rccc}
 \qquad \Psi\colon& \mU&\to&\W(f,f_x)\times\R\times\R\\
 &(t,g,g_x, x_0, y_0)&\mapsto& \big(g_t,(g_x)_t, x(t,g,x_0), y(t,g,g_x,x_0,y_0)\big)
\end{array}
\]
defines a local continuous flow on $\W(f,f_x)\times\R\times\R$,
with $y(t,g,g_x,x_0,1)=(\partial/\partial x_0)\,x(t,g,x_0)$.
\item[\rm(iv)] If, in addition, $f$ satisfies \hyperlink{f6}{\bf f6} and \hyperlink{f7}{\bf f7}, and
$f_{t_n}\xrightarrow{\sigma}g\in\mF$ for a sequence $(t_n)$ in $\R$, then also $g$
satisfies \hyperlink{f6}{\bf f6} and \hyperlink{f7}{\bf f7}.
In particular, $\Psi_1(t,g,x_0):=(g_t,x(t,g,x_0))$ defines a strictly concave local flow
on $\W(f)\times\R$, where $\W(f):=\{g\,|\;(g,g_x)\in\W(f,f_x)\}\subset\mF$.
\end{itemize}
\end{teor}
\begin{proof}
(i) As said in the proof on Lemma \ref{A.lema},
the pairs of positive constants $(m_j^f,m_j^{f_x})$ and $(l_j^f,l_j^{f_x})$ provided by conditions
\hyperlink{f2}{\bf f2} and \hyperlink{f3}{\bf f3} for $(f,f_x)$ are $m$-bounds and $l$-bounds
for any $(g,G)\in\W(f,f_x)$. This fact and \cite[Theorem 4.1]{paper:ZA1} ensure
the compactness of $\W(f,f_x)$.
\smallskip\par
(ii) The compactness of $\W(f,f_x)$ ensures that, if $f_{t_n}\xrightarrow{\sigma}g\in\mF$,
then there exists a subsequence $(t_m)$ of $(t_n)$ and $G\in\mF$ such that
$(f_x)_{t_m}\xrightarrow{\sigma}G$. Therefore, to prove (ii), we must check that,
for $l$-a.a.~$t\in\R$, there exists $g_x(t,x)=G(t,x)$ for all $x\in\R$.
\par
We fix $\bar x\in\Q$. It is not hard to check that the four functions
\[
\begin{split}
 F_{\bar x,1}(t,y)&:=f(t,y)-f(t,\bar x)\,,\quad F_{\bar x,2}(t,y):=(y-\bar x)\int _0^1f_x\big(t,sy+(1-s)\bar x\big)\, ds\,,\\
 G_{\bar x,1}(t,y)&:=g(t,y)-g(t,\bar x)\,,\quad G_{\bar x,2}(t,y):=(y-\bar x)\int _0^1 G\big(t,sy+(1-s)\bar x\big)\, ds\\
\end{split}
\]
belong to $\mF$, being their $m$-bounds and $l$-bounds determined by $|\bar x|$ and those of $f$ and $f_x$.
It follows easily from Lemma \ref{A.lema} that $(F_{\bar x,1})_{t_m}\xrightarrow{\sigma}G_{\bar x,1}$.
Let us check that $(F_{\bar x,2})_{t_m}\xrightarrow{\sigma}G_{\bar x,2}$.
We fix $p,q$ in $\Q$ with $p<q$ and $y\in[-j,j]$ for a $j\in\N$. By Fubini's Theorem,
\[
\begin{split}
&I_{\bar x,m}(p,q,y):=\left|\int_p^q(F_{\bar x,2})_{t_m}(t,y)\,dt-G_{\bar x,2}(t,y)\,dt\,\right|\\
&\qquad=|y-\bar x|\left|\int_0^1\int_p^q \big(f_x\big(t+t_m,sy+(1-s)\bar x\big)-G\big(t,sy+(1-s)\bar x\big)\big)\,dt\,ds\,\right| \\
&\qquad\le |y-\bar x|\int_0^1\left|\int_p^q\big(f_x\big(t+t_m,sy+(1-s)\bar x\big)-G\big(t,sy+(1-s)\bar x\big)\big)\,dt\,\right|ds.
\end{split}
\]
Lemma \ref{A.lema} shows that the inner integral converges to zero as $n\to\infty$
for every fixed $s\in[0,1]$.
Moreover, if $k\ge |\bar x|+j$ belongs to $\N$, then
\[
 \left|\int_p^q\big(f_x\big(t+t_m,sy+(1-s)\bar x\big)-G\big(t,sy+(1-s)\bar x\big)\big)\,dt\,\right|\le\,2\,(q-p)\,m^{f_x}_k.
\]
Lebesgue's Dominated Convergence Theorem implies that $I_{\bar x,m}(p,q,y)\to0$ as $m\to\infty$, which proves our assertion.
\par
Since $(f_x)_{t_m}=(f_{t_m})_x$, Barrow's Rule yields $(F_{\bar x,1})_{t_m}=(F_{\bar x,2})_{t_m}$ for all $n\in\N$.
Taking limit as $n\to\infty$ in $(\mF,\sigma)$, we deduce that, for $t$ in a full Lebesgue measure subset
$\R_{\bar x}\subseteq\R$ (i.e., with $l(\R-\R_{\bar x})=0$), $G_{\bar x,1}(t,y)=G_{\bar x,2}(t,y)$ for all $y\in\R$.
Since the map $(x,y)\mapsto G_{x,1}(t,y)-G_{x,2}(t,y)$ is continuous at $l$-a.a.~$t\in\R$ and
$\bigcap_{\bar x\in\Q}\R_{\bar x}$ has full measure, there exists $\R_0\subset\R$
with full measure such that, if $t\in\R_0$, then
$G_{x,1}(t,y)=G_{x,2}(t,y)$ for all $x,y\in\R$.
Taking $(t,x)\in\R_0\times\R$ and $y=x+h$ yields
\[
 \frac{1}{h}\:(g(t,x+h)-g(t,x))=\int _0^1G\big(t,x+sh\big)\, ds\,,
\]
and hence, taking limit as $h\to 0$, we get $g_x(t,x)=G(t,x)$. This proves (ii).
\smallskip\par
(iii) The initial properties follow by \cite[Theorem 3.9(ii)]{paper:LNO2}, since condition \hyperlink{f3}{\bf f3}
for $f$ and $f_x$ ensure that our topology $\sigma$ coincides with the topology there used
(see \cite[Theorem 2.20(ii)]{paper:LNO3}). The last one is standard:
see e.g.~\cite[Theorem 2.3.1]{brpi}.
\smallskip\par
(iv) Let us check that, if the sequences $(f_n)_{n\in\N}$ and $(h_n)_{n\in\N}$ of $\mF$ converge
(with respect to $\sigma$) to the functions $f$ and $h$ of $\mF$ and,
for $l$-a.a.~$t\in\R$, $f_n(t,x)< h_n(t,x)$ for all $x\in\R$ and any $n\in\N$,
then, for $l$-a.a.~$t\in\R$, $f(t,x)\le h(t,x)$ for all $x\in\R$.
\par
Since $\int_p^q f_n(s,x)\,ds<\int_p^q h_n(s,x)\,ds$ for $p,q,x\in\Q$ with $p<q$, taking limit
as $n\to\infty$ yields $\int_p^q f(s,x)\,ds\le \int_p^q h(s,x)\,ds$. As in the proof on Lemma
\ref{A.lema}, we use the existence of constant (although not common) $m$-bounds and $l$-bounds
for $f$ and $h$ to check that the last inequality holds for any $p,q\in\R$ and $x\in\Q$.
This ensures that,
for any fixed $\bar x\in\Q$, $(1/l)\int_t^{t+l}(h(s,\bar x)-f(s,\bar x))\,ds\ge 0$ for all $t,l\in\R$,
and Lebesgue's Differentiation Theorem ensures that $h(t,\bar x)-f(t,\bar x)\ge 0$
for $l$-a.a.~$t\in\R$. The density of $\Q$ and the strong Carath\'{e}odory character of the map
$h-f$ provide a subset $\R_1\subseteq\R$ with full Lebesgue measure such that, for all $t\in\R_1$,
$h(t,x)-f(t,x)\ge 0$ for all $x\in\R$.
This proves the claim.
\par
Assume now that $f$ satisfies \hyperlink{f6}{\bf f6} and that
$f_{t_n}\xrightarrow{\sigma}g\in\mF$ for a sequence $(t_n)$ in $\R$.
Then, for every $n\in\N$, there exists
$r_n>0$ such that, if $|x|>r_n$, then $f(t,x)<(-\delta+1/n)x^2$ for all $t\in\mR$.
The previously proved property yields $g(t,x)\le -\delta x^2$ if $|x|>r_n$ and $t\in\mR$,
which in turn ensures that $g$ satisfies \hyperlink{f6}{\bf f6}.
\par
Finally, if $f$ satisfies \hyperlink{f7}{\bf f7} and
$f_{t_n}\xrightarrow{\sigma}g\in\mF$, we have
$(f_x)_{t_m}\xrightarrow{\sigma}g_x\in\mF$ for a suitable subsequence
$(t_m)$, by (ii). The same argument as above ensures that $g_x$ inherits the
property required to $f_x$.
\par
The last assertion is standard. See e.g.~\cite[Proposition 2.3]{nuos5}.
\end{proof}
Theorem~\ref{A.teorcont} provides the suitable dynamical framework to
to prove Theorems \ref{3.teorhyp} and \ref{3.teorlb*}, which we do by
adapting to our present setting the proofs of \cite[Theorem 3.5 and 3.6]{lnor}.
A sketch of these proofs, insisting in those steps which are different,
complete this appendix and the paper.
\smallskip\par\noindent
\noindent{\em Proofs of Theorems {\rm \ref{3.teorhyp}} and {\rm \ref{3.teorlb*}}}.
Since $f$ satisfies \hyperlink{f1}{\bf f1}-\hyperlink{f7}{\bf f7},
Theorem \ref{A.teorcont} ensures that
$\W(f):=\{g\,|\;(g,g_x)\in\W(f,f_x)\}$ is compact and the map
$\Psi_1(t,g,x_0):=(g_t,x(t,g,x_0))$ defines
a local flow on $\W(f)\times\R$ which is continuous, $C^1$ with respect to the state variable,
strictly concave, and given by a family of equations with coercive
coefficients. The arguments used in the (very long and technical)
proof of~\cite[Theorem 3.5]{lnor} can be adapted to our current setting in order
to prove (a)$\Rightarrow$(b), as well as points (i) and (ii) of Theorem \ref{3.teorhyp}.
One of these arguments applies a ``first approximation theorem"~to
a scalar equation of the type $z'=f_x(t,\wit b(t))\,z+h(t,z)$, where
$\wit b(t)$ is a hyperbolic solution of $x'=f(t,x)$ and
$h(t,z):=f(t,z+\wit b(t))-f(t,\wit b(t))-f_x(t,\wit b(t))\,z$.
Since $h(t,0)=0$, the existence of $L^\infty$ $l$-bounds for
$f_x$ ensures that, given $\ep>0$, there exists $\rho_\ep>0$
such that $|h(t,z)|\le \ep|z|$ $l$-a.a.~if $|z|\le\rho_\ep$.
This condition suffices to repeat the proof of
Theorem III.2.4 in~\cite{hale}, which is the required first
approximation result.
\par
The assertion (b)$\Rightarrow$(c) of Theorem \ref{3.teorhyp} is trivial, and
(c)$\Rightarrow$(a) can be deduced, for instance, of Theorem \ref{3.teorlb*}(iv),
whose proof (below explained) is independent of Theorem \ref{3.teorhyp}.
\par
The proof of Theorem \ref{3.teorlb*} adapts to the current setting that
of~\cite[Theorem 3.6]{lnor}. The arguments to check (i), (ii), (iii)
and (v) are completely analogous. (In particular, the assertion
(a)$\Rightarrow$(b) of Theorem \ref{3.teorhyp}
is used in the proof of (iii).) The unique significative difference is in
point (iv), so that we will describe it in detail.
\par
Let us assume for contradiction that
$\inf_{t\in\R}(a_\lb(t)-r_\lb(t))=0$ for a $\lb>\lb^*$, which we fix.
It follows from point (ii) (of Theorem \ref{3.teorlb*}) that
$r_\lb(t)<x_\lb(t,s,a_{\lb^*}(s))<a_\lb(t)$ for any $s,t\in\R$, and
from a standard comparison result that
$d_s(t):=x_\lb(t,s,a_{\lb^*}(s))-a_{\lb^*}(t)\ge 0$ for any
$s\in\R$ and $t\ge s$. Note that there exists $\kappa_1>0$
such that $|d_s(t)|\le\kappa_1$ for any $s\in\R$ and $t\ge s$.
In addition, for all $s\in\R$ and $l$-a.a.~$t\in\R$,
\[
 d_s'(t)=f_x(t,a_{\lb_*}(t)+\xi_{s,t})\,d_s(t)+\lb-\lb_*
\]
for a point $\xi_{s,t}\in[0,d_s(t)]$. We use condition \hyperlink{f2}{\bf f2}
for $f_x$ (ensured by \hyperlink{f3}{\bf f3} for $f$)
to find a constant $\kappa>0$ such that $|f_x(t,a_{\lb_*}(t)+\xi_{s,t})|
\le\kappa$ for all $s\in\R$ and $l$-a.a.~$t\in\R$.
Then,
\[
 d'_s(t)\ge-\kappa\,d_s(t)+\lb-\lb^*
 \quad\text{for all $s\in\R$ and $l$-a.a.~$t\in\R$}\,,
\]
and hence
\[
 d_s(t)\ge \frac{\lb-\lb^*}{\kappa}\:\big( 1-e^{-\kappa(t-s)}\big)
 \quad \text{for all $s\in\R$ and $t\ge s$}\,.
\]
In particular, there exists $l>0$ such that
$d_s(t)\ge (\lb-\lb^*)/2\kappa=:\wit \kappa$
whenever $s\in\R$ and $t\ge s+l$.
Now we take $t_0\in\R$ such that
$a_\lb(t_0)-r_\lb(t_0)<\wit\kappa$. But then, by (ii),
\[
 \wit\kappa>x_\lb(t_0,t_0-l,a_\lb(t_0-l))-r_\lb(t)>
 x_\lb(t_0,t_0-l,a_{\lb^*}(t_0-l))-a_{\lb^*}(t)=d_{t_0-l}(t_0)\,,
\]
which provides the required contradiction.
\qed

\end{document}